\documentclass[11pt]{article}

\usepackage{graphicx}
\usepackage{amssymb}
\usepackage{amscd}
\usepackage{amsmath}
\usepackage{amsthm}
\usepackage{enumerate}
\usepackage[all]{xy}
\usepackage{comment}
\usepackage{subfigure}
\usepackage{lineno}
\usepackage{tikz}
\usepackage{tikz-cd}
\usepackage{stmaryrd}

\newtheorem{theorem}{Theorem}
\newtheorem{lemma}{Lemma}
\newtheorem{proposition}{Proposition}

\newtheorem{definition}{Definition}

\newtheorem{remark}{Remark}

\newcommand{\suchthat}{\ifnum\currentgrouptype=16 \;\middle|\;\else\mid\fi} 

\def\vector#1{\mbox{\boldmath $#1$}}

\def\clap#1{\hbox to 0pt{\hss#1\hss}}

\makeatletter

\def\dimInd #1{\expandafter\indm@i#1,,,,,\@nil}
\def\indm@i #1,#2,#3,#4,#5,#6,#7\@nil{%
    \ifx$#1$ \edef\paramA{0}\else \edef\paramA{#1} \fi
    \ifx$#2$ \edef\paramB{0}\else \edef\paramB{#2} \fi
    \ifx$#3$ \edef\paramC{0}\else \edef\paramC{#3} \fi
    \ifx$#4$ \edef\paramD{0}\else \edef\paramD{#4} \fi
    \ifx$#5$ \edef\paramE{0}\else \edef\paramE{#5} \fi
    \ifx$#6$ \edef\paramF{0}\else \edef\paramF{#6} \fi
    I({\tiny
        {\arraycolsep = 0.1mm
            \begin{array}{ccc}
                \paramD & \paramE & \paramF\\
                \paramA & \paramB & \paramC
            \end{array}}})
}

\def\dimVec #1{\expandafter\dimVec@i#1,,,,,\@nil}
\def\dimVec@i #1,#2,#3,#4,#5,#6,#7\@nil{%
    \ifx$#1$ \edef\paramA{0}\else \edef\paramA{#1} \fi
    \ifx$#2$ \edef\paramB{0}\else \edef\paramB{#2} \fi
    \ifx$#3$ \edef\paramC{0}\else \edef\paramC{#3} \fi
    \ifx$#4$ \edef\paramD{0}\else \edef\paramD{#4} \fi
    \ifx$#5$ \edef\paramE{0}\else \edef\paramE{#5} \fi
    \ifx$#6$ \edef\paramF{0}\else \edef\paramF{#6} \fi
    {\tiny\arraycolsep = 0.1mm
        \begin{array}{ccc}
            \paramD & \paramE & \paramF\\
            \paramA & \paramB & \paramC
        \end{array}}
}

\def\smdimVec #1{\expandafter\smdimVec@i#1,,,,,\@nil}
\def\smdimVec@i #1,#2,#3,#4,#5,#6,#7\@nil{%
    \ifx$#1$ \edef\paramA{0}\else \edef\paramA{#1} \fi
    \ifx$#2$ \edef\paramB{0}\else \edef\paramB{#2} \fi
    \ifx$#3$ \edef\paramC{0}\else \edef\paramC{#3} \fi
    \ifx$#4$ \edef\paramD{0}\else \edef\paramD{#4} \fi
    \ifx$#5$ \edef\paramE{0}\else \edef\paramE{#5} \fi
    \ifx$#6$ \edef\paramF{0}\else \edef\paramF{#6} \fi
    {\small\arraycolsep = 0.1mm
        \begin{array}{ccc}
            \paramD & \paramE & \paramF\\
            \paramA & \paramB & \paramC
        \end{array}}
}

\def\spaceInd #1{\expandafter\spcindm@i#1,,,,,\@nil}
\def\spcindm@i #1,#2,#3,#4,#5,#6,#7\@nil{%
    \ifx$#1$ \edef\paramA{0}\else \edef\paramA{#1} \fi
    \ifx$#2$ \edef\paramB{0}\else \edef\paramB{#2} \fi
    \ifx$#3$ \edef\paramC{0}\else \edef\paramC{#3} \fi
    \ifx$#4$ \edef\paramD{0}\else \edef\paramD{#4} \fi
    \ifx$#5$ \edef\paramE{0}\else \edef\paramE{#5} \fi
    \ifx$#6$ \edef\paramF{0}\else \edef\paramF{#6} \fi
    \begin{tikzcd}[ampersand replacement=\&]
        \paramD \rar \& \paramE \& \paramF \lar \\
        \paramA \rar\uar \& \paramB \uar \& \paramC \lar\uar
    \end{tikzcd}
}

\makeatother

\def\drawmytikzmatrix{
    \pgfmathtruncatemacro{\sz}{10};
    \pgfmathtruncatemacro{\szpred}{\sz-1};
    \pgfmathtruncatemacro{\baruno}{\sz/2+1};
    
    \filldraw[fill=black!20] (\sz,0)
    \foreach \dy in {1,1,2,1,2,2,1}{
        -- ++(-1,0) -- ++(0,\dy)
    }
    --(\sz,\sz) -- cycle;

    \draw (0,0) rectangle (\sz,\sz);
    \draw (\baruno,0) -- (\baruno,\sz);

    \draw (0,5) -- (\sz,5);

    \node at (3,2.5) {\large$\mathbf{0}$};

}

\newcommand{\A}{\mathbb{A}}
\newcommand{\Aright}{{\mathbb{A}}^{\shortrightarrow}}
\newcommand{\D}{\mathbb{D}}
\newcommand{\E}{\mathbb{E}}

\newcommand{\I}{\mathbb{I}}
\newcommand{\N}{\mathbb{N}}
\newcommand{\R}{\mathbb{R}}

\newcommand{\rank}{{\rm Rank~}}
\newcommand{\trace}{{\rm Tr}}
\newcommand{\kernel}{{\rm Ker~}}
\newcommand{\image}{{\rm Im~}}
\newcommand{\cokernel}{{\rm Coker~}}

\newcommand{\Hom}{{\rm Hom}}

\newcommand{\KQ}{K\!Q}
\newcommand{\CL}{C\!L}
\newcommand{\rad}{{\rm rad}}
\newcommand{\rep}{{\rm rep}}

\newcommand{\myarrowLR}[1]{\tikz{\draw[commutative diagrams/leftrightarrow](0,0)--+(#1,0);}}
\newcommand{\myarrowL}[1]{\tikz{\draw[commutative diagrams/leftarrow](0,0)--+(#1,0);}}
\newcommand{\myarrowR}[1]{\tikz{\draw[commutative diagrams/rightarrow](0,0)--+(#1,0);}}

\newcommand{\sio}{${\rm SiO}_2$ }

\numberwithin{equation}{section}

\begin{document}

\title{Persistence Modules on Commutative Ladders of Finite Type}

\author{
Emerson G. Escolar\thanks{Graduate School of Mathematics, Kyushu University.  
744, Motooka, Nishi-ku, Fukuoka, 819-0395, Japan (e-mail: \texttt{eescolar@math.kyushu-u.ac.jp}).}
\and
Yasuaki Hiraoka\thanks{WPI - Advanced Institute for Materials Research (WPI-AIMR), Tohoku University
2-1-1, Katahira, Aoba-ku, Sendai, 980-8577, Japan (e-mail: \texttt{hiraoka@wpi-aimr.tohoku.ac.jp}).}
}

%\author{
%Emerson G. Escolar         \and
%Yasuaki Hiraoka
%}
%
%%\authorrunning{Short form of author list} % if too long for running head
%
%\institute{E. G. Escolar \at
%              Graduate School of Mathematics, Kyushu University\\ 
%744, Motooka, Nishi-ku, Fukuoka, 819-0395, Japan \\
%              \email{eescolar@math.kyushu-u.ac.jp}           %  \\
%%             \emph{Present address:} of F. Author  %  if needed
%           \and          		        
%           Y. Hiraoka \at
%%	        Institute of Mathematics for Industry, Kyushu University\\
%%			744, Motooka, Nishi-ku, Fukuoka, 819-0395, Japan\\
%%              Tel.: +81-92-802-4402\\
%%              Fax: +81-92-802-4405
%			WPI - Advanced Institute for Materials Research (WPI-AIMR), Tohoku University\\
%			2-1-1, Katahira, Aoba-ku, Sendai, 980-8577, Japan\\
%              \email{hiraoka@wpi-aimr.tohoku.ac.jp}           \\			
%%              Tel.: +81-22-217-6150\\
%%              Fax: +81-22-217-6151              
%}

%\date{Received: date / Accepted: date}
% The correct dates will be entered by the editor

\maketitle

\begin{abstract}
We study persistence modules defined on commutative ladders. This class of persistence modules frequently appears in topological data analysis, and the theory and algorithm proposed in this paper can be applied to these practical problems. A new algebraic framework deals with persistence modules as representations on associative algebras and the Auslander-Reiten theory is applied to develop the theoretical and algorithmic foundations. In particular, we prove that the commutative ladders of length less than 5 are representation-finite and explicitly show their Auslander-Reiten quivers. Furthermore, a generalization of persistence diagrams is introduced by using Auslander-Reiten quivers. We provide an algorithm for computing persistence diagrams for the commutative ladders of length $3$ by using the structure of Auslander-Reiten quivers.
\end{abstract}

\noindent\textbf{Keywords} Persistence Module, Commutative Ladder, Representation Theory, Auslander-Reiten Theory, Topological Data Analysis
% \PACS{PACS code1 \and PACS code2 \and more}
%\subclass{MSC 55N99}

%
%
%from here
%
%

%
%
%
%
%
%
%
%
%
\section{Introduction}\label{sec:introduction}
The theory of topological persistence \cite{elz,zc} has become popular in topological data analysis (TDA) \cite{carlsson}. Given an input data such as a point cloud or a sequence of topological spaces, persistence modules capture topological features of the input and the robustness of those features. One of the important algebraic frameworks that makes the theory of topological persistence more useful is to deal with persistence modules as representations on $\A_n$ quivers \cite{cd}. In this framework, the representation-finite property of $\A_n$ quivers allows us to measure topological features in a compact format called a persistence diagram and to realize efficient computations \cite{nanda,zc} for practical applications.

On the other hand, there are several important problems to be solved in TDA to which persistence modules on $\A_n$ quivers cannot be applied. For example, we often need to compare two input data and capture robust {\it and} common topological features between them. By restricting our attention to only robust {\it or} common topological features, persistence modules on $\A_n$ quivers can solve the restricted problems. However, we cannot directly apply this method for studying both properties simultaneously. 

In this paper, we extend the theory of topological persistence to persistence modules on commutative ladders. This extension provides us with a tool to solve the above problem, as an example. A new algebraic framework we propose in this paper deals with persistence modules as representations on quivers with relations 
%(i.e., associative algebras) 
and applies a powerful tool in representation theory called the Auslander-Reiten theory. By this tool, we can study the representation-finite property of commutative ladders and obtain a natural generalization of persistence diagrams. 

%Furthermore, our algorithm for the Auslander-Reiten theory also provides us with algorithms for computing generalized persistence diagrams. 

%
%
%
\subsection{Motivations}\label{section:motivation}
This work is motivated by recent results in the theory of topological persistence and applications in TDA. 
On the theoretical side, this work is strongly inspired by the paper \cite{cd}. 
In that paper, Carlsson and de Silva explain persistence modules as representations of $\A_n$ quivers and generalize to zigzag persistence modules. 
By this connection to the representation theory, algebraic aspects of persistence modules have been understood further and have led to several generalizations \cite{bubenik,wcb}. 
This paper deals with an extension to representations on associative algebras. 

On the side of applications, our motivation arises from TDA on materials science \cite{glass} and protein structural analysis \cite{protein}. In both studies, atomic arrangements in molecules such as glasses and proteins are the input data, and the authors study topological characterizations of some physical properties of those molecules by using standard persistence modules of the weighted alpha complex filtrations \cite{alpha} of the input. 

Let us explain the TDA studied in \cite{glass} in more detail to clarify the motivation of this work.
The authors study geometric structures of rings in glasses. The left part of Figure \ref{fig:pd_glass} shows the persistence diagram in dimension one of an \sio glass with an atomic arrangement $P_X$, where the multiplicities of the generators are measured in double logarithmic scale. One of the important results in \cite{glass} is that there are correlations between the generators in the blue square and the rigidity of glasses under small deformations of atomic arrangements. After seeing this physical correspondence, it is natural to ask how these generators actually persist under physical pressurization. The right part  of Figure \ref{fig:pd_glass} shows the persistence diagram in dimension one of an atomic arrangement $P_Y$ obtained by adding pressure to the original $P_X$. It should be noted that the similarity of these diagrams, which can be measured by using the bottleneck distance, is not sufficient to verify that these generators actually persist under pressurization. In order to precisely answer the above question, we have to directly study generators in the blue regions which are shared between the left and right.
\begin{figure}[h!]
\begin{center}
\includegraphics[width=7.5cm]{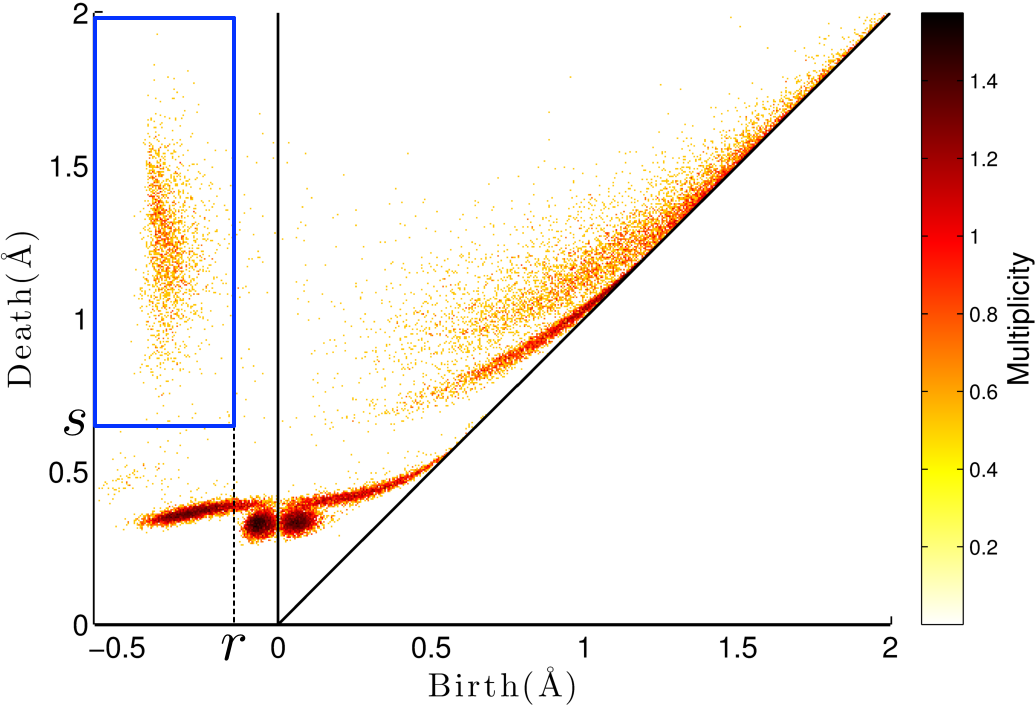}
\includegraphics[width=7.5cm]{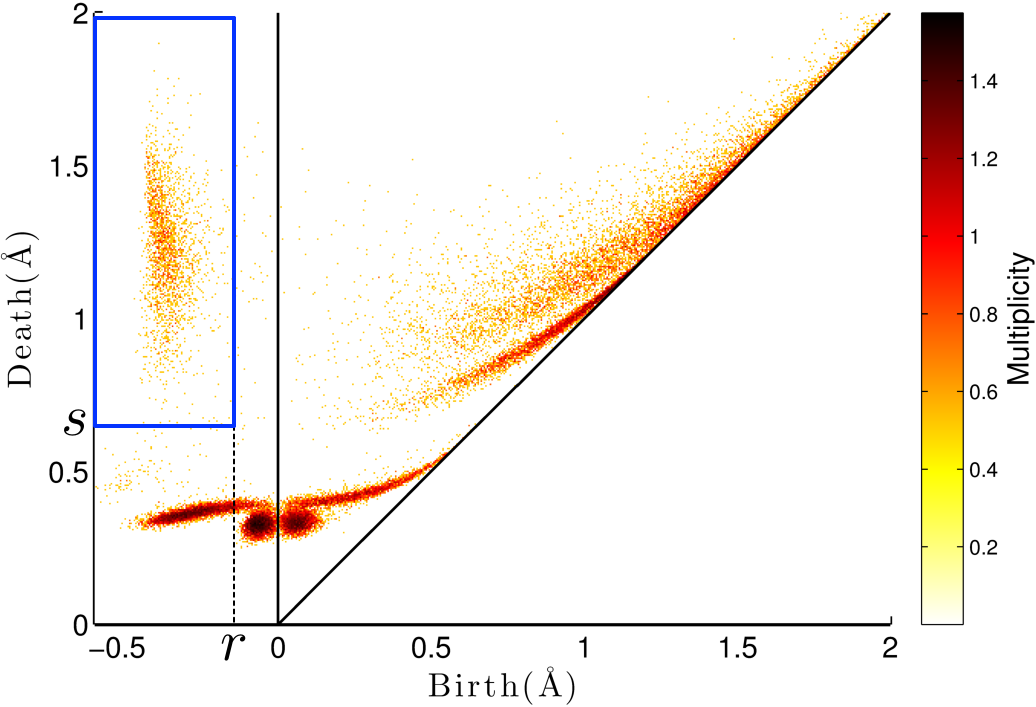}
\caption{Persistence diagrams in dimension one of the weighted alpha complex filtrations of the molecular structures of glasses. The point cloud for the left is given by an atomic arrangement $P_X$ of a glass, whereas the atomic arrangement $P_Y$ for the right is obtained by adding isotropic pressure to $P_X$. The multiplicities are measured in double logarithmic scale. The weighted alpha complexes are constructed by the software CGAL \cite{cgal}. For more details, see \cite{glass}.}
\label{fig:pd_glass}
\end{center}
\end{figure}

Let $X_\alpha$ and $Y_\alpha$ be weighted alpha complexes with parameter $\alpha$ of the atomic arrangements $P_X$ and $P_Y$, respectively, where the initial weight of each atom is given by its ionic radius. We note that the birth and death parameters of the generators in the blue region are less than $r$ and greater than $s$, respectively. In the 2-step persistence modules 
\begin{eqnarray*}
H_1(X_r)\rightarrow H_1(X_s)&\simeq& \I[r,s]^{m_{rs}}\oplus \I[r,r]^{m_r} \oplus \I[s,s]^{m_s},\\
H_1(Y_r)\rightarrow H_1(Y_s)&\simeq& \I[r,s]^{n_{rs}}\oplus \I[r,r]^{n_r} \oplus \I[s,s]^{n_s}
\end{eqnarray*}
with decompositions using the interval representations 
\[
	\I[r,s] = K\rightarrow K,~~~\I[r,r]=K\rightarrow 0,~~~\I[s,s]=0\rightarrow K,
\]
these generators are specified by the direct summands $\I[r,s]^{m_{rs}}$ and $\I[r,s]^{n_{rs}}$. Here, we denote the coefficient field of the homology vector spaces by $K$. Then, in order to study the relationship between $\I[r,s]^{m_{rs}}$ and $\I[r,s]^{n_{rs}}$, we can use the commutative diagram
\begin{equation}\label{eq:pm_glass}
\spaceInd{H_1(X_r), H_1(X_r \cup Y_r), H_1(Y_r),
    H_1(X_s), H_1(X_s \cup Y_s), H_1(Y_s)},
\end{equation}
where all the linear maps are induced by inclusions.
To be more precise, it suffices to study algebraic decompositions of this commutative diagram in a  way similar to standard persistence modules and to investigate the direct summand
\[
\spaceInd{K,K,K,K,K,K}.
\]

From these motivations, we treat quivers of the form 
\begin{equation}\label{eq:cl}
\includegraphics{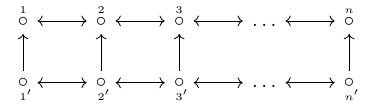}
\end{equation}
in this paper, where $\longleftrightarrow$ expresses either $\longrightarrow$ or $\longleftarrow$, and the orientations of the arrows in the upper sequence are the same as those in the lower sequence. 
We assume that any two directed paths with the same start and end vertices are identified, and call this type of quiver a \emph{commutative ladder} of length $n$. 
Then, we define persistence modules on commutative ladders by representations on them (precise definitions of these concepts are given in Section \ref{sec:preliminaries}). Our primary interest is to extend the theory of topological persistence to this class of quivers. The commutative diagram (\ref{eq:pm_glass})  is thus an example of a persistence module on a commutative ladder of length three.

\subsection{Approach}\label{section:approach}
Our approach treats persistence modules as representations on quivers with nontrivial relations and applies the Auslander-Reiten theory for studying this extension. Here we summarize the essence of this new framework. A detailed account is given in Section \ref{sec:preliminaries} and Section \ref{section:pm_clq}. 

In the paper \cite{cd}, the authors study zigzag persistence modules as representations on quivers, where Gabriel's theorem \cite{gabriel} plays a key role. This theorem provides a complete characterization of quivers that have only finitely many isomorphism classes of indecomposables. A quiver with this property is called representation-finite, and the class of $\A_n$ quivers 
%\[
%\xymatrix{
%1&2&3&&~~~n
%}
%\]
%\vspace*{-0.9cm}
%\[
%\xymatrix{
%\circ \ar@{<->}[r]&\circ \ar@{<->}[r]&\circ\ar@{<->}[r]&\dots\ar@{<->}[r]&\circ
%}
%\]
\[
\includegraphics{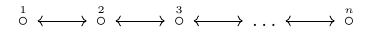}
\]
is shown to be one of the classes of quivers with this property. 
Then, zigzag persistence modules can be formulated as representations on $\A_n$ quivers, and all the isomorphism classes of indecomposables are given by interval representations, which we usually call barcodes \cite{elz,zc} in TDA.

On the other hand, the quivers (\ref{eq:cl}) of interest in this paper possess the commutativity relations. Note that Gabriel's theorem does not treat quivers with such nontrivial relations. Hence, it is still possible for our quivers to be representation-finite even though they are not one of the classes in Gabriel's theorem, and we may develop a similar theory of topological persistence on these quivers.

For the study of representations on quivers with relations such as commutativity, the Auslander-Reiten theory plays an important role. In particular, for each quiver with relations, its Auslander-Reiten quiver captures all the isomorphism classes of indecomposables by a set of vertices and describes relationships among them by a set of arrows. Our approach explicitly constructs the Auslander-Reiten quivers of (\ref{eq:cl}) that are representation-finite. Furthermore, we show that the Auslander-Reiten quivers give a natural generalization of persistence diagrams. Applying these theoretical concepts and tools sheds new light on the theory of topological persistence. In this paper, we explain only specific portions of the Auslander-Reiten theory that are necessary to derive our main results in order for this paper to be self-contained. For a complete treatment of the Auslander-Reiten theory, the reader is referred to \cite{ass,ars,ringel}. 

\subsection{Main Results}\label{section:mainresults}
Our contributions in this paper are as follows:
\begin{enumerate}[(i)]
\item We prove that the commutative ladders (\ref{eq:cl}) of length $n\leq 4$ are representation-finite (Theorem \ref{theorem:cl}) by computing their Auslander-Reiten quivers, which we show in Figures \ref{fig:AR_f} to \ref{fig:AR_bbf}. 
%{\color{blue} We also prove, using this technique, that for some orientations the commutative ladders of length $n\geq 5$ are representation-infinite (Theorem \ref{theorem:cl_infinite}).}
\item We generalize the concept of persistence diagrams to functions on the vertex set of an Auslander-Reiten quiver. 
\item We present an algorithm for computing these generalized persistence diagrams in the case $n=3$ by using the structure of Auslander-Reiten quivers.
\item Finally, we show numerical examples including our original motivation from the TDA on glasses.
\end{enumerate}
\vspace{0.3cm}

This paper is organized as follows. In Section \ref{sec:preliminaries}, we briefly explain the results of \cite{cd} in our framework of representations on associative algebras. This serves as a link to prior works. 
Section \ref{section:pm_clq} is the core section of this paper. Our main theorem about the representation-finite property and the computation of the Auslander-Reiten quivers of the commutative ladders of length $n\leq 4$ is given in Section \ref{section:clq}. For this purpose, in Section \ref{sec:recipe}, we give a recipe for constructing the Auslander-Reiten quivers with a necessary introduction of concepts in the representation theory. 

We also construct the Auslander-Reiten quiver for conventional persistence modules in Section \ref{sec:ar_an} in order to compare this new framework with the existing one. 
This comparison also leads to a natural generalization of persistence diagrams, which we present in Section \ref{sec:generalization_pd}. In Section \ref{sec:arq_clfb}, we give some interpretations of the Auslander-Reiten quivers from the viewpoint of TDA by focusing on our original example (\ref{eq:pm_glass}). 

Section \ref{sec:algorithm} is devoted to presenting an algorithm for computing persistence diagrams. Motivated by our example (\ref{eq:pm_glass}),
we focus only on $n=3$.
Our algorithm computes persistence diagrams by inductively applying echelon form reductions to a given persistence module on a commutative ladder of finite type. In particular, the flowchart of our algorithm closely follows the structure of the corresponding Auslander-Reiten quiver. Based on these theoretical and computational tools, we show some numerical experiments in Section \ref{sec:numerics} to demonstrate the feasibility of our algorithm. 

In Appendix \ref{sec:appendix}, we give a complete list of the Auslander-Reiten quivers of commutative ladders of length $n\leq 4$.

\subsection{Related Works}\label{section:relatedworks}
The goal of this paper is to present a theoretical and algorithmic framework for persistence modules on commutative ladders in a self-contained and accessible way to practitioners of TDA. For this purpose, we adopt Auslander-Reiten quivers as our main tool. There are five reasons of this choice: 
(i) it is tractable and does not require much preparation from quiver representation theory, 
(ii) it can give a simple proof for the representation type once we fix our quiver, 
(iii) it can also give a clear algorithmic foundation,
(iv) the concept of persistence diagrams can be naturally generalized by means of  Auslander-Reiten quivers, and
(v) it can be easily applied to other types of quivers.

There are several works in the representation theory related to our paper. 
For example, the representation types of certain classes of algebras have been extensively studied in the literature. In connection with our paper, note that the triangular matrix algebra of an $\mathbb{A}_n$ quiver is isomorphic to the commutative ladder (\ref{eq:cl}).  
Then, the fact that (\ref{eq:cl}) is representation-finite for $n\leq 4$ and representation-infinite for $n \geq 5$ is a corollary to Theorem 4 of the paper \cite{ttma}. 
Furthermore, there are several papers \cite{ptma,decompo} treating the problem of computing indecomposable decompositions of representations. For example, it was shown in \cite{ptma} that, given a finite dimensional representation $V$ over a finite dimensional $K$-algebra, where $K$ is a finite field, there is a polynomial time algorithm for computing the indecomposable decomposition of $V$. Although these works are done in a general setting, we choose a more constructive strategy for the above reasons. 
We believe that this strategy will enlarge the scope of persistence modules and 
accelerate the integration between TDA and quiver representation theory.

We also remark about the relationship of our work to multidimensional persistence modules. 
The commutative ladders with horizontal arrows all directed forwards can be considered as a restricted version of the bifiltrations considered by Carlsson and Zomorodian in \cite{multi}. They show that there is no complete discrete invariant for persistence modules over multifiltrations in general case. In this paper, by our restriction, we provide a complete discrete invariant (and in fact, a finite one) for commutative ladders of length $n \leq 4$. This is given by the generalized persistence diagram provided in Definition \ref{definition:pd}. 
At the same time, we also remark that the commutative ladders allow some of the horizontal arrows to point backward. 

\section{Preliminaries}\label{sec:preliminaries}
In this section, we first recall some fundamental concepts from the representation theory of associative algebras. For more details, the reader may refer to \cite{ass,ars,ringel}. Then, we review the recent results by Carlsson and de Silva \cite{cd} which provide us with a connection between persistence modules and representations on $\A_n$ quivers.
\subsection{Basics of Representations of Associative Algebras}\label{section:basics_raa}
Let us denote our base field by $K$. A {\em quiver} $Q=(Q_0,Q_1,s,t)$ is given by two sets $Q_0,Q_1$ and two maps $s,t: Q_1\rightarrow Q_0$. The set $Q_0$ is called the set of vertices, and the set $Q_1$ is called the set of arrows. For every arrow $\alpha\in Q_1$, the vertices $s(\alpha)$ and $t(\alpha)$ are called the source and target of $\alpha$, respectively, and we denote the arrow by $\alpha:s(\alpha)\rightarrow t(\alpha)$. We simply denote a quiver by $Q=(Q_0,Q_1)$ without writing the two maps $s,t$. 
The underlying graph $\bar{Q}$ of $Q$ is the undirected graph obtained by removing the orientations of arrows in $Q$.

A {\em path} $(a\mid \alpha_1,\dots,\alpha_\ell\mid b)$ from a source $a\in Q_0$ to a target $b\in Q_0$ of length $\ell$ is given by a list of arrows $\alpha_i\in Q_1$ ($i=1,\dots,\ell$) such that $s(\alpha_1)=a$, $t(\alpha_i)=s(\alpha_{i+1})$ for $i=1,\dots,\ell-1$, and $t(\alpha_\ell)=b$. A path at a vertex $a$ of length $\ell=0$ is called the stationary path and denoted by $\epsilon_a=(a\mid\mid a)$. A path of length $\ell \geq 1$ whose source and target  coincide is called a cycle, and a quiver which contains no cycles is called acyclic. In this paper, we only consider acyclic connected quivers, and assume that $Q_0$ and $Q_1$ are finite.

\begin{definition}{\rm 
Given a quiver $Q$, the path algebra $\KQ$ is defined by a $K$-vector space spanned by all paths in $Q$ as its basis and equipped with a product
\[
	(a\mid \alpha_1,\dots,\alpha_\ell\mid b)(c\mid \beta_1,\dots,\beta_m\mid d)=\delta_{b,c}(a\mid \alpha_1,\dots,\alpha_\ell,\beta_1,\dots,\beta_m\mid d), 
\]
where $\delta_{b,c}$ is the Kronecker delta. The product of two arbitrary elements of $\KQ$ is defined by linear extension of this product.
}\end{definition}

For a quiver $Q$, a {\em relation} $\rho$ is a $K$-linear combination
\[
	\rho = \sum_{i=1}^k c_i w_i,~~~c_i\in K
\]
of paths $w_i$ of length at least two with the same source and target. We note that a set of relations $\{\rho_1,\dots,\rho_s\}$ naturally generates a two-sided ideal $I=\langle \rho_1,\dots,\rho_s\rangle$ in $\KQ$ and defines the quotient algebra $A=\KQ/I$. 
For each $a\in Q_0$, let us denote the class of $\epsilon_a$ in $\KQ/I$ by $e_a = \bar{\epsilon}_a = \epsilon_a+I$.
%The following properties follow from the definition of path algebras. 
%
\begin{lemma}~Let $Q$ be a quiver. The path algebra $\KQ$ is a finite dimensional associative algebra with the identity $1_{\KQ}=\sum_{a\in Q_0}\epsilon_a$, and $\{\epsilon_a\mid a\in Q_0\}$ is a set of orthogonal idempotents in $\KQ$. 
For an ideal $I=\langle \rho_1,\dots,\rho_s\rangle$ generated by relations $\rho_i$, the quotient algebra $A=\KQ/I$ is also a finite dimensional associative algebra with the identity $1_A=\sum_{a\in Q_0}e_a$, and $\{e_a\mid a\in Q_0\}$ is a set of orthogonal idempotents in $A$.

%In both cases, the sets $\{\epsilon _a\mid a\in Q_0\}$ and $\{e_a\mid a\in Q_0\}$ are complete sets of primitive orthogonal idempotents, respectively.
\end{lemma}
\begin{proof}
Since $Q_0$ and $Q_1$ are finite and $Q$ is acyclic, there are only a finite number of paths in $Q$. This implies that $\KQ$ is   finite dimensional. The associativity of $\KQ$ results from the definition of products as compositions of paths. 
It is straightforward to verify that $1_{\KQ}=\sum_{a\in Q_0}\epsilon_a$ is the identity in $\KQ$. 
It follows from $\epsilon_a\epsilon_b=\delta_{a,b}\epsilon_a$ that $\{\epsilon_a\mid a\in Q_0\}$ is a set of orthogonal idempotents. 
%Let $\epsilon$ be an idempotent such that $\epsilon_a=\epsilon+(\epsilon_a-\epsilon)$. 
The statements for $A=\KQ/I$ are proven similarly.
\qed\end{proof}

In this work, we consider only associative algebras defined by quotients of path algebras.
Next, we introduce some basic facts about quiver representations.
\begin{definition}{\rm 
Let $Q$ be a quiver. A representation $M=(M_a,\varphi_\alpha)_{a\in Q_0,\alpha\in Q_1}$ on $Q$ is a set of $K$-vector spaces $\{M_a\mid a\in Q_0\}$ together with linear maps $\varphi_\alpha:M_a\rightarrow M_b$ for each arrow 
$\alpha:a\rightarrow b$ in $Q_1$. 
}
\end{definition}

In this paper, we assume that each vector space on the vertices is finite dimensional. 
For representations $M=(M_a,\varphi_\alpha)$ and $N=(N_a,\psi_\alpha)$ on $Q$, a morphism $f:M\rightarrow N$ is a set of linear maps $f_a:M_a\rightarrow N_a$ at each $a\in Q_0$ such that these maps satisfy the commutativity $\psi_\alpha f_a=f_b\varphi_\alpha$ of the diagram
\[
\xymatrix{
	M_a \ar[r]^-{\varphi_{\alpha}}\ar[d]_-{f_a}&M_b\ar[d]^-{f_b}\\
	N_a \ar[r]_-{\psi_\alpha}&N_b
}
\]
for each $\alpha:a\rightarrow b$.

Let $M=(M_a,\varphi_\alpha)$ be a representation on a quiver $Q$ and  $w=( a\mid \alpha_1,\dots,\alpha_\ell \mid b)$ be a path in $Q$. The {\em evaluation} $\varphi_w$ of $M$ on the path $w$ is the composition of linear maps 
$\varphi_w=\varphi_{\alpha_\ell}\dots\varphi_{\alpha_1}$. For a linear combination $\rho=\sum_{i=1}^kc_i w_i$ of paths $w_i$ with the same source and target, we extend the definition of the evaluation as $\varphi_\rho = \sum_{i=1}^kc_i\varphi_{w_i}$. 
Then, we define a representation $M=(M_a,\varphi_\alpha)_{a\in Q_0,\alpha\in Q_1}$ on $A=\KQ/I$ as a representation on the quiver $Q$ satisfying $\varphi_\rho = 0$ for any relation $\rho\in I$. For $I=\langle \rho_1,\dots,\rho_s\rangle$, $M$ is a representation on $A=\KQ/I$ if and only if $\varphi_{\rho_i}=0$ for $i=1,\dots,s$. A morphism of representations on $A=\KQ/I$ is defined in a way similar to the above. The set of morphisms from $M$ to $N$ is denoted by $\Hom_A(M,N)$.

For a representation $M$ on $A=\KQ/I$, the {\em dimension vector} $\dim M$ is defined by 
\[
	\dim M=(\dim M_a)_{a\in Q_0}.
\]
We also express $\dim M$ by assigning $\dim M_a$ at each $a\in Q_0$ on the given quiver $Q$.
For example, the dimension vector of a representation 
\[
\spaceInd{K,K^3,K^2,K^2,K,0}
\]
is given by $\smdimVec{1,3,2,2,1,0}$.

For a quiver $Q$, let us denote by ${\rm rep}(\KQ)$ the category of representations on $Q$ and their morphisms. 
For $A=\KQ/I$, we denote by ${\rm rep}(A)$ the full subcategory of ${\rm rep}(\KQ)$ consisting of representations on $A$. Let ${\rm mod}(\KQ)$ and ${\rm mod}(A)$ be the categories of finitely generated right $\KQ$-modules and right $A$-modules, respectively. Then, the following proposition states that we can identify representations and modules by categorical equivalences. 
\begin{proposition}
For a quiver $Q$, there exists an equivalence of categories
\[
	{\rm mod}(\KQ)\simeq {\rm rep}(\KQ).
\]
Similarly, given $A=\KQ/I$ where $I$ is an ideal of $\KQ$ generated by relations, there also exists an equivalence of categories
\[
	{\rm mod}(A)\simeq {\rm rep}(A).
\]
\end{proposition}
\begin{proof}
We define two functors $F:{\rm mod}(A)\rightarrow {\rm rep}(A)$ and $G:{\rm rep}(A)\rightarrow {\rm mod}(A)$ as follows.
Let $\{e_a\mid a\in Q_0\}$ be the set of orthogonal idempotents determined by the stationary paths of $A=\KQ/I$.  For an $A$-module $M$, we define a vector space on each $a\in Q_0$ by $F(M)_a=Me_a$. To an arrow $\alpha:a\rightarrow b$, we assign a linear map $\varphi_\alpha: F(M)_a\rightarrow F(M)_b$ by $\varphi_\alpha(x)=x\bar{\alpha}$, where $\bar{\alpha}=\alpha+I\in A$. Note that $\varphi_\alpha(x)=\varphi_\alpha(xe_a)=xe_a \bar{\alpha}=xe_a \bar{\alpha}e_b\in F(M)_b$. Given a morphism $f:M\rightarrow N$ of $A$-modules $M$ and $N$, we define a $K$-linear map $F(f)_a: F(M)_a\rightarrow F(N)_a$ by the restriction of $f$ to $F(M)_a$. This is well-defined because $f(F(M)_a)=f(Me_a)=f(M)e_a\subset F(N)_a$, and we can easily check the commutativity for each $\alpha\in Q_1$. These data define the functor $F:{\rm mod}(A)\rightarrow {\rm rep}(A)$.

On the other hand, let $M=(M_a,\varphi_\alpha)_{a\in Q_0,\alpha\in Q_1}$ be a representation on $A$. As a $K$-vector space, we construct the direct sum $G(M)=\bigoplus_{a\in Q_0}M_a$. Then, the right multiplication of $\KQ$ on $G(M)$ is defined as follows. For a stationary path $\epsilon_a$ at $a\in Q_0$ and $x=(x_b)_{b\in Q_0}\in G(M)$, we define $(x\epsilon_a)_c=\delta_{a,c}x_a$ for $c\in Q_0$. For a path $w=( a\mid \alpha_1,\dots,\alpha_\ell \mid b)$ from $a$ to $b$,  we define $xw$ by $(xw)_c=\delta_{b,c}\varphi_w(x_a)$ for $c\in Q_0$, and we linearly extend it for arbitrary elements in $\KQ$. Note that we have $x\rho = 0$ for any $\rho\in I$. This means that these right multiplications by $\KQ$ naturally define those by $A=\KQ/I$ and determine a  right $A$-module structure on $G(M)$. Given a morphism $f=(f_a): M\rightarrow N$ for representations $M$ and $N$ on $A$, we define a morphism $G(f)=\bigoplus_{a\in Q_0}f_a: G(M)=\bigoplus_{a\in Q_0}M_a\rightarrow G(N)=\bigoplus_{a\in Q_0}N_a$ as the direct sum. These data define the functor $G:{\rm rep}(A)\rightarrow {\rm mod}(A)$. 

It is straightforward to check that these functors lead to an equivalence of categories ${\rm mod}(A)\simeq {\rm rep}(A)$. The first statement follows by taking $I=0$. 
\qed\end{proof}

Let $M=(M_a,\varphi_\alpha)_{a\in Q_0,\alpha\in Q_1}$ be a representation on $A$. A {\em subrepresentation} $N$ of $M$ is given by a set $\{N_a\}_{a\in Q_0}$ of vector subspaces $N_a\subset M_a$ closed under the linear map $\varphi_\alpha(N_a)\subset N_b$ for any $\alpha:a\rightarrow b$. Hence, the subrepresentation $N$ naturally becomes a representation on $A$ by restriction of the linear maps $\varphi_\alpha$. 
A subrepresentation $N$ is called a direct summand if there exists a subrepresentation $N'$ such that $M_a= N_a\oplus N'_a$ for all $a\in Q_0$. In this case, we write $M= N\oplus N'$ and call this a direct sum decomposition of $M$. A representation $M$ is called {\em indecomposable} if $M$ is nonzero and cannot be decomposed as a direct sum of proper subrepresentations.  
On decompositions of representations, the following theorem is fundamental. 
\begin{theorem}[Krull-Remak-Schmidt]\label{theorem:krull}
A representation $M$ on a finite dimensional algebra $A$ can be expressed as a direct sum
\begin{equation}\label{eq:indecomposable_decomposition}
	M\simeq N^{(1)}\oplus \dots \oplus N^{(n)}
\end{equation}
of indecomposables $N^{(1)},\dots,N^{(n)}$. 
In addition, if 
\[
	M\simeq N^{(1)}\oplus\dots\oplus N^{(n)}\simeq L^{(1)}\oplus\dots\oplus L^{(\ell)},
\]
where $N^{(1)},\dots,N^{(n)}$ and $L^{(1)},\dots,L^{(\ell)}$ are indecomposables,
then $n=\ell$ and there exists a permutation $\sigma$ on $\{1,\dots,n\}$ such that
$N^{(i)}\simeq L^{(\sigma(i))}$ for $i=1,\dots,n$.
\end{theorem}

From this theorem, the classification of representations on an associative algebra can be performed by studying indecomposable decompositions (\ref{eq:indecomposable_decomposition}). Hence, given $A=\KQ/I$, it is important to know 
\begin{enumerate}
\item how many isomorphism classes of indecomposables exist, and
\item what kinds of isomorphism classes of indecomposables exist.
\end{enumerate}

An associative algebra is said to be {\em representation-finite} if the number of all isomorphism classes of indecomposables is finite. Otherwise, it is said to be {\em representation-infinite}.
The following theorem provides a complete answer to the above two questions for quiver representations with no relations. Here we give the list of all isomorphism classes of indecomposables only for $\A_n$ quivers. For the complete statement and the lists for the other quivers of finite type, we refer to \cite{ass,ars,gabriel,ringel}. 
\begin{theorem}[Gabriel]\label{theorem:gabriel} 
A quiver $Q$ is representation-finite if and only if the underlying graph $\bar{Q}$ is one of the Dynkin diagrams $\A_n, \D_n, \E_6,\E_7, \E_8$. 
The isomorphism classes of indecomposables on quivers whose underlying graph is the $\A_n$ diagram
%\[
%\xymatrix{
%%\circ_1 \ar@{-}[r]&\circ_2 \ar@{-}[r]&\circ_3\ar@{-}[r]&\ar@{..}[r]&\ar@{-}[r]&\circ_n
%\circ_1 \ar@{-}[r]&\circ_2 \ar@{-}[r]&\circ_3\ar@{-}[r]&\dots\ar@{-}[r]&\circ_n
%}
%\]
\[
\includegraphics{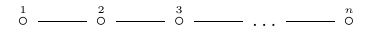}
\]
are given by the interval representations $\I[b,d]$ for $1\leq b\leq d\leq n$ defined by the set of vector spaces
\begin{equation}\label{eq:p_interval}
\I[b,d]_i=\left\{\begin{array}{ll}
K,&b\leq i\leq d\\
0, &{\rm otherwise}
\end{array}\right.
\end{equation}
for $i=1,\dots,n$ and the identity map for each $K\rightarrow K$.
This list of isomorphism classes is independent of orientations of arrows in $\A_n$ quivers.
\end{theorem}
\begin{remark}{\rm 
Note that Gabriel's theorem provides us with a classification of quivers without relations, and does not tell us anything about the associative algebra $A=\KQ/I$ with nontrivial relations $I$. 
Hence, it is possible to have a quiver with some relations to be representation-finite, even if its underlying graph is not included in the list provided by Gabriel's theorem.
%Hence, it is possible to have quivers of finite type by assigning some relations, even if they are not included in the list of Gabriel's theorem.
}\end{remark}

Let $Q$ be an $\A_n$ quiver
%\[
%\xymatrix{
%\circ_1 \ar@{<->}[r]&\circ_2 \ar@{<->}[r]&\circ_3\ar@{<->}[r]&\dots\ar@{<->}[r]&\circ_n
%}
%\]
\[
\includegraphics{./fig2}
\]
where each 
%$\xymatrix{\ar@{<->}[r]&}$ 
$\myarrowLR{1}$
expresses an orientation given by either $\myarrowR{1}$ or $\myarrowL{1}$. By assigning a symbol $f$ or $b$ to $\myarrowR{1}$ or $\myarrowL{1}$, respectively, we can represent an $\A_n$ quiver by a sequence ${\tau_n=O_1\dots O_{n-1}}$ of these symbols. For example, $\tau_3=ff$ and $\tau_4=fbb$ represent the $\A_3$ quiver 
$\begin{tikzcd}\circ \rar &\circ \rar &\circ  \end{tikzcd}$ 
and the $\A_4$ quiver 
$\begin{tikzcd}\circ \rar &\circ &\circ \lar & \circ \lar \end{tikzcd}$ , respectively.

\begin{definition}\label{definition:cl}{\rm 
Let $n\geq 2$. The ladder quiver $L(\tau_n)$
%\[
%\xymatrix{
%\bullet_{1} \ar@{<->}[r]&\bullet_{2} \ar@{<->}[r]&\bullet_{3}\ar@{<->}[r]&\dots\ar@{<->}[r]&\bullet_{n}\\
%\circ_{1}\ar[u]\ar@{<->}[r]&\circ_2\ar[u] \ar@{<->}[r]&\circ_3\ar[u]\ar@{<->}[r]&\dots\ar@{<->}[r]&\circ_n\ar[u]
%}
%\]
\[
\includegraphics{./fig1}
\]
of length $n$ is defined by two $\A_n$ quivers of the same orientation $\tau_n=O_1\dots O_{n-1}$, together with an arrow 
%$\xymatrix{\circ_{i}\ar[r]&\bullet_{i}}$ at each $i=1,\dots,n$
$\begin{tikzcd} \overset{\tiny{i'}}{\circ} \rar & \overset{\tiny{i}}{\circ}\end{tikzcd}$ for every $i=1,\dots,n$.
The commutative ladder $\CL(\tau_n)$ is defined as the associative algebra $\CL(\tau_n)=K\!L(\tau_n)/I$, where the ideal $I$ is generated by commutative relations $\alpha\beta=\gamma\delta$ and $\gamma\alpha=\delta\beta$ in the squares
\[
\includegraphics{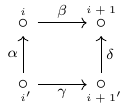}~~\raisebox{+20pt}{ and }~~\includegraphics{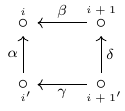}\raisebox{+20pt}{,}
\]
respectively.
}\end{definition}

\subsection{Persistence Modules on $\A_n$}\label{sec:pm_an}
Let $\mathbb{X}: X_1\subset\dots\subset X_n$ be a filtration of topological spaces. The concepts of persistent homology and barcodes for $\mathbb{X}$ are first introduced by  Edelsbrunner, Letscher, and Zomorodian \cite{elz}. Then, Zomorodian and Carlsson \cite{zc} explain persistent homology as a graded $K[z]$-module
%\begin{equation}
%\xymatrix{
%H_\ell(X_1) \ar[r]&H_\ell(X_2) \ar[r]&H_\ell(X_3)\ar[r]&\dots\ar[r]&H_\ell(X_n),
%%1 \ar[r]&2 \ar[r]&3\ar[r]&\ar@{..}[r]&\ar[r]&n
%}
%\end{equation}
\begin{equation}\label{eq:ph}
    \begin{tikzcd}
        H_\ell(X_1) \rar & H_\ell(X_2) \rar & H_\ell(X_3) \rar &
        \hdots \rar & H_\ell(X_n)
    \end{tikzcd}
\end{equation}
over the polynomial ring $K[z]$ with one indeterminate $z$, and call this a persistence module. Here, the action of $z$ on the persistence module is defined by the induced homology maps of $X_i\hookrightarrow X_{i+1}$.
They also show a  clear relationship between barcodes and elementary divisors of the graded $K[z]$-modules. 

In the paper \cite{cd}, Carlsson and de Silva generalize persistence modules as representations
\[
    M:
    \begin{tikzcd}
        H_\ell(X_1) \rar[leftrightarrow] &
        H_\ell(X_2) \rar[leftrightarrow] &
        H_\ell(X_3) \rar[leftrightarrow] &
        \hdots \rar[leftrightarrow] &
        H_\ell(X_n)
    \end{tikzcd}
\]
on any $\A_n$ quiver, where each $\myarrowLR{1}$ expresses either a linear map $\myarrowL{1}$ or $\myarrowR{1}$ determined by the orientation in its $\A_n$ quiver. 
From this extension, we can generally treat a sequence of topological spaces connected by continuous maps
\begin{equation}\label{eq:zigzag}
    \begin{tikzcd}
        X_1 \rar[leftrightarrow] &
        X_2 \rar[leftrightarrow] &
        X_3 \rar[leftrightarrow] &
        \hdots \rar[leftrightarrow] &
        X_n
    \end{tikzcd},
\end{equation}
which is not necessarily a filtration, and study persistent topological features in the sequence.
In this setting, the persistent homology (\ref{eq:ph}) can be regarded as a representation on an $\A_n$ quiver
\begin{equation}\label{eq:Aright_n}
\includegraphics{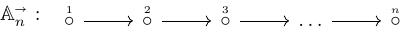}
\raisebox{+5pt}{.}
\end{equation}
We denote this $\A_n$ quiver by $\Aright_n$, since it frequently appears in this paper.

By Theorems \ref{theorem:krull} and \ref{theorem:gabriel}, the persistence module $M$ can be uniquely expressed as an indecomposable decomposition
\begin{equation}\label{eq:pm_decom}
	M\simeq \bigoplus_{1\leq b\leq d\leq n}\I[b,d]^{m_{bd}}, ~~~m_{bd}\in \N_0=\{0,1,2,\dots\}.
\end{equation}
From this decomposition, we obtain a barcode as a multiset of closed intervals $[b,d]$ appearing in \eqref{eq:pm_decom} even in the setting of persistence modules on $\A_n$ quivers with arbitrary orientations. Hence, we define persistence diagrams as follows:
\begin{definition}\label{definition:pd_an}{\rm 
The persistence diagram $D_M$ of a persistence module $M$ on an $\A_n$ quiver is a multiset of points in $\{(b,d)\in \N^2\mid 1\leq b\leq d\leq n\}$ such that the multiplicity of $(b,d)$ is $m_{bd}$ from the indecomposable decomposition (\ref{eq:pm_decom}). }
\end{definition}

A point $(b,d)$ in a persistence diagram, or equivalently an interval representation $\I[b,d]$, encodes a topological feature in (\ref{eq:zigzag}) which appears at $X_b$ and disappears at $X_{d+1}$. This means that the birth, death, and lifetime of this feature are given by $b, d+1$, and $d-b+1$, respectively. We note that in this setting, the death time $d+1$ is greater than the coordinate $d$ by 1. Hence, the diagonal consisting of generators with zero lifetime, which is necessary to discuss stability theorems \cite{proximity,stability}, is given by $\{(i,i-1)\mid 1\leq i\leq n\}$. 
In order to deal with the diagonal in Section \ref{sec:generalization_pd}, we introduce the notation $Z(i)=\I[i,i-1]$ corresponding to the point $(i,i-1)\in \N^2_0$ and call this the {\em zero representation} at $(i,i-1)$. 

In view of these progresses, we define a {\em persistence module} on $A=\KQ/I$ as a representation on $A$. In particular, we study a generalization of the above scenario into persistence modules on commutative ladders.
For this purpose, we consider the following issues:
\begin{enumerate}[(Q1)]
\item Are commutative ladders representation-finite or not?
\item If they are representation-finite, what are the counterparts of interval representations and persistence diagrams?
\end{enumerate}
In the next section, we give answers to these questions.
\section{Persistence Modules on Commutative Ladders}\label{section:pm_clq}

\subsection{Auslander-Reiten Quivers}
We introduce two concepts from the representation theory of associative algebras. In this section, let $A=\KQ/I$ be an associative algebra defined by a quiver $Q$ and an ideal $I=\langle \rho_1,\dots,\rho_s\rangle$ generated by relations $\rho_i$.
\begin{definition}\label{definition:irreducible}{\rm 
Let $M,N$ be representations in $\rep(A)$. A morphism $f:M\rightarrow N$ is called irreducible if 
\begin{enumerate}[(i)]
\item $f$ is neither a section nor a retraction, and
\item any factorization $f=f_1f_2$:
\[
\xymatrix{
M \ar[rr]^{f}\ar[rd]_{f_2}&&N\\
&W\ar[ru]_{f_1}&
}
\]
implies that $f_1$ is a retraction or $f_2$ is a section.
\end{enumerate}
}\end{definition}

Recall that a morphism $s:M\rightarrow N$ is called a {\em section} if it has a left inverse. Dually, a morphism $r:N\rightarrow M$ is called a {\em retraction} if it has a right inverse. It follows from (i) that an irreducible morphism $f$ cannot be a split monomorphism nor a split epimorphism. On the other hand, (ii) implies that $f$ does not allow nontrivial factorizations. In this sense, we can think of this definition as capturing the concept of irreducibility with respect to factorizations of morphisms.

\begin{definition}\label{definition:ar_quiver}{\rm 
The Auslander-Reiten quiver $\Gamma(A)=(\Gamma_0,\Gamma_1)$ of $A$ is defined as follows:
\begin{enumerate}
\item The vertices in $\Gamma_0$ are all the isomorphism classes of indecomposables in $\rep(A)$.
\item For two indecomposable isomorphism classes $[M]$ and $[N]$, an arrow $[M]\rightarrow [N]$ is assigned in $\Gamma_1$ if and only if there exists an irreducible morphism $M\rightarrow N$ (this is independent of the choice of representatives).
\end{enumerate}
}
\end{definition}
\begin{remark}
{\rm 
This definition of the Auslander-Reiten quiver is slightly different from the usual to simplify the presentation. In the usual definition, the set of vertices is the same, but we assign $\ell=\dim {\rm Irr}(M,N)$ arrows from $[M]$ to $[N]$, where ${\rm Irr}(M,N)$ is the $K$-vector space of irreducible morphisms. Hence, Definition \ref{definition:ar_quiver}  corresponds to projecting multiple arrows into one, and we take only the existence of irreducible morphisms into account. This simple version is sufficient for the arguments in this paper.
}
\end{remark}

The Auslander-Reiten quiver $\Gamma(A)$ has the following characteristic property. 
\begin{proposition}\label{proposition:finitecc}
If there exists a finite connected component $\Gamma$ in the Auslander-Reiten quiver $\Gamma(A)$, then $\Gamma=\Gamma(A)$.
\end{proposition}
For a proof, the reader is referred to page 141 in \cite{ass} or page 79  in \cite{ringel}. This proposition is used to explicitly construct Auslander-Reiten quivers in later sections.

\subsection{Representation Type of Commutative Ladders}\label{section:clq}
The following theorem provides us with commutative ladders of finite type and explicitly shows their Auslander-Reiten quivers.
\begin{theorem}\label{theorem:cl}
The commutative ladders $\CL(\tau_n)$ of length $n\leq 4$ are representation-finite for arbitrary orientations $\tau_n$. The Auslander-Reiten quivers of $\CL(\tau_n)$ with $n\leq 4$ are shown in Figures \ref{fig:AR_f} to \ref{fig:AR_bbf}.
\end{theorem}

Figures \ref{fig:AR_f} to \ref{fig:AR_bbf} are placed in Appendix \ref{sec:appendix}, where the isomorphism class of each indecomposable is expressed by its dimension vector. In Section \ref{sec:proof}, we give a proof of this theorem for the commutative ladder $\CL(fb)$
\[
\xymatrix{
\circ \ar[r]&\circ&\ar[l]\circ\\
\circ \ar[u]\ar[r]&\ar[u]\circ \ar[u] &\ar[l]\circ\ar[u]
}
\]
by explicitly constructing its Auslander-Reiten quiver, which is shown also in Figure \ref{fig:AR_fb_maintext} for the sake of convenience.
\begin{figure}[h!]
\begin{center}
\includegraphics[width=14cm]{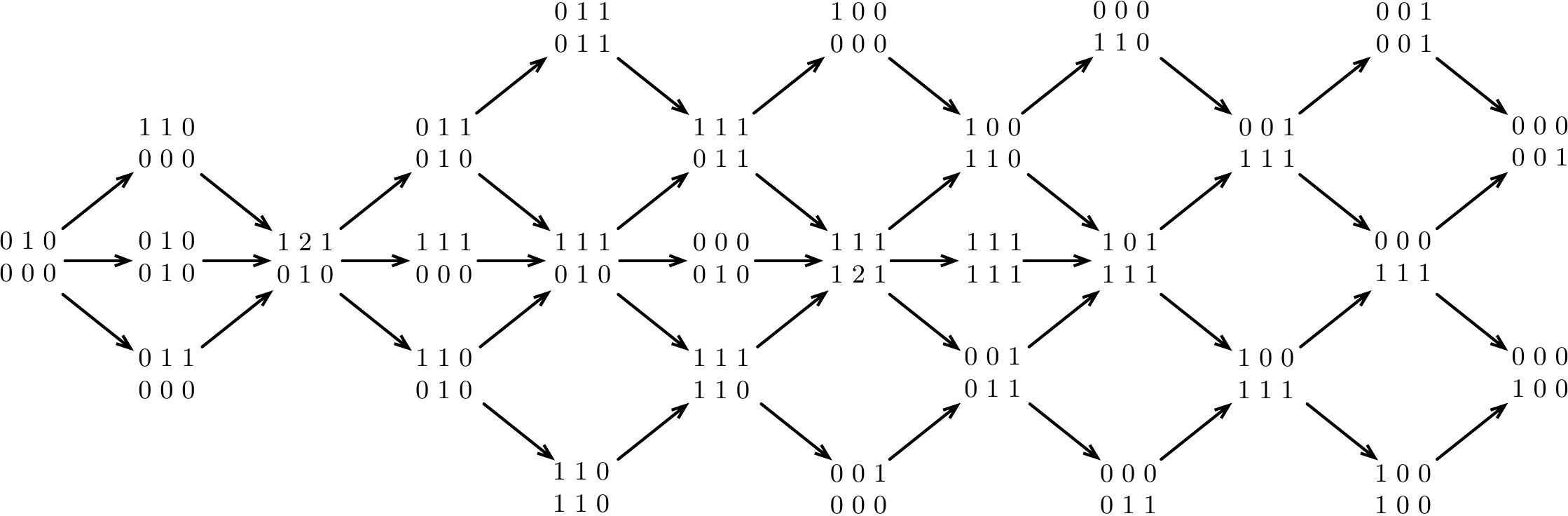}
\caption{The Auslander-Reiten quiver of $\CL(fb)$.}
\label{fig:AR_fb_maintext}
\end{center}
\end{figure}
We note that there exists an up-down symmetry in $\Gamma(\CL(fb))$ induced by the left-right symmetry of the quiver $\CL(fb)$.
In each indecomposable, each linear map connecting two one dimensional vector spaces is the identity map $K\rightarrow K$, while those in 
$\smdimVec{1,2,1,1,1,1}$
and 
$\smdimVec{0,1,0,1,2,1}$
are given in matrix forms by
\begin{eqnarray}
&&\xymatrix{
K \ar[r]&K&\ar[l]K\\
K \ar[r]_-{f_{12}}\ar[u]&K^{2}\ar[u]^-{f_{25}} &\ar[l]^-{f_{32}}K\ar[u]
},~~~
f_{12}=\left[\begin{array}{c}1\\0\end{array}\right],~
f_{32}=\left[\begin{array}{c}0\\1\end{array}\right],~
f_{25}=\left[\begin{array}{cc}1 &1\end{array}\right],\label{eq:twodimentry1}\\
&&\xymatrix{
K \ar[r]^-{f_{45}}&K^2&\ar[l]_-{f_{65}}K\\
0 \ar[r]\ar[u]&K\ar[u]^-{f_{25}} &\ar[l]0\ar[u]
},~~~
f_{45}=\left[\begin{array}{c}1\\0\end{array}\right],~
f_{65}=\left[\begin{array}{c}0\\1\end{array}\right],~
f_{25}=\left[\begin{array}{c}1 \\1\end{array}\right].\label{eq:twodimentry2}
\end{eqnarray}
As explained in Section \ref{sec:introduction}, persistence modules on the commutative ladder $\CL(fb)$ frequently arise in TDA. In Section \ref{sec:arq_clfb}, we discuss in detail the interpretations of these isomorphism classes of indecomposables for TDA.

The next theorem shows that the converse of Theorem \ref{theorem:cl} also holds.
\begin{theorem}\label{theorem:cl_infinite}
The commutative ladders $\CL(\tau_n)$ of length $n\geq 5$ are representation-infinite for arbitrary orientations $\tau_n$. \end{theorem}

This follows as a corollary to Theorem 4 of the paper \cite{ttma}. In Section \ref{sec:proof}, we give a constructive proof of this statement using the Auslander-Reiten quivers.

% for the orientations $\tau_n=O_1\dots O_{n-1}$ given by
%\begin{enumerate}[{\rm (i)}]
%\item $O_i=f$ for $i=1,\dots,n-1$, or
%\item $O_i=\left\{\begin{array}{ll}
%f,&i=1~{\rm (mod~2)}\\
%b,&i=0~{\rm (mod~2)}
%\end{array}
%\right.$ for $i=1,\dots,n-1$~or,
%\item $O_i=\left\{\begin{array}{ll}
%b,&i=1~{\rm (mod~2)}\\
%f,&i=0~{\rm (mod~2)}
%\end{array}
%\right.$ for $i=1,\dots,n-1$,
%\end{enumerate}
%with $n\geq 5$.
%}
%
%We remark that the orientation (i) may arise in TDA for filtrations $X_1\hookrightarrow \dots\hookrightarrow X_n$ of topological spaces, while (ii) and (iii) may appear for sequences of the types $X_j\hookrightarrow X_j\cup X_{j+1}\hookleftarrow X_{j+1}$ and $X_j\hookleftarrow X_j\cap X_{j+1}\hookrightarrow X_{j+1}$, respectively.

%
%
%
%
\subsection{Recipe for Auslander-Reiten Quivers}\label{sec:recipe}
We introduce a recipe to construct the Auslander-Reiten quiver $\Gamma(A)$ of $A=\KQ/I$. For this purpose, let us recall some fundamental definitions and their properties. We basically follow the exposition in \cite{ass} with appropriate simplifications (see also \cite{ars,ringel}).

A representation $P$ in $\rep(A)$ is {\em projective} if, for any epimorphism $f: M\rightarrow N$ and any morphism $g:P\rightarrow N$, there exists a morphism $h:P\rightarrow M$ satisfying $g=fh$. Dually, a representation $I$ is {\em injective} if, for any monomorphism $f: L\rightarrow M$ and any morphism $g:L\rightarrow I$, there exists a morphism $h:M\rightarrow I$ satisfying $g=hf$. A representation $S$ is called {\em simple} if $S\neq 0$ and any subrepresentation of $S$ is either zero or $S$. Furthermore, given a representation $M$, the {\em radical} $\rad M$ of $M$ is the intersection of all maximal subrepresentations of $M$. 
\begin{definition}{\rm 
Let $a\in Q_0$ be a vertex of the quiver $Q=(Q_0,Q_1)$ defining $A=KQ/I$.
\begin{enumerate}
\item $P(a)=(P(a)_b,\varphi_\alpha)_{b\in Q_0,\alpha\in Q_1}$ is the representation consisting of a vector space $P(a)_b$ at each $b\in Q_0$ spanned by the paths from $a$ to $b$ in $A$, and a linear map $\varphi_\alpha$ for each $\alpha:b\rightarrow c$ given by the right multiplication of $\bar{\alpha}$ on $P(a)_b$.
\item $I(a)=(I(a)_b,\varphi_\alpha)_{b\in Q_0,\alpha\in Q_1}$ is the representation consisting of a vector space $I(a)_b$ at each $b\in Q_0$ which is the dual vector space spanned by the paths from $b$ to $a$ in $A$, and a linear map $\varphi_\alpha$ for each $\alpha:b\rightarrow c$ given by the dual of the left multiplication of $\bar{\alpha}$.
\item $S(a)=(S(a)_b,\varphi_\alpha)_{b\in Q_0,\alpha\in Q_1}$ is the representation consisting of a vector space 
\[
	S(a)_b=\left\{\begin{array}{ll}
	K, & b=a\\
	0, & b\neq a
	\end{array}\right.
\]
at each $b\in Q_0$, and $\varphi_\alpha=0$ for every $\alpha\in Q_1$.
\item $\rad P(a)=(M_b,\varphi_\alpha)_{b\in Q_0,\alpha\in Q_1}$ is the representation consisting of a vector space $M_b$ at each $b\in Q_0$ spanned by the non-stationary paths from $a$ to $b$ in $A$, and a linear map $\varphi_\alpha$ for each $\alpha:b\rightarrow c$ given by the right multiplication of $\bar{\alpha}$ on $M_b$.
\end{enumerate}
}\end{definition}

We can easily verify that $S(a)$ is simple for each $a\in Q_0$.
It is also known that all the indecomposable projectives and injectives are precisely given by $P(a)$ and $I(a)$ for 
$a\in Q_0$, respectively, and that $\rad P(a)$ is the radical of $P(a)$ (e.g., see III.2 in \cite{ass}).

%
%
%\begin{proposition}\label{proposition:simpleprojective}
%Let $f: M\rightarrow P$ be irreducible and $P$ be simple projective. Then, $M=0$. Dually, let $g:I\rightarrow N$ be irreducible and $I$ be simple injective. Then, $N=0$.
%\end{proposition}
\begin{proposition}\label{proposition:simpleprojective}
For a simple projective $P$, there are no irreducible morphisms ending at $P$. Dually, for a simple injective $I$, there are no irreducible morphisms starting at $I$.
\end{proposition}
\begin{proof}
Let $f:M\rightarrow P$ be an irreducible morphism ending at a simple projective $P$. It can be easily checked that $f$ must be nonzero. Thus, $f$ is surjective because $P$ is simple. Since $P$ is projective, $f$ is a retraction, which is a contradiction. The second statement can be proved similarly.

%Thus, $f=0$ and $f$ factors through $0$ as 
%\[
%    \begin{tikzcd}
%        M \arrow{rr}{f} \arrow[swap]{dr}{f_2=0} &  & P \\
%        &   0 \arrow[swap]{ur}{f_1=0} & 
%    \end{tikzcd}.
%\]
%Since $P$ is nonzero, $f_1$ cannot be a retraction, so that $f_2$ is a section. Clearly, this implies that $M=0$ and that $f$ is a section, and contradicts the irreducibility of $f$. The second statement can be proved similarly.
%}
%
%Suppose that there is an irreducible morphism $f:M\rightarrow P$ ending at $P$.
%It follows from the projectivity of $P$ that $\image f\subsetneq P$, otherwise $f$ would be a retraction, 
%in contradiction to its irreducibility. We note that the radical $\rad P$ of the indecomposable projective $P$ is given by the unique maximal subrepresentation. 
%This implies a factorization of $f$:
%\[
%    \begin{tikzcd}
%        M \arrow{rr}{f} \arrow[swap]{dr}{\tilde{f}} &  & P \\
%        &   \rad P \arrow[swap]{ur}{i} & 
%    \end{tikzcd},
%\]
%where $\tilde{f}$ is naturally defined by $f$ and $i$ is an inclusion. Since $f$ is irreducible and $i$ is not a retraction, $\tilde{f}$ is a section. Then, $M$ is a direct summand of $\rad P$. However, since $P$ is simple, we have $\rad P=0$,  leading to $M=0$. This implies that $f$ is a section, and contradicts the irreducibility of $f$. The second statement can be proved similarly.
\qed\end{proof}

\begin{definition}{\rm 
A morphism $f: L\rightarrow M$ is called left minimal if every endomorphism $h:M\rightarrow M$ with $hf=f$ is an automorphism. Dually, a morphism $g: M\rightarrow N$ is called right minimal if every endomorphism $h:M\rightarrow M$ with $gh=g$ is an automorphism.
}
\end{definition}
\begin{definition}{\rm 
A morphism $f: L\rightarrow M$ is called left almost split if $f$ is not a section, and, for any $f':L\rightarrow M'$ that is not a section, there exists $h:M\rightarrow M'$ such that $f'=hf$. Dually,
A morphism $g: M\rightarrow N$ is called right almost split if $g$ is not a retraction, and, for any $g':M'\rightarrow N$ that is not a retraction, there exists $h:M'\rightarrow M$ such that $g'=gh$. 
}
\end{definition}

A morphism is called {\em left minimal almost split} (or {\em right minimal almost split}) if it is left (resp., right) minimal and left (resp., right) almost split. 

We adopt the following convention to express morphisms on direct sums. Given morphisms $f_i: L\rightarrow M_i$ ($i=1,2$), the morphism $f=(f_1~f_2)^t:L\rightarrow M_1\oplus M_2$ is defined as $f(x)=(f_1(x),f_2(x))$. Similarly, given $g_i: M_i\rightarrow N$ ($i=1,2$), the morphism $g=(g_1~g_2):M_1\oplus M_2\rightarrow N$ is defined as $g(x_1,x_2)=g_1(x_1)+g_2(x_2)$.
\begin{proposition}\label{proposition:al_indecomposable}
Given a left almost split morphism $f: L\rightarrow M$, $L$ is indecomposable. Dually, 
given a right almost split morphism $g: M\rightarrow N$, $N$ is indecomposable.
\end{proposition}
\begin{proof}
If not, let $L=L_1\oplus L_2$ be a nontrivial direct sum decomposition and $p_i:L\rightarrow L_i$ be the corresponding projections. Since $p_i$ is not a monomorphism, it is not a section. Then, there exists $h_i:M\rightarrow L_i$ such that $p_i=h_if$. However, it follows from $(h_1~h_2)^tf=(p_1~p_2)^t=1_L$ that $f$ is a section, and this contradicts the assumption that $f$ is left almost split. The second statement is proved in a similar way.
\qed\end{proof}

The following is a uniqueness property (up to isomorphism) of left and right minimal almost split morphisms.
\begin{proposition}\label{proposition:mal_unique}
Let $f: L\rightarrow M$ and $f': L\rightarrow M'$ be left minimal almost split. Then, there exists an isomorphism $h: M\rightarrow M'$ such that $f'=hf$. Dually, 
let $g: M\rightarrow N$ and $g': M'\rightarrow N$ be right minimal almost split. Then, there exists an isomorphism $h: M\rightarrow M'$ such that $g=g'h$.
\end{proposition}
\begin{proof}
Since $f$ and $f'$ are not sections, there exist $h:M\rightarrow M'$ and $h':M'\rightarrow M$ such that $f'=hf$ and $f=h'f'$, respectively. These lead to $f=h'hf$ and $f'=hh'f'$. Then, the minimality of $f$ and $f'$ implies that $h'h$ and $hh'$ are automorphisms, and hence $h$ and $h'$ are isomorphisms. The second statement is proved similarly. 
\qed\end{proof}

\begin{proposition}\label{proposition:irr_mal}
Let $f: L\rightarrow M$ be left minimal almost split. Then, $f$ is irreducible. 
Furthermore, $f':L\rightarrow M'$ is irreducible if and only if $M'\neq 0$ and there exists a direct sum decomposition $M\simeq M'\oplus M''$ and a morphism $f'':L\rightarrow M''$ such that $(f'~f'')^t:L\rightarrow M'\oplus M''$ is left minimal almost split. 
Dually, let $g: M\rightarrow N$ be right minimal almost split. Then, $g$ is irreducible. 
Furthermore, $g':M'\rightarrow  N$ is irreducible if and only if $M'\neq 0$ and there exists a direct sum decomposition $M\simeq M'\oplus M''$ and a morphism $g'':M''\rightarrow N$ such that $(g'~g''): M'\oplus M''\rightarrow N$ is right minimal almost split.
\end{proposition}
For a proof, the reader may refer to page 103 in \cite{ass}. 
Thus, by Propositions \ref{proposition:mal_unique} and \ref{proposition:irr_mal}, all the irreducible morphisms starting at $L$ (ending at $N$) can be obtained by studying the left (resp. right) minimal almost split morphisms from $L$ (resp. to $N$).

We give some characterizations of left and right minimal almost split morphisms. First of all, let us begin with left minimal almost split morphisms starting at indecomposable injectives and right minimal almost split morphisms ending at indecomposable projectives.
\begin{proposition}\label{proposition:mal_pro_inj}
Let $P(a)$ be an indecomposable projective. Then, a morphism $g:M\rightarrow P(a)$ is right minimal almost split if $g$ is a monomorphism and $\image g \simeq \rad P(a)$. Dually, let $I(a)$ be an indecomposable injective. Then, a morphism $f:I(a)\rightarrow M$ is left minimal almost split if $f$ is an epimorphism and $\kernel f\simeq S(a)$.
\end{proposition}
\begin{proof}
It suffices to prove the case $g: \rad P(a)\hookrightarrow P(a)$. Since $g$ is a monomorphism, any $h:\rad P(a)\rightarrow \rad P(a)$ with $gh=g$ is an automorphism. Thus, $g$ is right minimal. 
Obviously, $g$ is not a retraction. 
Furthermore, let $g': M'\rightarrow P(a)$ be any morphism that is not a retraction. Then, $g'$ is not an epimorphism, otherwise the projectivity of $P(a)$ implies that $g'$ is a retraction. 
Since $\rad P(a)$ is the unique maximal subrepresentation in $P(a)$, $g'$ factors through $\rad P(a)$. Hence, $g$ is also right almost split. The statement for $I(a)$ is proved in a similar way.
\qed\end{proof}

Next, we give a characterization of left minimal almost split morphisms starting at indecomposable non-injectives and right minimal almost split morphisms ending at indecomposable non-projectives, respectively. But, before that, we introduce a few more concepts. 

A subrepresentation $L$ of $M$ is called superfluous if for every subrepresentation $X$ of $M$ the equality $L+X=M$ implies $X=M$. An epimorphism $h:P\longrightarrow M$ in $\rep(A)$ is called a projective cover of $M$ if $P$ is projective and $\kernel h$ is superfluous in $P$. 
For a representation $M$, let 
\[
\xymatrix{
	P_1\ar[r]^p&P_0\ar[r]^q&M\ar[r] & 0
}
\]
be a minimal projective presentation in $\rep(A)$. 
Namely, it is an exact sequence such that 
$\xymatrix{
	P_1\ar[r]^p&\kernel q
}
$
and 
$\xymatrix{
	P_0\ar[r]^q&M
}
$
are projective covers.

Let $A^{\textrm{op}}$ be the opposite algebra of $A$. Then, by applying the functor $(-)^t=\Hom_A(-,A): \rep(A) \rightarrow \rep(A^{\textrm{op}})$ to the sequence above, we obtain an exact sequence
\[
\xymatrix{
	0\ar[r]&M^t\ar[r]&P_0^t\ar[r]^{p^t}&P_1^t\ar[r] & \trace M\rightarrow 0,
}
\]
where $\trace M=\cokernel p^t \in \rep(A^{\textrm{op}})$ is called the {\em transpose} of $M$. We remark that $\trace M$ is well-defined up to isomorphism, by the uniqueness of minimal projective presentations up to isomorphism. We denote by $D(-)=\Hom_K(-,K)$ the duality functors
\[
\begin{tikzcd}
\rep(A) \rar{D} & \rep(A^{\textrm{op}}) \rar{D} &\rep(A).
\end{tikzcd}
\]
\begin{definition}{\rm 
%For an associative algebra $A=\KQ/I$, 
The Auslander-Reiten translations $\tau,\tau^{-1}:\rep(A)\rightarrow\rep(A)$ are defined by the compositions
\[
	\tau=D \trace ~~~{\rm and}~~~\tau^{-1}=\trace D.
\]
}\end{definition}
We warn that $\tau$ and $\tau^{-1}$, taken as functors from $\rep(A)$ to itself, are not inverses of each other.
For example, it is easy to check that $N$ is projective if and only if $\tau N = 0$. 
To explain the notation, we note that they do induce mutually inverse equivalences between the projectively stable and injectively stable categories of $A$. See Corollary 2.11 in IV.2 of \cite{ass}.
These translations play an important role in determining the Auslander-Reiten quiver of $A$. In Section \ref{sec:ar_an} we illustrate this with the quiver $\Aright_n$.
\begin{definition}\label{definition:ass}{\rm 
A short exact sequence 
\[
\xymatrix{
	0\ar[r]&L\ar[r]^f&M\ar[r]^g&N\ar[r] & 0
}
\]
is called an almost split sequence if $f$ is left minimal almost split and $g$ is right minimal almost split.
}
\end{definition}
We note that an almost split sequence starting from $L$ or ending at $N$ is uniquely determined up to isomorphism by Proposition \ref{proposition:mal_unique}.
The following proposition gives a characterization of left and right minimal almost split morphisms for indecomposable non-injectives and non-projectives.
\begin{proposition}\label{proposition:ass}
For any indecomposable non-projective $N$, there exists an almost split sequence 
\begin{equation}\label{eq:ass_to_nonpro}
\xymatrix{
	0\ar[r]&\tau N\ar[r]&M\ar[r]&N\ar[r] & 0.
}
\end{equation}
Dually, for any indecomposable non-injective $L$, there exists an almost split sequence 
\begin{equation}\label{eq:ass_from_noninj}
\xymatrix{
	0\ar[r]&L\ar[r]&M\ar[r]&\tau^{-1}L\ar[r] & 0.
}
\end{equation}
\end{proposition}
For a proof, the reader is referred to page 120 in \cite{ass}. 
This proposition asserts that the left minimal almost split morphism starting at an indecomposable non-injective $L$ appears in the almost split sequence $\xymatrix{
	0\ar[r]&L\ar[r]&M\ar[r]&\tau^{-1}L\ar[r] & 0
}
$, and its dual statement for indecomposable non-projectives.
We can use these almost split sequences to construct new indecomposables, since $\tau N$ and $\tau^{-1}L$ in (\ref{eq:ass_to_nonpro}) and (\ref{eq:ass_from_noninj}) are indecomposables by Proposition \ref{proposition:al_indecomposable}. 
Furthermore, it follows from Proposition \ref{proposition:irr_mal} that these almost split sequences also produce new irreducible morphisms. 

In this paper, we frequently use 
almost split sequences of the form (\ref{eq:ass_from_noninj}) to construct Auslander-Reiten quivers. For this purpose, we introduce a characterization of the Auslander-Reiten translation $\tau^{-1}$ by using the so-called Nakayama functor $\nu^{-1}(-)=\Hom_A(DA,-) = (D-)^t: \rep(A) \rightarrow \rep(A)$. It is known that $\nu^{-1}$ induces a duality from injective representations to projective representations. Here, we use $\nu^{-1}$ to compute the translation $\tau^{-1}L$ of an indecomposable non-injective $L$.
\begin{proposition}\label{proposition:nakayama}
For a representation $L$, let 
$
\xymatrix{
	0\ar[r]&L\ar[r]^{i_0}&E_0\ar[r]^{i_1}&E_1
}
$ be a minimal injective presentation in $\rep(A)$. Then, there exists an exact sequence
\[
\xymatrix{
	0\ar[r]&\nu^{-1}L\ar[r]^{\nu^{-1}i_0}&\nu^{-1}E_0\ar[r]^{\nu^{-1}i_1}&\nu^{-1}E_1\ar[r]&\tau^{-1}L\ar[r] &0.
}	
\]
\end{proposition}
\begin{proof}
The statement is proved by successively applying the functors $D$ and $(-)^t$ and using a functorial isomorphism $(DX)^t\simeq \nu^{-1}X$.
\qed\end{proof}

Based on these propositions, we provide a recipe to construct the Auslander-Reiten quiver $\Gamma(A)$ for $A=\KQ/I$ as follows:
\begin{enumerate}[(i)]
\item Add vertices to $\Gamma(A)$ for all indecomposable projectives $P(a)$, $a\in Q_0$. By Proposition \ref{proposition:simpleprojective}, the vertices for simple projectives do not have incoming arrows in $\Gamma(A)$.
\item For irreducible morphisms starting at indecomposable injectives or ending at indecomposable projectives, apply Proposition \ref{proposition:mal_pro_inj} and then Proposition \ref{proposition:irr_mal}.
\item For an indecomposable non-injective $L$, construct the $\tau^{-1}L$ by using Proposition \ref{proposition:nakayama}, and 
derive the almost split sequence 
\[
\xymatrix{
	0\ar[r]&L\ar[r]^f&M\ar[r]^g&\tau^{-1}L\ar[r] & 0.
}
\]
Then, add the vertex $[\tau^{-1}L]$ to $\Gamma(A)$, since, by Proposition \ref{proposition:al_indecomposable},  $\tau^{-1}L$ is an indecomposable. 
Furthermore, add newly obtained vertices $[M_i]$ and arrows $[L]\rightarrow [M_i]$ and $[M_i]\rightarrow [\tau^{-1}L]$ in $f=(f_1\dots f_s)^t: L\rightarrow M=\bigoplus_{i=1}^sM_i$ and $g=(g_1\dots g_s): M=\bigoplus_{i=1}^sM_i\rightarrow \tau^{-1}L$, respectively, where $M=\bigoplus_{i=1}^sM_i$ is an indecomposable decomposition of $M$.
By Proposition \ref{proposition:irr_mal}, we can obtain all the irreducible morphisms starting at $L$ and ending at $\tau^{-1}L$ in this manner. 
\item Once we obtain a finite connected component $\Gamma$, this is the Auslander-Reiten quiver $\Gamma(A)=\Gamma$ by Proposition \ref{proposition:finitecc}. 
\end{enumerate}

Part (iii) is to be done inductively.
By Proposition \ref{proposition:simpleprojective}, the inductive application of (iii) stops when we obtain simple injectives. While doing this inductive process, if we get an indecomposable injective $L$ for example, then we cannot use (iii) to compute the irreducible morphisms starting from $L$. In such a case, we use part (ii).
It should also be mentioned that since we capture all the irreducible morphisms using (ii) and (iii), we can completely derive the connected component if $\Gamma(A)$ is finite.
We warn that the above recipe does not terminate in the representation-infinite case.

\subsection{Auslander-Reiten Quiver of $\Aright_n$}\label{sec:ar_an}
Before proving Theorem \ref{theorem:cl}, let us consider the Auslander-Reiten quiver of persistence modules on $\Aright_n$ 
%(defined in (\ref{eq:Aright_n}))
%\begin{equation}\label{eq:Aright_n}
%\xymatrix{
%\circ_1 \ar[r]&\circ_2 \ar[r]&\circ_3\ar[r]&\ar@{..}[r]&\ar[r]&\circ_n
%}
%\end{equation}
and study how to use the recipe presented in the previous subsection. Let us denote the path algebra of $\Aright_n$ by $A$.

\begin{theorem}\label{theorem:ar_an}
Let $\Gamma=(\Gamma_0,\Gamma_1)$ be the Auslander-Reiten quiver of $\Aright_n$.
Then, the set $\Gamma_0$ of vertices consists of the isomorphism classes of interval representations $\I[b,d]$ for $1\leq b\leq d\leq n$, and the set $\Gamma_1$ of arrows consists of the irreducible morphisms
\[
    \begin{tikzcd}
        \I[b-1,d] & \I[b,d] \lar\dar \\
                 & \I[b,d-1]
    \end{tikzcd} 
\]
for $1<d\leq n$, where $1<b\leq d$ for horizontal arrows and $1\leq b <d$ for vertical arrows.
Figure \ref{fig:ar_an} shows $\Gamma$ sketched in $\N^2$.
\begin{figure}[h!]
\begin{center}
\includegraphics[width=6cm]{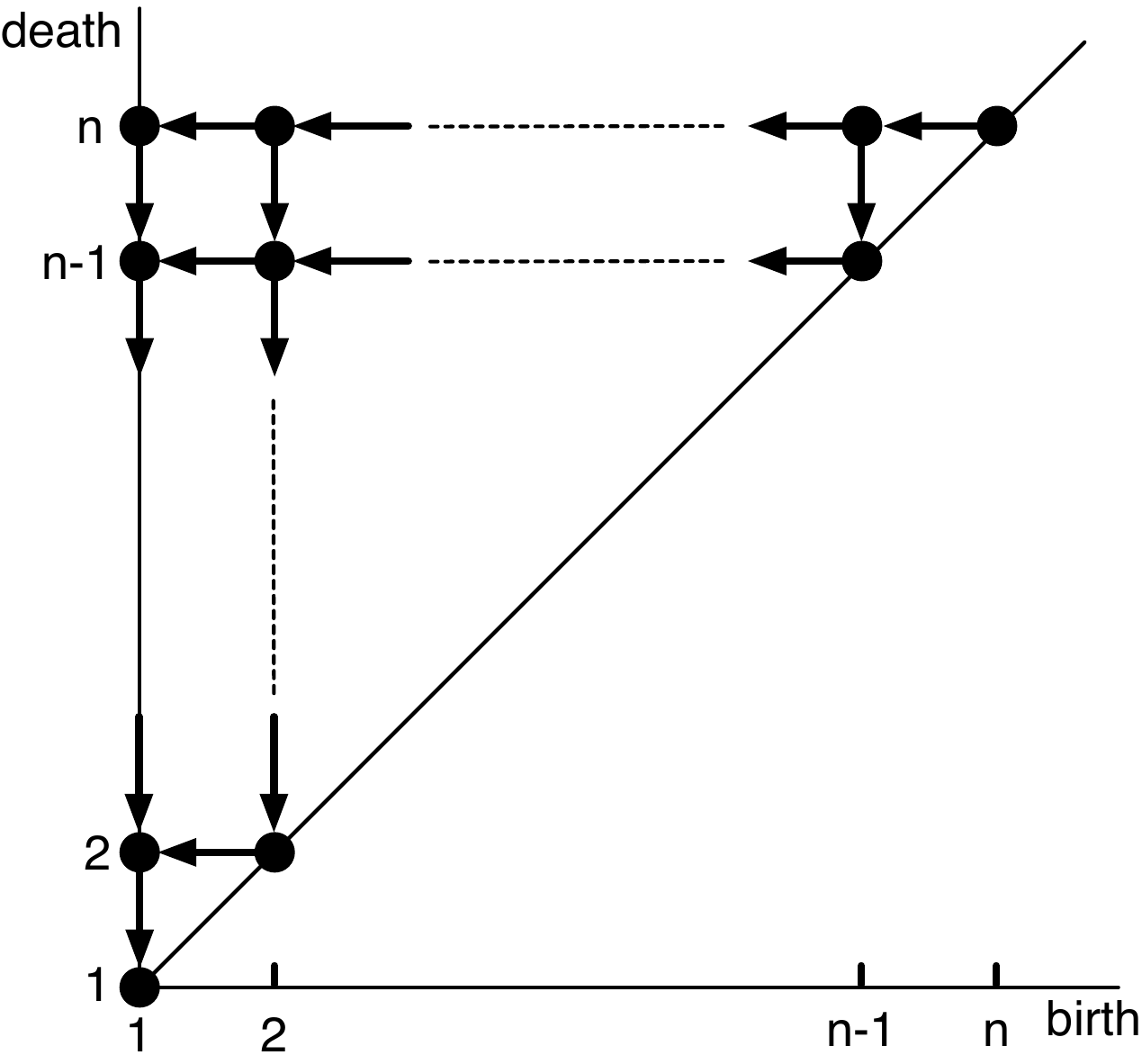}
\caption{The Auslander-Reiten quiver of $\Aright_n$ and the grid points of a persistence diagram.}
\label{fig:ar_an}
\end{center}
\end{figure}
\end{theorem}
\begin{proof}
The indecomposable projectives and injectives of this quiver are 
\[
	P(k)=\I[k,n]~~{\rm and}~~I(k)=\I[1,k]
\]
for $k=1,\dots,n$. From (i) in the recipe we start to construct the Auslander-Reiten quiver by including the indecomposable projectives $P(k)$. Since $P(n)$ is simple projective, there are no irreducible morphisms ending at $P(n)$. We also note $\rad(P(k))=P(k+1)$. This implies from (ii) in the recipe that $P(k+1)\rightarrow P(k)$ is the right minimal almost split morphism at each $k=1,\dots,n-1$. Hence, we obtain a sequence of irreducible  morphisms
\[
	\xymatrix{
%		\I[1,n]&\I[2,n]\ar[l]&\ar[l]&\ar@{--}[l]&\ar[l]\I[n-1,n]&\ar[l]\I[n,n]
		\I[1,n]&\I[2,n]\ar[l]&\ar[l]\dots\dots&\ar[l]\I[n-1,n]&\ar[l]\I[n,n].		
	}
\]
%ending at indecomposable projectives. 

Our proof is by induction on the rows of Figure \ref{fig:ar_an}, and for every row another induction on the columns. 
Let us suppose that there exists a sequence of irreducible morphisms
\[
	\xymatrix{
		\I[1,j]&\I[2,j]\ar[l]&\ar[l]\dots\dots&\ar[l]\I[j-1,j]&\ar[l]\I[j,j]
	}
\]
among the indecomposables $\I[1,j],\dots,\I[j,j]$ for some $j=2,\dots,n$, and execute (iii) in the recipe. We first construct the almost split sequence
\[
	\xymatrix{
		0\ar[r]&\I[j,j]\ar[r]&\I[j-1,j]\oplus M\ar[r]&\tau^{-1}\I[j,j]\ar[r]&0
	}
\]
starting at the indecomposable non-injective $\I[j,j]$ with some possible direct summand $M$. For this purpose, let us study $\tau^{-1}\I[j,j]$ by Proposition \ref{proposition:nakayama}. 

A minimal injective presentation of $\I[j,j]$ is given by
\[
	\xymatrix{
		0\ar[r]&\I[j,j]\ar[r]&I(j)\ar[r]&I(j-1).
	}
\]
By applying the Nakayama functor $\nu^{-1}$ and noting that $\nu^{-1}I(k)=P(k)$ for $k=1,\dots,n$, we obtain the exact sequence
\[
	\xymatrix{
		0\ar[r]&\nu^{-1}\I[j,j]\ar[r]&P(j)\ar[r]&P(j-1)\ar[r]&\tau^{-1}\I[j,j]\ar[r]&0.
	}
\]
Since $\nu^{-1}L\simeq\Hom_A(I(1),L)\oplus\dots\oplus \Hom_A(I(n),L)$ for any $L\in\rep(A)$,
we have  $\nu^{-1}\I[j,j]=0$. Then, we compute
\[
	\dim\tau^{-1}\I[j,j]=\dim P(j-1)-\dim P(j)=\dim \I[j-1,j-1],
\]
and hence $\tau^{-1}\I[j,j]= \I[j-1,j-1]$.
Furthermore, $M=0$, since $\dim M=\dim \I[j,j]+\dim \I[j-1,j-1]-\dim \I[j-1,j]=0$.
Hence, the almost split sequence starting at $\I[j,j]$ is given by
\[
	\xymatrix{
		0\ar[r]&\I[j,j]\ar[r]&\I[j-1,j]\ar[r]&\I[j-1,j-1]\ar[r]&0
	}
\]
and we obtain the new indecomposable $\I[j-1,j-1]$ and 
the new irreducible morphism $\I[j-1,j]\rightarrow \I[j-1,j-1]$.

Next, let us similarly construct the almost split sequence starting at the indecomposable non-injective $\I[i,j]$ with $i=2,\dots,j-1$.
By induction on $i$, we have irreducible morphisms $\I[i,j]\rightarrow \I[i-1,j]$ and $\I[i,j]\rightarrow \I[i,j-1]$, leading to
\[
	\xymatrix{
		0\ar[r]&\I[i,j]\ar[r]&\I[i-1,j]\oplus \I[i,j-1]\oplus M\ar[r]&\tau^{-1}\I[i,j]\ar[r]&0
	}
\]
as the almost split sequence.
A minimal injective presentation 
\[
	\xymatrix{
		0\ar[r]&\I[i,j]\ar[r]&I(j)\ar[r]&I(i-1)
	}
\]
leads to an exact sequence
\[
	\xymatrix{
		0\ar[r]&\nu^{-1}\I[i,j]\ar[r]&P(j)\ar[r]&P(i-1)\ar[r]&\tau^{-1}\I[i,j]\ar[r]&0.
	}
\]
From this exact sequence, we obtain $\tau^{-1}\I[i,j]=\I[i-1,j-1]$, $M=0$, and the almost split sequence
\[
	\xymatrix{
		0\ar[r]&\I[i,j]\ar[r]&\I[i-1,j]\oplus \I[i,j-1]\ar[r]&\I[i-1,j-1]\ar[r]&0
	}
\]
in a similar way. For the left minimal almost split morphism starting at the indecomposable injective $\I[1,j]$, 
we note that the irreducible morphism $\I[1,j]\rightarrow \I[1,j-1]$ obtained in the previous step at $i=2$ is an epimorphism with the kernel $S(j)$. Hence, it follows from (ii) in the recipe that there are no irreducible morphisms starting at $\I[1,j]$ to other indecomposables.

By induction on $j$, we obtain the isomorphism classes of  indecomposables $\I[i,j]$ for $1\leq i\leq j\leq n$. 
Since $\I[1,1]$ is simple injective, there are no irreducible morphisms starting at $\I[1,1]$, and these $n(n+1)/2$ indecomposables and their irreducible morphisms form a connected component. This completes the proof according to (iv) in the recipe.
\qed\end{proof}
%herehere
%
%\ref{fig:ar_an}\ref{theorem:ar_an}
%
\begin{remark}\label{remark:ar_an}
{\rm 
%The Auslander-Reiten quiver obtained in Theorem \ref{theorem:ar_an} is drawn in Figure \ref{fig:ar_an}.
%%
%\begin{figure}[h!]
%\begin{center}
%\includegraphics[width=6cm]{./ar_an}
%\caption{The Auslander-Reiten quiver of $\Aright_n$ and the grid points of a persistence diagram.}
%\label{fig:ar_an}
%\end{center}
%\end{figure}
%
The vertex located at $(b,d)\in \N^2, 1\leq b\leq d\leq n,$ corresponds to the isomorphism class of the interval $\I[b,d]$. 
Hence, $\Gamma_0$ of the Auslander-Reiten quiver is the same as the grid points 
${\{(b,d)\in \N^2\mid 1\leq b\leq d\leq n\}}$ of standard persistence diagrams on $\Aright_n$. This relationship between persistence diagrams and the Auslander-Reiten quiver is generalized in Section \ref{sec:generalization_pd}.
%For arbitrary $\A_n$ quiver
%\[
%\xymatrix{
%\circ_1 \ar@{<->}[r]&\circ_2 \ar@{<->}[r]&\circ_3\ar@{<->}[r]&\ar@{..}[r]&\ar@{<->}[r]&\circ_n,
%}
%\]
%the orientations of arrows in Figure \ref{fig:ar_an} are suitably changed.
}
\end{remark}

\subsection{Proofs of Theorems \ref{theorem:cl} and \ref{theorem:cl_infinite}}\label{sec:proof}
\noindent{{\it Proof of Theorem \ref{theorem:cl}}}:
We give a proof of the theorem for the commutative ladder $\CL(fb)$:
\[
\includegraphics{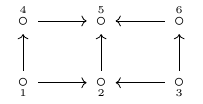}\raisebox{+20pt}{.}
\]
The proofs for the other $\CL(\tau_n)$ with $n\leq 4$ are similar.
In this proof, we successively construct indecomposables and irreducible morphisms starting from the left of the Auslander-Reiten quiver $\Gamma(\CL(fb))$ and ending on the right. 
For this purpose, we label the columns from 1 to 12 as shown in Figure \ref{fig:AR_fb_column}
\begin{figure}[h]
\begin{center}
\includegraphics[width=14cm]{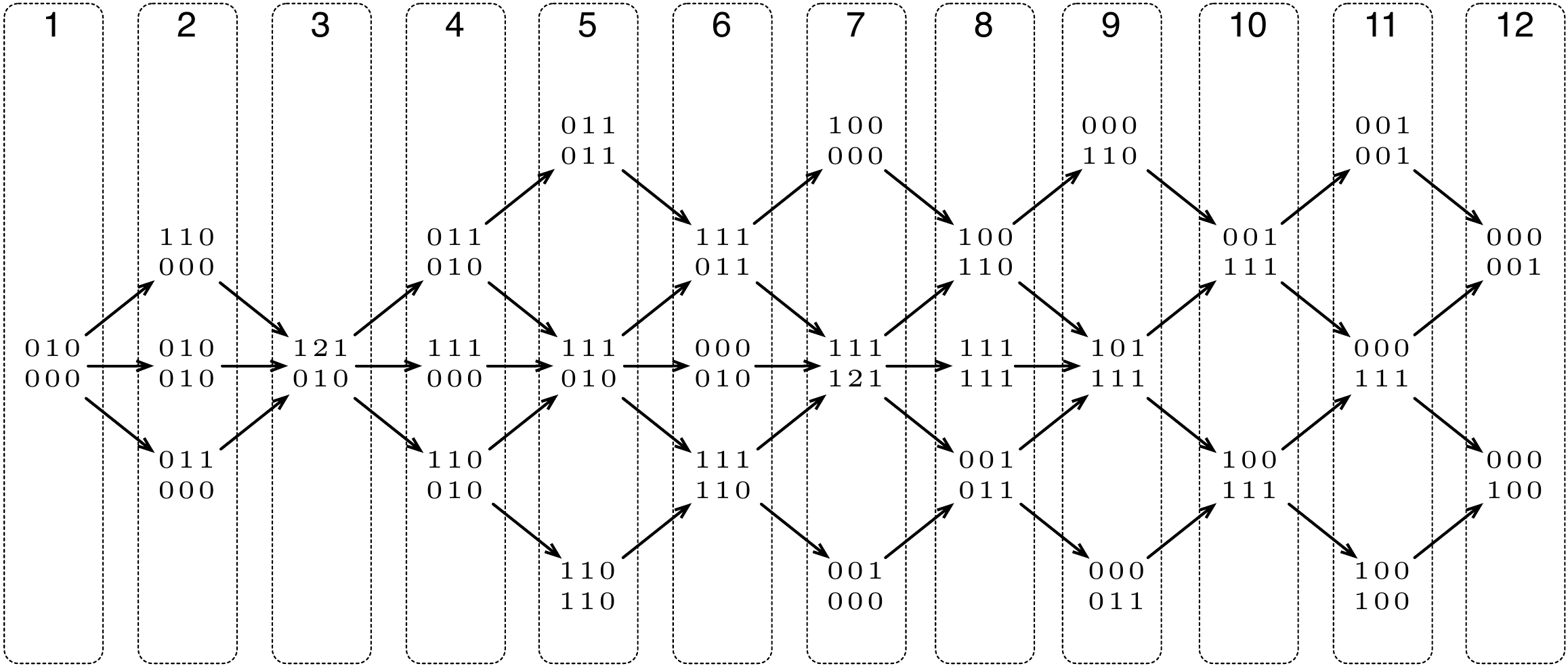}
\caption{The Auslander-Reiten quiver of $\CL(fb)$ with columns labeled from $1$ to $12$.}
\label{fig:AR_fb_column}
\end{center}
\end{figure}

For the sake of notational simplicity, we denote each indecomposable by its dimension vector without using the symbol $``\dim"$. In addition, we take all linear maps $K\rightarrow K$ as the identity map. 

The indecomposable projectives and injectives are
\[
%\begin{array}{llllll}
\begin{array}{cccccc}
P(1)=\smdimVec{1,1,0,1,1,0}, 
&
P(2)=\smdimVec{0,1,0,0,1,0}, 
&
P(3)=\smdimVec{0,1,1,0,1,1}, 
&
P(4)=\smdimVec{0,0,0,1,1,0}, 
&
P(5)=\smdimVec{0,0,0,0,1,0}, 
&
P(6)=\smdimVec{0,0,0,0,1,1}, 
\\
I(1)=\smdimVec{1,}, 
&
I(2)=\smdimVec{1,1,1,},
&
I(3)=\smdimVec{0,0,1,}, 
&
I(4)=\smdimVec{1,0,0,1,0,0}, 
&
I(5)=\smdimVec{1,1,1,1,1,1}, 
&
I(6)=\smdimVec{0,0,1,0,0,1}.
\end{array}
\]
We start to construct $\Gamma(\CL(fb))$ from the simple projective $P(5)$. It follows from $\rad P(2)=\rad P(4) = \rad P(6) = P(5)$ that the almost split sequence starting at $P(5)$ takes the form
\begin{equation}\label{eq:ass_starting_at_p5}
	\xymatrix{
		0\ar[r] & P(5) \ar[r] & P(2)\oplus P(4) \oplus P(6)\oplus M \ar[r] & N \ar[r] & 0
	}
\end{equation}
with some possible direct summand $M$, where $N=\tau^{-1}P(5)$. A minimal injective presentation 
\[
	\xymatrix{
		0\ar[r] & P(5) \ar[r] & I(5) \ar[r]^-{\varphi} & I(2)\oplus I(4)\oplus I(6)
	}
\]
leads to an exact sequence
\begin{equation}\label{eq:p5mpp}
	\xymatrix{
		0\ar[r] &\nu^{-1}P(5)\ar[r] & P(5)\ar[r]^-{\nu^{-1}\varphi} & P(2)\oplus P(4)\oplus P(6)\ar[r] &N\ar[r] & 0.
	}
\end{equation}
It follows from $\nu^{-1}P(5)\simeq\Hom_A(I(1),P(5))\oplus \dots\oplus\Hom_A(I(6),P(5)) = 0$ that we obtain the vector spaces of $N$ 
as
\[
\xymatrix{
K \ar[r]^-{f_{45}}&K^2&\ar[l]_-{f_{65}}K\\
0 \ar[r]\ar[u]&K\ar[u]^-{f_{25}} &\ar[l]0\ar[u]
}
\]
by counting dimensions in (\ref{eq:p5mpp}). In order to determine the linear maps $f_{25}, f_{45}$,  and $f_{65}$, we need to study the cokernel of $\nu^{-1}\varphi$. 
Note that, at the vertex 5 in $\CL(fb)$, we have
\[
	P(5)e_5\simeq\Hom_A(DA,I(5))e_5\simeq\Hom_A(I(5),I(5))
\]
and similarly
\[
	(P(2)\oplus P(4)\oplus P(6))e_5\simeq\Hom_A(I(5),I(2))\oplus\Hom_A(I(5),I(4))\oplus\Hom_A(I(5),I(6)).
\]
Then, by appropriately choosing bases, we can obtain a matrix form of the linear map
\[
P(5)e_5\simeq K\ni x\longmapsto (-x~x~x)^t\in
K^3\simeq (P(2)\oplus P(4)\oplus P(6))e_5.
\]
This leads to the linear maps
\[
f_{25}=\left[\begin{array}{c}1 \\1\end{array}\right],~
f_{45}=\left[\begin{array}{c}1\\0\end{array}\right],~
f_{65}=\left[\begin{array}{c}0\\1\end{array}\right]
\]
in $N=\cokernel \nu^{-1}\varphi$.
Furthermore, we also obtain  $M=0$ by counting dimensions in (\ref{eq:ass_starting_at_p5}), and hence the almost split sequence starting from $P(5)$ is determined. This gives a portion of the Auslander-Reiten quiver up to the third column.
It also follows from Proposition \ref{proposition:mal_pro_inj} that we do not have other irreducible morphisms ending at $P(2), P(4)$, or $P(6)$.

In a similar way, we can construct the almost split sequences 
\begin{eqnarray*}
&& 
0\rightarrow
\smdimVec{0,0,0,1,1,0} \rightarrow
\smdimVec{0,1,0,1,2,1} \rightarrow 
\smdimVec{0,1,0,0,1,1} \rightarrow 
0,
\\
&&
0\rightarrow
\smdimVec{0,1,0,0,1,0} \rightarrow
\smdimVec{0,1,0,1,2,1} \rightarrow 
\smdimVec{0,0,0,1,1,1} \rightarrow 
0,
\\
&&
0 \rightarrow
\smdimVec{0,0,0,0,1,1} \rightarrow
\smdimVec{0,1,0,1,2,1} \rightarrow 
\smdimVec{0,1,0,1,1,0} \rightarrow 
0
\end{eqnarray*}
starting from the 2nd column, and 
\[
0 \rightarrow
\smdimVec{0,1,0,1,2,1} \rightarrow
\smdimVec{0,1,0,0,1,1} \oplus
\smdimVec{0,0,0,1,1,1} \oplus
\smdimVec{0,1,0,1,1,0} \rightarrow 
\smdimVec{0,1,0,1,1,1} \rightarrow 
0
\]
starting from the 3rd column.

At the 4th column, by noting 
$\rad P(3)=\smdimVec{0,1,0,0,1,1}$ and 
$\rad P(1)=\smdimVec{0,1,0,1,1,0}$ and by adding $P(3)$ and $P(1)$ in the middle terms of the sequences, we obtain the almost split sequences
\begin{eqnarray*}
&& 
0 \rightarrow
\smdimVec{0,1,0,0,1,1} \rightarrow
\smdimVec{0,1,1,0,1,1} \oplus
\smdimVec{0,1,0,1,1,1} \rightarrow 
\smdimVec{0,1,1,1,1,1} \rightarrow 
0,
\\
&&
0 \rightarrow
\smdimVec{0,0,0,1,1,1} \rightarrow
\smdimVec{0,1,0,1,1,1} \rightarrow 
\smdimVec{0,1,0,0,0,0} \rightarrow 
0,\\
&&
0 \rightarrow
\smdimVec{0,1,0,1,1,0} \rightarrow
\smdimVec{0,1,0,1,1,1} \oplus
\smdimVec{1,1,0,1,1,0} \rightarrow 
\smdimVec{1,1,0,1,1,1} \rightarrow 
0
\end{eqnarray*}
in a similar way.

The almost split sequences starting at the 5th, 6th, and 7th columns are similarly and inductively derived as follows:
\begin{flushleft}
$\bullet$ 5th column: 
\end{flushleft}
\begin{eqnarray*}
&& 
0 \rightarrow
\smdimVec{0,1,1,0,1,1} \rightarrow
\smdimVec{0,1,1,1,1,1} \rightarrow 
\smdimVec{0,0,0,1,0,0} \rightarrow 0,
\\
&&
0 \rightarrow
\smdimVec{0,1,0,1,1,1} \rightarrow
\smdimVec{0,1,1,1,1,1} \oplus
\smdimVec{0,1,0,} \oplus
\smdimVec{1,1,0,1,1,1} \rightarrow 
\smdimVec{1,2,1,1,1,1} \rightarrow 0,\\
&&
0 \rightarrow
\smdimVec{1,1,0,1,1,0} \rightarrow
\smdimVec{1,1,0,1,1,1} \rightarrow 
\smdimVec{0,0,0,0,0,1} \rightarrow 0.
\end{eqnarray*}

\begin{flushleft}
$\bullet$ 6th column
\end{flushleft}
\begin{eqnarray*}
&& 
0 \rightarrow
\smdimVec{0,1,1,1,1,1} \rightarrow
\smdimVec{0,0,0,1,} \oplus
\smdimVec{1,2,1,1,1,1} \rightarrow 
\smdimVec{1,1,0,1,0,0,} \rightarrow 0,
\\
&&0\rightarrow
\smdimVec{0,1,0} \rightarrow
\smdimVec{1,2,1,1,1,1} \rightarrow 
\smdimVec{1,1,1,1,1,1} \rightarrow 0,\\
&&
0 \rightarrow
\smdimVec{1,1,0,1,1,1} \rightarrow
\smdimVec{1,2,1,1,1,1} \oplus
\smdimVec{0,0,0,0,0,1} \rightarrow 
\smdimVec{0,1,1,0,0,1} \rightarrow 0.
\end{eqnarray*}
\begin{flushleft}
$\bullet$ 7th column
\end{flushleft}
\begin{eqnarray*}
&& 
0 \rightarrow
\smdimVec{0,0,0,1,0,0} \rightarrow
\smdimVec{1,1,0,1,0,0}\rightarrow 
\smdimVec{1,1,0} \rightarrow 0,
\\
&&0\rightarrow
\smdimVec{1,2,1,1,1,1} \rightarrow
\smdimVec{1,1,0,1,} \oplus
\smdimVec{1,1,1,1,1,1} \oplus
\smdimVec{0,1,1,0,0,1} \rightarrow 
\smdimVec{1,1,1,1,0,1} \rightarrow 0,\\
&&
0 \rightarrow
\smdimVec{0,0,0,0,0,1} \rightarrow
\smdimVec{0,1,1,0,0,1} \rightarrow 
\smdimVec{0,1,1,0,0,0} \rightarrow 0.
\end{eqnarray*}
In the part of the 5th column where the indecomposable 
$\smdimVec{1,2,1,1,1,1}$ appears, the linear maps in it can be obtained by studying the cokernel, as was done for
$\smdimVec{0,1,0,1,2,1}$.

At the 8th column, we note that 
$I(5)=\smdimVec{1,1,1,1,1,1}$ 
is indecomposable injective and 
$\smdimVec{1,1,1,1,1,1}
\rightarrow 
\smdimVec{1,1,1,1,0,1}$ is an epimorphism whose kernel is $S(5)$. Hence, this morphism is the left minimal almost split and there are no other irreducible morphisms staring at $I(5)$ by Proposition \ref{proposition:mal_pro_inj}.
Then, we have the almost split sequences
\begin{eqnarray*}
&& 
0\rightarrow
\smdimVec{1,1,0,1,0,0} \rightarrow
\smdimVec{1,1,} \oplus
\smdimVec{1,1,1,1,0,1} \rightarrow 
\smdimVec{1,1,1,0,0,1} \rightarrow 0,
\\
&& 
0\rightarrow
\smdimVec{0,1,1,0,0,1} \rightarrow
\smdimVec{1,1,1,1,0,1} \oplus
\smdimVec{0,1,1,} \rightarrow 
\smdimVec{1,1,1,1,0,0} \rightarrow 0
\end{eqnarray*}
at this column.

At the 9th column, we obtain three almost split sequences
\begin{eqnarray*}
&&0\rightarrow
\smdimVec{1,1,} \rightarrow
\smdimVec{1,1,1,0,0,1} \rightarrow 
\smdimVec{0,0,1,0,0,1} \rightarrow 0,\\
&&0\rightarrow
\smdimVec{1,1,1,1,0,1} \rightarrow
\smdimVec{1,1,1,0,0,1} \oplus
\smdimVec{1,1,1,1,} \rightarrow 
\smdimVec{1,1,1,} \rightarrow 0,\\
&&0\rightarrow
\smdimVec{0,1,1,} \rightarrow
\smdimVec{1,1,1,1,} \rightarrow 
\smdimVec{1,0,0,1} \rightarrow 0.
\end{eqnarray*}
It should be remarked that we have a nonzero image of the Nakayama functor 
$\nu^{-1}\left(\smdimVec{1,1,1,1,0,1}\right)
=\smdimVec{0,0,0,0,1,0}$
in the construction of the almost split sequence starting at 
$\smdimVec{1,1,1,1,0,1}$.

At the 10th column, by again noting nonzero images of the Nakayama functor
$\nu^{-1}\left(\smdimVec{1,1,1,0,0,1}\right)=
\smdimVec{0,0,0,0,1,0}$ and 
$\nu^{-1}\left(\smdimVec{1,1,1,1}\right)=
\smdimVec{0,0,0,0,1,0}$, 
we obtain the almost split sequences
\begin{eqnarray*}
&&0\rightarrow
\smdimVec{1,1,1,0,0,1} \rightarrow
\smdimVec{0,0,1,0,0,1} \oplus
\smdimVec{1,1,1,0,0,0} \rightarrow 
\smdimVec{0,0,1} \rightarrow 0,\\
&&
0\rightarrow
\smdimVec{1,1,1,1} \rightarrow
\smdimVec{1,1,1,} \oplus
\smdimVec{1,0,0,1} \rightarrow 
\smdimVec{1,} \rightarrow 0.
\end{eqnarray*}

At the 11th column, it follows from Proposition \ref{proposition:mal_pro_inj} that
\begin{eqnarray*}
\smdimVec{0,0,1,0,0,1} \rightarrow
\smdimVec{0,0,1,0,0,0},~~~~~~
\smdimVec{1,1,1} \rightarrow
\smdimVec{0,0,1} \oplus
\smdimVec{1,},~~~~~~
\smdimVec{1,0,0,1,0,0} \rightarrow
\smdimVec{1,} 
\end{eqnarray*}
are the left minimal almost split morphisms, and hence there are no other irreducible morphisms starting at the 11th column. Since 
$\smdimVec{0,0,1}$ and 
$\smdimVec{1,}$ are simple injectives, there are no arrows starting at these injectives by Proposition \ref{proposition:simpleprojective}. 
Hence we obtain a finite connected component, and Proposition \ref{proposition:finitecc} guarantees that this is the Auslander-Reiten quiver of $\CL(fb)$. This completes the proof for $\CL(fb)$. 

The Auslander-Reiten quivers for the other commutative ladders in Theorem \ref{theorem:cl} are similarly constructed with slight modifications. We omit their proofs. 
\qed \vspace{0.5cm}

\noindent{{\it Proof of Theorem \ref{theorem:cl_infinite}}}:
Suppose that $O_i=f$ for $i=1,\dots,n-1$ in the orientation $\tau_n=O_1\dots O_{n-1}$. 
We first claim that the commutative ladder $\CL(\tau_5)$:
\[
\includegraphics{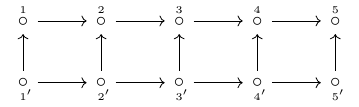}
\]
is representation-infinite.  
A direct proof is to apply the same recipe above and to observe that the Auslander-Reiten quiver $\Gamma(\CL(\tau_5))$ contains isomorphism classes of indecomposables whose dimension vectors have elements greater than 6. This implies that the commutative ladder $\CL(\tau_5)$  is representation-infinite by a theorem of Ovsienko (see \cite{ringel}). 

Another proof for this claim, which fits the induction below better, is to study a quotient algebra of $\CL(\tau_5)$ and use the Happel-Vossieck list (see A.2 in \cite{ringel}) in which some representation-infinite algebras are shown. Let $I=\langle e_{1'}, e_{5}\rangle$ be the ideal in $\CL(\tau_5)$ generated by the idempotents $e_{1'}$ and $e_{5}$. Then, we can check that the quotient algebra $\CL(\tau_5)/I$ appears in the list under $\tilde{\E}_7$. Then, the claim follows from the embedding 
\[
	\rep(\CL(\tau_5)/I)\hookrightarrow \rep(\CL(\tau_5)),
\]
since the left hand side is representation-infinite.

Suppose that $\CL(\tau_{n-1})$ is representation-infinite and consider $\CL(\tau_n)$:
\begin{equation}
\includegraphics{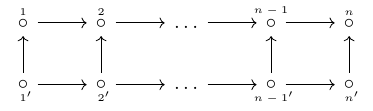}\raisebox{+20pt}{.}
\end{equation}
We have an isomorphism $\CL(\tau_{n-1})\simeq\CL(\tau_n)/I$, where $I=\langle e_n, e_{n'}\rangle$. Since we have an embedding $\rep(\CL(\tau_{n-1}))\hookrightarrow \rep(\CL(\tau_n))$, $\CL(\tau_n)$ is also representation-infinite by the inductive assumption. This completes the proof for the case of the orientation $\tau_n = f\hdots f$. The cases for the other orientations can be proved in the same way.
\qed
\subsection{Persistence Diagrams and Auslander-Reiten Quivers}\label{sec:generalization_pd}
Theorem \ref{theorem:cl} states that the commutative ladders of length up to $4$ are representation-finite, and answers the question (Q1) raised in Section \ref{sec:pm_an}. It shows that we can deal with persistence modules on these commutative ladders in a similar manner to standard persistence modules on $\A_n$ quivers. In particular, given a persistence module $M$ on a commutative ladder $\CL$ of finite type, there is a unique decomposition
\begin{equation}\label{eq:decomposition}
	M\simeq \bigoplus_{[I]\in \Gamma_0} I^{k_{[I]}},~~~~k_{[I]}\in \N_0=\{0,1,2,\dots\}
\end{equation}
by indecomposables specified by the vertices of its Auslander-Reiten quiver ${\Gamma=(\Gamma_0,\Gamma_1)}$. Hence, these indecomposables are the counterparts to the intervals in standard persistence modules on $\A_n$ quivers. 

In Section \ref{sec:ar_an}, we observe that the vertex set of the Auslander-Reiten quiver and the grid points $\{(b,d)\in \N^2\mid 1\leq b\leq d\leq n\}$ of persistence diagrams on the quiver $\Aright_n$ coincide. This is a natural consequence, since the interval representations are precisely the indecomposables of $\Aright_n$  by Gabriel's theorem \cite{gabriel}. In view of this correspondence, it is natural to define generalized persistence diagrams of persistence modules on commutative ladders of finite type as follows. 

\begin{definition}\label{definition:pd}{\rm 
Let $M$ be a persistence module on a commutative ladder of finite type, and let $\Gamma=(\Gamma_0,\Gamma_1)$ be its Auslander-Reiten quiver. The persistence diagram $D_M$ of $M$ is the function defined by
\[
	D_M: \Gamma_0\ni [I] \mapsto k_{[I]}\in \N_0,
\]
where $k_{[I]}$ is given by the indecomposable decomposition (\ref{eq:decomposition}).
}\end{definition}

As usual, we can equivalently regard this function as a multiset of vertices from $\Gamma_0$. We use both expressions of persistence diagrams in this paper.
It is obvious that the standard definition of persistence diagrams (Definition \ref{definition:pd_an}) coincides with the above definition applied to $\A_n$ quivers. These arguments answer the question (Q2) raised in Section \ref{sec:pm_an}. 

Definition \ref{definition:pd} uses the vertex set $\Gamma_0$ of the Auslander-Reiten quiver $\Gamma=(\Gamma_0,\Gamma_1)$. On the other hand, we emphasize that the Auslander-Reiten quiver  also contains information in $\Gamma_1$, the set of irreducible morphisms. 
Thus, let us consider some interpretations of $\Gamma_1$ particularly by focusing on undirected and directed paths in $\Gamma$ .

Let us first consider the relationship between the bottleneck distance defined on persistence modules on $\Aright_n$ and the underlying undirected graph of $\Gamma(\Aright_n)$.
Recall that the bottleneck distance between persistence modules $M,N$ on $\Aright_n$ is defined as
\begin{equation}\label{eq:bd}
	d_{\rm B}(M,N)=\inf_\varphi\sup_{p\in \overline{D_M}} ||p-\varphi(p)||.
\end{equation}
Here $\overline{D_M}$ and $\overline{D_N}$ are the extended persistence diagrams of $M$ and $N$ obtained by adding the set of diagonal points $(i,i-1)$, $i=1,\dots,n$, with infinite multiplicity to the original persistence diagrams. (recall the remark after Definition \ref{definition:pd_an} concerning the diagonal in the representation setting).
In addition, we take the infimum over all bijections $\varphi:\overline{D_M}\rightarrow\overline{D_N}$.
%In addition, $\varphi:\overline{D_M}\rightarrow\overline{D_N}$ is a bijection on these extended persistence diagrams, and $\varphi$ ranges over all these bijections. 
The distance $||p-\varphi(p)||$ is measured by a norm on $\R^2$, and we usually use $||x||_\infty=\max\{|x_1|,|x_2|\}$. This distance plays an important role in the stability theorems \cite{proximity,stability}. 

By the equivalence between the Auslander-Reiten quiver $\Gamma(\Aright_n)$ and the grid points of standard persistence diagrams, the minimum length of paths in the underlying undirected graph of $\Gamma(\Aright_n)$ between two intervals $p_1=(b_1,d_1), p_2=(b_2,d_2)\in \Gamma(\Aright_n)_0$ is equal to the $\ell_1$-distance $||p_1-p_2||_{1}=|b_1-b_2|+|d_1-d_2|$ between them. 
In view of this context, let us define the extended Auslander-Reiten quiver $\overline{\Gamma(\Aright_n)}$ by adding $n$ vertices corresponding to the zero representations $Z(i)$ to $\Gamma(\Aright_n)_0$,  and $n$ arrows $S(i)=\I[i,i]\rightarrow Z(i)$, $i=1,\dots,n$ from the simple representations $S(i)$ to $\Gamma(\Aright_n)_1$. 
Then, the $\ell_1$-bottleneck distance can be equivalently defined on 
%the extended Auslander-Reiten quiver 
$\overline{\Gamma(\Aright_n)}$.

From this extension, it is reasonable to define the bottleneck distance even in a commutative ladder $\CL$ of finite type by just replacing the quiver $\Aright_n$ with $\CL$. Namely, the distance between two vertices $[I],[J]\in\overline{\Gamma(\CL)}_0$ in the extended Auslander-Reiten quiver $\overline{\Gamma(\CL)}$
is defined to be the minimum length of undirected paths between them in $\overline{\Gamma(\CL)}$. 
Here, $\overline{\Gamma(\CL)}=(\overline{\Gamma(\CL)}_0,\overline{\Gamma(\CL)}_1)$ is defined by adding vertices corresponding to the zero representation $Z(i)$ for each simple representation $S(i)$ to $\Gamma(\CL)_0$ and arrows $S(i)\rightarrow Z(i)$ to $\Gamma(\CL)_1$. 
Then, the bottleneck distance between two persistence diagrams on a commutative ladder $\CL$ can be defined by the same formula (\ref{eq:bd}). It is interesting to study the possibility of a stability theorem even in the commutative ladder case by using this definition.

Next, we study what information we can obtain from the directions of arrows in Auslander-Reiten quivers. We note that there exists a morphism from $M$ to $N$, both of which are indecomposables, in $\rep(A)$ if and only if there exists a directed path from $[M]$ to $[N]$ in the Auslander-Reiten quiver $\Gamma(A)=(\Gamma_0,\Gamma_1)$. This claim follows from the definition of arrows in $\Gamma_1$ and the irreducibility of morphisms (Definition \ref{definition:irreducible}). 

Accordingly, let us consider the case of persistence modules on $\Aright_n$. Theorem \ref{theorem:ar_an} gives the arrows from an interval $\I[i,j]$, and one of the irreducible morphisms $\I[i,j]\rightarrow \I[i,j-1]$ is described by
%
%\[
%\xymatrix{
%\Aright_n&	\dots\ar[r]&\circ_i\ar[r] & \dots \ar[r]&\circ_{j-1}\ar[r]&\circ_j\ar[r]&\cdots\\
%\I[i,j]\ar[d]&	\dots\ar[r]^0&K\ar[r]^1\ar[d]_c & \dots \ar[r]^1&K\ar[r]^1\ar[d]_c&K\ar[r]^0\ar[d]_0&\cdots\\
%\I[i,j-1]&	\dots\ar[r]^0&K\ar[r]^1 & \dots \ar[r]^1&K\ar[r]^0&0\ar[r]&\cdots&
%}
%\]
\[
\includegraphics{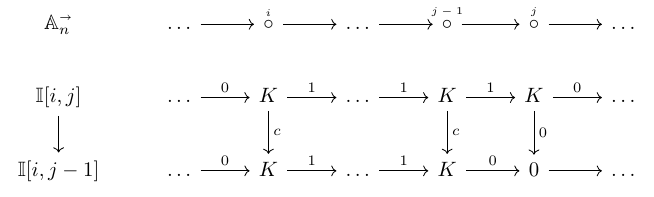}
\]
for any nonzero $c\in K$. 
The existence of the arrow $\I[i,j]\rightarrow \I[i,j-1]$ in $\Gamma(\Aright_n)$ means that the interval $\I[i,j]$ can be nontrivially mapped to $\I[i,j-1]$ in $\rep(\Aright_n)$. 

On the other hand, a morphism $\I[i,j-1]\rightarrow \I[i,j]$ in the opposite direction
\[
\includegraphics{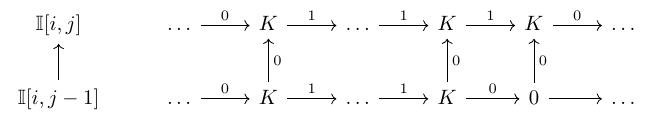}
\]
must be the zero map by the commutativity. Thus, the nonexistence of an arrow ${\I[i,j-1]}\rightarrow {\I[i,j]}$ in $\Gamma(\Aright_n)$ means that the interval $\I[i,j-1]$ cannot be nontrivially mapped to $\I[i,j]$ in $\rep(\Aright_n)$. 
Figure \ref{fig:topological_constraints} shows examples of diagrams of topological spaces that induce these algebraic constraints (only drawing at the vertices $j-1$ and $j$).
\begin{figure}[h!]
\begin{center}
\includegraphics[width=6cm]{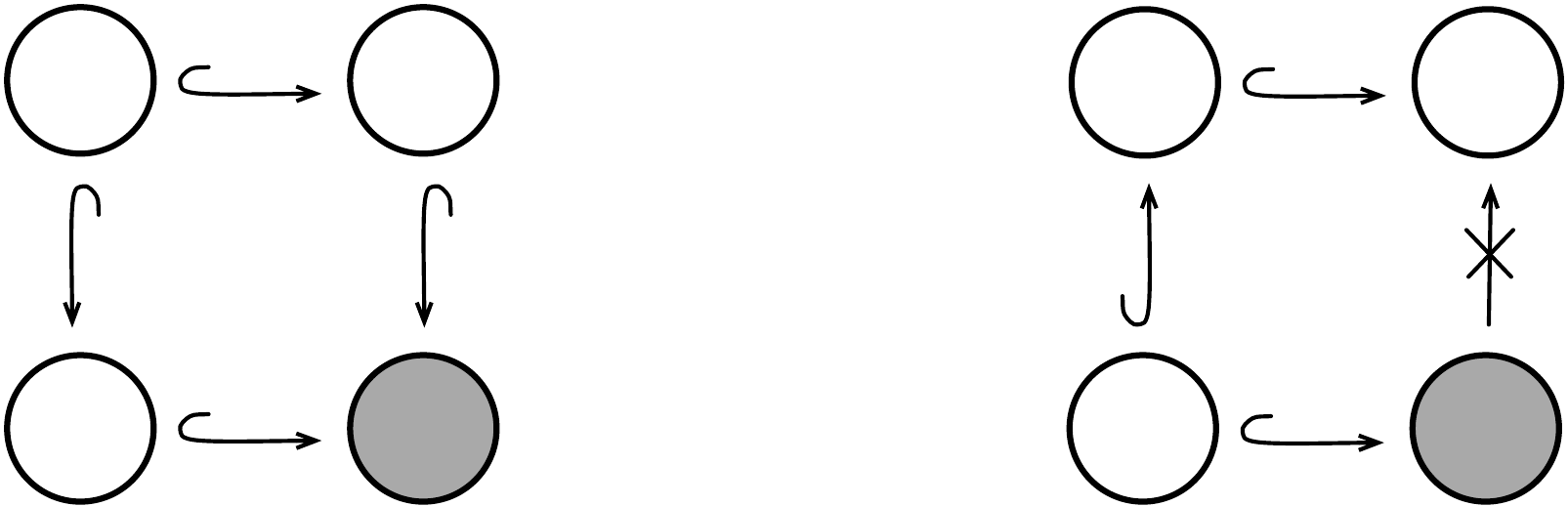}
\caption{Diagrams of topological spaces (spheres and discs) inducing algebraic constraints by taking  homology $H_1(-)$. 
The left (resp., right) diagram corresponds to $\I[j-1,j]\rightarrow \I[j-1,j-1]$ (resp., $\I[j-1,j-1]\rightarrow \I[j-1,j]$). 
On the right diagram, we cannot assign a continuous map satisfying the commutativity.}
\label{fig:topological_constraints}
\end{center}
\end{figure}

We remark that the above arguments can be applied to persistence modules on commutative ladders of finite type.
To summarize, the direction of arrows in the Auslander-Reiten quiver in general provides algebraic constraints to maps between indecomposables. This implies that indecomposables are not symmetric in general. However, it should be remarked that the bottleneck distance deals with indecomposables symmetrically (by definition of the distance). This new perspective from the Auslander-Reiten quiver may give further useful information for topological data analysis.

\subsection{Further Decomposition of Interval Representations
%Auslander-Reiten Quiver of $\CL(fb)$
}\label{sec:arq_clfb}
For applications of persistence modules on a commutative ladder $\CL(\tau_n)$ of finite type, we need to interpret the indecomposables in the context of input data. In this subsection, we discuss this issue for the orientation $\tau_2=fb$ as an example, which is the quiver relevant to our original motivation of TDA on materials science. We show that the Auslander-Reiten quiver $\Gamma(\CL(fb))=(\Gamma_0,\Gamma_1)$ serves as a suitable visualization tool to capture the roles of the indecomposables and the relationships among them. In Section \ref{sec:numerics}, we provide numerical experiments and illustrate the contents of this section through those examples.

A persistence module 
\begin{equation}\label{eq:pm_cl}
M= \spaceInd{H_\ell(X_r), H_\ell(X_r \cup Y_r), H_\ell(Y_r),
    H_\ell(X_s), H_\ell(X_s\cup Y_s), H_\ell(Y_s)}
\end{equation}
on $\CL(fb)$ arises from a diagram of topological spaces related by inclusions:
\[
\spaceInd{X_r, X_r \cup Y_r, Y_r,
    X_s, X_s \cup Y_s, Y_s}.
\]
For example, the pairs of topological spaces $X_r\subset X_s$ and $Y_r\subset Y_s$ can be constructed from two point cloud data by some geometric model such as the Vietoris-Rips \cite{eh} or the alpha complex \cite{alpha} model with two parameters $r<s$. 

The diagram \eqref{eq:pm_cl} contains both persistence (in the vertical direction) and zigzag persistence (in the horizontal direction). 
Recall that the vertical direction measures the robustness of topological features in $X_r\hookrightarrow X_s$ and $Y_r\hookrightarrow Y_s$, and the horizontal direction detects common topological features.
%
%The vertical arrows 
%\begin{eqnarray*}
%H_\ell(X_r)\rightarrow H_\ell(X_s)&\simeq& \I[r,r]^{m_r} \oplus \I[s,s]^{m_s}\oplus \I[r,s]^{m_{rs}},\\
%H_\ell(Y_r)\rightarrow H_\ell(Y_s)&\simeq& \I[r,r]^{n_r} \oplus \I[s,s]^{n_s}\oplus \I[r,s]^{n_{rs}}
%\end{eqnarray*}
%express the 2-step persistence modules on $\Aright_2$ with decompositions using the interval representations 
%\[
%	\I[r,r]=K\rightarrow 0,~~~\I[s,s]=0\rightarrow K,~~~\I[r,s] = K\rightarrow K.
%\]
%This implies that they measure the robustness of topological features in $X_r\hookrightarrow X_s$ and $Y_r\hookrightarrow Y_s$, respectively.
%On the other hand, the horizontal arrows
%\begin{eqnarray*}
%H_\ell(X_s)\rightarrow H_\ell(X_s\cup Y_s)\leftarrow H_\ell(Y_s)&\simeq& \I[1,1]^{s_{11}}\oplus \I[2,2]^{s_{22}} \oplus \I[3,3]^{s_{33}}\oplus \I[1,2]^{s_{12}}\oplus \I[2,3]^{s_{23}}\oplus \I[1,3]^{s_{13}},\\
%H_\ell(X_r)\rightarrow H_\ell(X_r\cup Y_r)\leftarrow H_\ell(Y_r)&\simeq& \I[1,1]^{r_{11}}\oplus \I[2,2]^{r_{22}} \oplus \I[3,3]^{r_{33}}\oplus \I[1,2]^{r_{12}}\oplus \I[2,3]^{r_{23}}\oplus \I[1,3]^{r_{13}}
%\end{eqnarray*}
%express zigzag persistence modules and measure common topological features in the pairs $(X_s, Y_s)$ and $(X_r,Y_r)$, respectively. In particular, the direct summand consisting of copies of $\I[1,3]$ detects these common topological features.

Let us then interpret these data in the context of our generalized persistence diagram. 
We mark out special regions in our persistence diagram. This is in analogy to the way we classify persistent and nonpersistent features by their distances from the diagonal in the classical case.

First of all, we observe that indecomposables containing common topological features appear in the red colored region of the Auslander-Reiten quiver shown in Figure \ref{fig:ar_colored0}.
 \begin{figure}[h]
 \centering\includegraphics{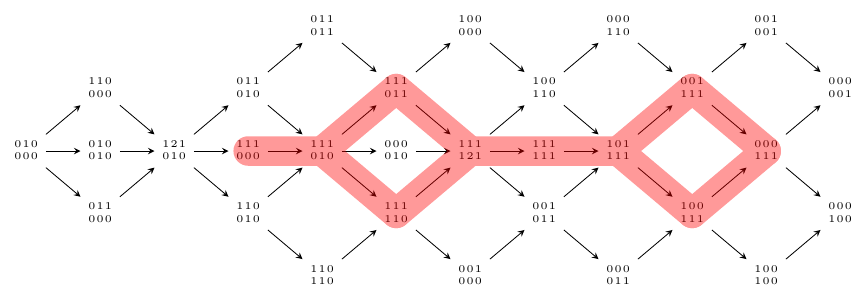}
     \caption{The indecomposables for common topological features.}
    \label{fig:ar_colored0}
\end{figure}
The vertices located in the left (or right) part of the region show common topological features shared between $X_s$ and $Y_s$ ($X_r$ and $Y_r$, resp.). The central vertex $\dimVec{1,1,1,1,1,1}$ shows common and robust topological features, and connects the common topological features at the parameters $r,s$.

Next, we study the 2-step persistence $H_\ell(X_r)\rightarrow H_\ell(X_s)\simeq \I[r,r]^{m_r}\oplus \I[s,s]^{m_s}\oplus \I[r,s]^{m_{rs}}$. Let us denote the sets of vertices colored green, red, and blue in Figure \ref{fig:ar_colored1} by $L_r$, $L_s$, and $L_{rs}$, respectively, and express the indecomposable decomposition of $M$ as
%\begin{equation}
%    \label{eq:decompo1}
\begin{equation}\label{eq:decom_twostep}
	M\simeq
    \bigoplus_{[I] \in L_{r}} I^{k_{[I]}} \oplus
    \bigoplus_{[I] \in L_{s}} I^{k_{[I]}} \oplus
    \bigoplus_{[I] \in L_{rs}} I^{k_{[I]}} \oplus
    \bigoplus_{[I] \in L'} I^{k_{[I]}},
\end{equation}
%\end{equation}
where $L' = \Gamma_0 \setminus \left(L_r \cup L_s \cup L_{rs}\right)$. 
Note that the indecomposables in $L_r, L_s, L_{rs}$ and $L'$ can be characterized as having the dimension vectors $\dimVec{1,*,*,0,*,*}$, $\dimVec{0,*,*,1,*,*}$, $\dimVec{1,*,*,1,*,*}$, and $\dimVec{0,*,*,0,*,*}$ in the left column, respectively. 
\begin{figure}
    \centering
    \centering\includegraphics{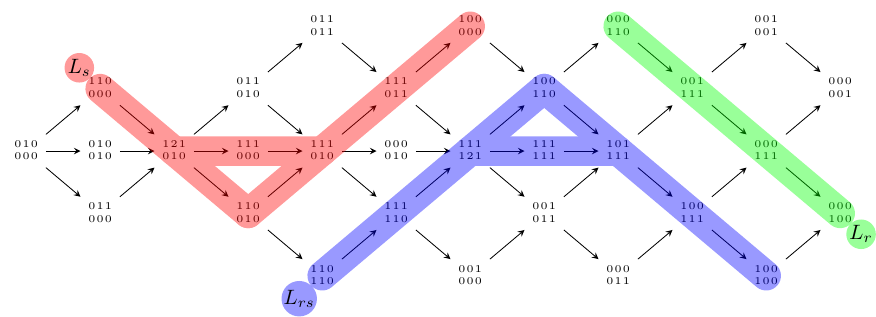}
	\caption{Three subsets $L_r$ (green), $L_s$ (red), and $L_{rs}$ (blue) corresponding to the indecomposable decomposition of the 2-step persistence module $H_\ell(X_r)\rightarrow H_\ell(X_s)$.}
    \label{fig:ar_colored1}
\end{figure}
Then, this decomposition leads to 
\[
H_\ell(X_r)\rightarrow H_\ell(X_s)\simeq \I[r,r]^{\sum_{[I]\in L_r}k_{[I]}}\oplus \I[s,s]^{\sum_{[I]\in L_s}k_{[I]}}\oplus \I[r,s]^{\sum_{[I]\in L_{rs}}k_{[I]}},
\]
and we deduce $m_r=\sum_{[I]\in L_r}k_{[I]}$, $m_s=\sum_{[I]\in L_s}k_{[I]}$, and $m_{rs}=\sum_{[I]\in L_{rs}}k_{[I]}$ by the uniqueness of the indecomposable decomposition. From this correspondence, we can recover the 2-step persistence module $H_\ell(X_r)\rightarrow H_\ell(X_s)$ from (\ref{eq:decom_twostep}).

From the opposite viewpoint, we can regard this correspondence as a further decomposition of the 2-step persistence module into that on $\CL(fb)$. This situation often occurs in TDA when we already know the 2-step persistence module $M_X=H_\ell(X_r)\rightarrow H_\ell(X_s)$ very well and investigate several inputs $\{Y^{(i)}_r\subset Y^{(i)}_s\mid i=1,\dots,a\}$ that are not well understood by using $M_X$ as a reference. 

Of course, we have similar relations for $Y_r\subset Y_s$ by flipping the colored regions in Figure \ref{fig:ar_colored1} with respect to the middle axis.
Furthermore, we can analogously consider the zigzag persistence module $H_\ell(X_s)\rightarrow H_\ell(X_s\cup Y_s)\leftarrow H_\ell(Y_s)$. In Figure \ref{fig:ar_colored2}, we show six subsets of vertices corresponding to the indecomposable decomposition 
\[
H_\ell(X_s)\rightarrow H_\ell(X_s\cup Y_s)\leftarrow H_\ell(Y_s)\simeq \I[1,1]^{s_{11}}\oplus \I[2,2]^{s_{22}} \oplus \I[3,3]^{s_{33}}\oplus \I[1,2]^{s_{12}}\oplus \I[2,3]^{s_{23}}\oplus \I[1,3]^{s_{13}},
\]
where $s_{11}=\sum_{[I]\in U_{100}}k_{[I]}$ and the other $s_{ij}$ are also derived in a similar way.
In this case, we observe that there exist two intersections of the colored regions at $\dimVec{0,1,0,1,2,1}$ and $\dimVec{1,1,1,1,0,1}$. It is clear that these intersections are induced by  
the decomposables
\[
K \rightarrow K^2 \leftarrow K \text{~~and~~} K \rightarrow 0 \leftarrow K
\]
in $H_\ell(X_s)\rightarrow H_\ell(X_s\cup Y_s)\leftarrow H_\ell(Y_s)$.

%Each region corresponds to a direct summand in 
%\[
%K \rightarrow K^2 \leftarrow K \simeq \left(K\rightarrow K \leftarrow 0\right) \oplus \left(0\rightarrow K \leftarrow K\right)
%\]
%and
%\[
%K \rightarrow 0 \leftarrow K \simeq \left(K\rightarrow 0 \leftarrow 0\right) \oplus \left(0\rightarrow 0 \leftarrow K\right),
%\]
%respectively.
\begin{figure}
    \centering\includegraphics{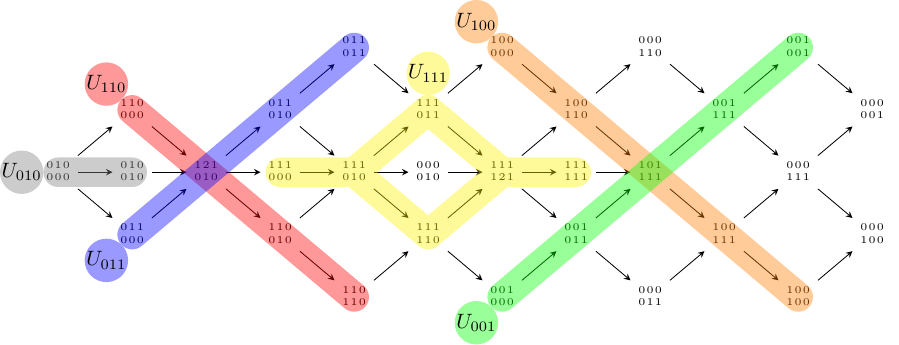}
		\caption{Six subsets corresponding to the indecomposable decomposition of the zigzag persistence module $H_\ell(X_s)\rightarrow H_\ell(X_s\cup Y_s)\leftarrow H_\ell(Y_s)$.}    
		\label{fig:ar_colored2}
\end{figure}

We finally remark that we can also derive these properties in the Auslander-Reiten quivers of the other commutative ladders of finite type in a similar manner.

\section{Algorithm}\label{sec:algorithm}
This section provides an algorithm for computing the indecomposable decomposition of a given persistence module on $\CL(\tau_n)$ of finite type. Here, we only discuss the case $\CL(fb)$.

We emphasize that our presentation is self-contained since we use only simple matrix reductions based on the structure of Auslander-Reiten quivers and do not require high-level representation theory for describing algorithms (e.g., \cite{ptma}). 
An additional benefit of using this strategy is that a relationship to the algorithm in \cite{cd} for zigzag persistence modules can be easily seen. 
We believe that this strategy improves the accessibility for readers in computational topology. 

We start with the persistence module as input. Given a diagram of simplicial complexes, we can compute its persistence module by computing homology vector spaces and the induced maps between them using standard methods. For example, let $X' \subset X$ and assume that we have bases for $H_\ell(X')$ and $H_\ell(X)$. Then, one way to compute the induced map $H_\ell(X') \rightarrow H_\ell(X)$ is to write the basis of $H_\ell(X')$ in terms of the basis of $H_\ell(X)$. This amounts to solving a linear system. For more details, see the book \cite{comp_homo}.

\subsection{Overview}
%We first describe an overview of our algorithm. Let 
%\[
%\xymatrix{
%V^4 \ar[r]^-{f_{45}}&V^5&\ar[l]_-{f_{65}}V^6\\
%V^1 \ar[r]_-{f_{12}}\ar[u]^-{f_{14}}&V^2\ar[u]^-{f_{25}} &\ar[l]^-{f_{32}}V^3\ar[u]^-{f_{36}}
%}
%\]
%be a persistence module $V$. As input, the algorithm accepts the linear maps of $V$ written as matrices with respect to some bases chosen for $V^1,\hdots, V^6$. Then, we successively perform changes of bases on $V^i$ by row or column operations to extract indecomposables in $V$. The algorithm is divided into 5 steps and moves from right to left along the Auslander-Reiten quiver $\Gamma(\CL(fb))$ as shown in Figure \ref{fig:algofive}. Each step extracts the indecomposables included in the specified region of the figure. From the up-down symmetry in $\Gamma(\CL(fb))$ induced by the left-right symmetry of $\CL(fb)$, it suffices to present the algorithm for only one side.

We first give an overview of our algorithm. Let 
\[
\begin{tikzcd}
V^4 \rar{f_{45}} & V^5 \rar[leftarrow]{f_{65}} & V^6\\
V^1 \rar{f_{12}} \uar{f_{14}} & V^2 \uar{f_{25}} \rar[leftarrow]{f_{32}} & V^3 \uar[swap]{f_{36}}
\end{tikzcd}
%,
%\xymatrix{
%V^4 \ar[r]^-{f_{45}}&V^5&\ar[l]_-{f_{65}}V^6\\
%V^1 \ar[r]_-{f_{12}}\ar[u]^-{f_{14}}&V^2\ar[u]^-{f_{25}} &\ar[l]^-{f_{32}}V^3\ar[u]^-{f_{36}}
%}
\]
be a persistence module $V$. As input, the algorithm accepts the linear maps of $V$ written as matrices with respect to some bases chosen for $V^1,\hdots, V^6$. Then, we successively extract appropriate subrepresentations and decompose them into indecomposables by changes of bases. 

Here, the extracted subrepresentations are essentially the same as streamlined modules  \cite{cd}. The concept of streamlined modules was introduced 
 in order to inductively extract indecomposables in $\rep(\A_n)$ (see Lemma 4.3 in that paper) and these are characterized by the injectivity (resp. surjectivity) of the maps in one direction (resp. in the other direction) in the $\A_n$ quiver. 
Although our representations are not in $\rep(\A_n)$, we see in the following that the same technique can be applied by specifying appropriate subspaces. 

%In general, there are two parts in each step of our algorithm. 
%The first part is called {\em subspace tracking}, denoted by S in the following. 
%In this part, we track how a subspace of interest is brought to the other vector spaces. 
%Once we finish obtaining these subspaces, we move on to {\em basis arrangement}, denoted by B. Here, we go in the reverse direction of the route taken in the subspace tracking, and rewrite the matrices of the maps so that we can extract the direct summand $V'$ of $V=V'\oplus V''$ constructed by the obtained subspaces.
%We note that, even though we can compute the multiplicities of the indecomposables at each step using only the subspace tracking, we must explicitly obtain the direct summand $V'$ by the basis arrangement. This is because we send the complementary summand $V''$, which has to be a representation on $\CL(fb)$, to the next step as input.

In general, there are two parts in each step of our algorithm. 
The first part, denoted by S in the following, is to extract a subrepresentation specified by a certain subspace.
In this part, we track how the subspace of interest is brought to the other vector spaces. 

Once we finish obtaining the subrepresentation, we move on to basis arrangement, denoted by B. Here, we go in the reverse direction of the route taken in the subspace tracking of S. We rewrite the matrices of the maps so that we extract the 
subrepresentation as a direct summand $V'$ in $V=V'\oplus V''$ and decompose it into indecomposables.

We note that, even though we can compute the multiplicities of the indecomposables by extracting only the subrepresentation, we must explicitly obtain the direct sum decomposition by the basis arrangement. This is because we send the complementary summand $V''$, which has to be a representation on $\CL(fb)$, to the next step as input.

The outline of our algorithm is shown as follows:
%\vspace{0.3cm}\\
%\noindent{{\bf Outline of Algorithm.}}
\begin{enumerate} [Step 1:]
  \setlength{\parskip}{0cm}
  \setlength{\itemsep}{0.1cm}
\item Extract and decompose the subrepresentation specified by $\kernel f_{12}$ ($\kernel f_{32}$ by symmetry) .
\item Extract and decompose the subrepresentation specified by $\kernel f_{14}$ ($\kernel f_{36}$ by symmetry).
\item Extract and decompose the subrepresentation specified by $\kernel f_{45}$ ($\kernel f_{65}$ by symmetry).
\item Extract and decompose the subrepresentation specified by $V_1$ ($V_3$ by symmetry).
\item Decompose the subrepresentation on $\D_4$.
\end{enumerate} 

The algorithm consists of $5$ main steps moving from right to left along the Auslander-Reiten quiver $\Gamma(\CL(fb))$ as shown in Figure \ref{fig:algofive}. 
Each step extracts the indecomposables included in the specified region of the figure. 
Steps 1, 2, and 3 are reduced to finding and decomposing certain streamlined modules of $\A_n$ type. The basic process in Step 4 is the same, but we need to take special care for bases arrangements to handle the indecomposable with dimension vector $\dimVec{1,2,1,1,1,1}$. At Step 5, the remaining representation will be defined on the
$\D_4$ quiver:
\[
    \xymatrix{
        \circ \ar[r]&\circ&\ar[l]\circ\\
        &\circ \ar[u] &
    }.
\]   

From the up-down symmetry in $\Gamma(\CL(fb))$ induced by the left-right symmetry of $\CL(fb)$, it suffices to present the algorithm for only one side. 

\begin{figure}
 \centering\includegraphics{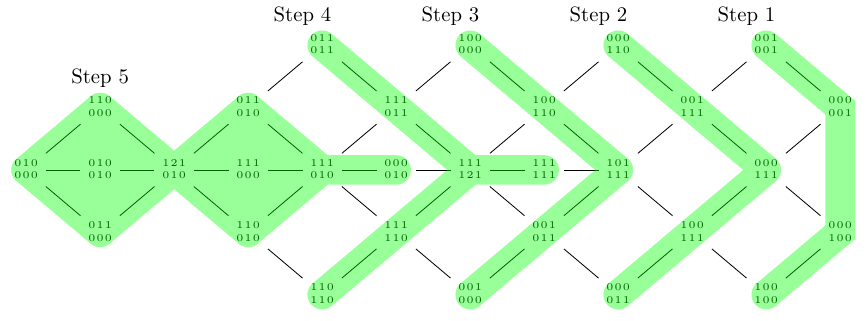}
	\caption{Flowchart of the algorithm for computing persistence diagrams.}
	\label{fig:algofive}
\end{figure}

\subsection{Detailed Description of the Algorithm}
We list some symbols and notations used in this section: 
\begin{enumerate}[(i)]
  \setlength{\parskip}{0cm}
  \setlength{\itemsep}{0.1cm}
\item $\vector{1}$ is an identity matrix with appropriate size.
\item $\vector{0}$ is a zero rectangular matrix with appropriate size.
\item $\vector{*}$ is a rectangular matrix with appropriate size. Capital letters are also used for this purpose.
\item Blank entries in matrices mean zeros.
\item $V^{i}_{k_1\dots k_a}$ is a subspace of $V^i$ extracted as a direct summand $V'$ with nonzero entries in its dimension vector $\dim V'$ specified by $k_1,\dots, k_a$.
\item $U^{i}_{k_1\dots k_a}$ is a subspace of $V^i$ specified by the subspace tracking, where the indices $k_1\dots k_a$ express the history of the track of the targeted subspace.
%\item $V^{i}_{k_1\dots k_a}=$ the subspace in $V^i$ extracted for an indecomposable $V'$
\item We abuse the symbol $f_{ij}$ to express both the linear map $f_{ij}: V^i\rightarrow V^j$ and its matrix representation, even if we change the bases of $V^i$ and $V^j$. 
\item We perform a change of basis on $V^i$ (resp. $V^j$) to simplify a map $f_{ij}:V^i\rightarrow V^j$. It should be automatically understood that the matrix forms of any other maps having $V^i$ (resp. $V^j$) as its domain or target space are appropriately modified by this change of basis.
\item A matrix form
$
    {\scriptsize \left[\begin{array}{cc}
            A_{11} & A_{12}\\
            A_{21} & A_{22}
        \end{array}\right]}
    :U_1\oplus U_2\rightarrow W_1\oplus W_2
$
    respects the direct sum decomposition, and each $A_{ij}$ expresses a submatrix with an appropriate size. If necessary, we add lines to explicitly describe the correspondence. For example, 
    \[
    {\scriptsize \left[\begin{array}{cc}
            A_{11} & A_{12}\\
			\hline            
            A_{21} & A_{22}\\
            \hline
            A_{31} & A_{32}\\
            A_{41} & A_{42}                        
        \end{array}\right]}
    :U_1\oplus U_2\rightarrow W_1\oplus W_2 \oplus W_3
    \]
    means that $A_{1j}$ and $A_{2j}$ correspond to $W_1$ and $W_2$, respectively, whereas both $A_{3j}$ and $A_{4j}$ correspond to $W_3$.
%\item The indecomposables are expressed by its dimension vector as in Section \ref{sec:proof}.
%\item $W$ (s.a.m.) and $X$ (b.a.m.) with appropriate indices indicates the complementary subspaces to the tracks. 
\item Column echelon and row echelon forms used in the exposition take the forms shown in Figure \ref{fig:echelon}. Namely, we set the bottom-left to be zero by column or row operations, respectively. 
\begin{figure}[h!]
    \centering
    \begin{tikzpicture}[scale=0.2]
        \drawmytikzmatrix
    \end{tikzpicture}~~~~~~~~
   \begin{tikzpicture}[scale=0.2,yscale=-1,rotate=270]
        \drawmytikzmatrix
    \end{tikzpicture}
\caption{Column echelon (left) and row echelon (right) forms}
\label{fig:echelon}
\end{figure}
\item Let $A: X\rightarrow Y$ be a matrix form of a linear map, and $P$ and $Q$ be basis change matrices on $X$ and $Y$, respectively. We adopt the usual convention of writing the new matrix form induced by these changes of bases as  $Q^{-1}AP$.
\item $\dimInd{a_1,a_2,a_3,a_4,a_5,a_6}$ expresses the indecomposable in $\Gamma(\CL(fb))$ whose dimension vector is $\dimVec{a_1,a_2,a_3,a_4,a_5,a_6}$.
\end{enumerate}

We warn that (v) and (vi) should not be taken as definitions of these subspaces at this point, but rather as guidelines on how we use the notation. The specific definitions will follow in the discussion below.

\begin{flushleft}{\bf Step 1:} $\dimInd{1,,,,,}$, $\dimInd{1,,,1,}$, $\dimInd{,,1}$, $\dimInd{,,1,,,1}$ 
\end{flushleft}
We extract the indecomposables $\dimInd{1,,,,,}$ and $\dimInd{1,,,1,}$; by symmetry,
the algorithm for $\dimInd{,,1}$ and $\dimInd{,,1,,,1}$ is similar.
Note that the indecomposables $\dimInd{1,,,,,}$ and $\dimInd{1,,,1,}$ are characterized by $\kernel f_{12}$. Thus, we extract them by applying the subspace tracking to $\kernel f_{12}$. 

\begin{enumerate}[S1.]
\item Perform a basis change on $V^1$ to obtain a column echelon form
    \[
    f_{12}=
    \left[\begin{array}{c|c}\vector{0} & A\end{array}\right]:
    V^1=U^1_1\oplus W^1 \rightarrow V^2,
    \]
    where the submatrix $A$ is in a column echelon form with nonzero columns.
We denote $U^1_1 = \kernel f_{12}$ and $k = \dim U^1_1$. 

\item Perform a basis change on $V^4$ to obtain a row echelon form
    \[
    f_{14} ={\scriptsize
    \left[\begin{array}{c|c}
            B & C\\
            \hline
            \vector{0} & D
        \end{array}\right]}
    :V^1=U^1_1\oplus W^1\rightarrow V^4=V^4_{14}\oplus W^4.
    \]
The dimension $\ell=\dim V^4_{14}$, which determines the position of the horizontal line, is given by the largest $\ell$ such that the left $k$ entries of the $\ell$th row is nonzero in this row echelon form for $f_{14}$. Equivalently, $V^4_{14}$ is determined by the image $V^4_{14}=f_{14}(U^1_1)$, and hence is obtained by tracking $U^1_1=\kernel f_{12}$.
We often use this way of decomposing target spaces in the following exposition without explicit comments.
\end{enumerate}

It follows from $f_{45}V^4_{14} = f_{45}f_{14}(\kernel f_{12}) = f_{25}f_{12}(\kernel f_{12}) = 0$ that there is no need to track $V^4_{14}$ by $f_{45}$. Thus, we stop the subspace tracking and switch to the basis arrangement. 
Here, we note that the representation 
$
U^1_1\stackrel{B}{\longrightarrow} V^4_{14}
$
can be regarded as a left-streamlined module in $\rep(\A_2)$. 
In the basis arrangement part, we choose suitable bases to extract the left-streamlined module as a direct summand and to decompose it into indecomposables.

%	By commutativity, we have $f_{45}V^4_{14} = f_{45}f_{14}(\kernel
%    f_{12}) = f_{25}f_{12}(\kernel f_{12}) = 0$. Thus, there is no need
%    to analyze $f_{45}$, and we may begin basis arrange mode.

\begin{enumerate}[B1.]
\item Perform a change of basis on $U^1_1$ to transform the submatrix $B$ as 
     \[
    f_{14}=
    {\scriptsize \left[\begin{array}{cc|c}
            \vector{1}&\vector{0} & C\\
            \vector{0}&\vector{0} & D
        \end{array}\right]}:
    V^1=\left(V^1_{14}\oplus V^1_{1} \right)\oplus W^1\rightarrow V^4=V^4_{14}\oplus W^4,
    \]
    where $U^1_1 = V^1_{14}\oplus V^1_1$.
     Then, we replace the subspace $W^1$ with $X^1$ in such a way that $f_{14}$ becomes
    \begin{equation}\label{eq:f14_step1}
        f_{14}=
        {\scriptsize \left[\begin{array}{cccc}
                \vector{1}&\vector{0} & \vector{0}\\
                \vector{0}&\vector{0} & D
            \end{array}\right]}:
        V^1=V^1_{14}\oplus V^1_{1}\oplus X^1\rightarrow V^4=V^4_{14}\oplus W^4
    \end{equation}
    by the basis change matrix 
${\scriptsize \left[\begin{array}{ccc}
                \vector{1} & \vector{0} & -C\\
                \vector{0} &\vector{1} & \vector{0}\\
                \vector{0} &\vector{0} &\vector{1}
            \end{array}
        \right]
}$ on $V^1$.
%    \[ \begin{array}{rl}
%        R =& \left[\begin{array}{ccc}
%                \vector{1} & \vector{0} & \vector{-C}\\
%                \vector{0} &\vector{1} & \vector{0}\\
%                \vector{0} &\vector{0} &\vector{1}
%            \end{array}
%        \right].
%    \end{array}
%    \]
    It should be noted that these changes of bases on $V^1$ do not change the matrix of $f_{12}$ at all, and hence $\kernel f_{12}=V^1_{14}\oplus V^1_1$ is maintained. 
%    \[f_{12} = \left[ \begin{array}{c|c|c}
%            \vector{0} & \vector{0} & \vector{A}
%        \end{array}
%    \right]:
%    V^1=V^1_{14}\oplus V^1_1 \oplus X^1 \rightarrow V^2.
%    \]
\end{enumerate}

From the subspace tracking and basis arrangement, we obtain the following direct summands of $V$
\[
\begin{array}{rcccrcc}
    \dimInd{1,}^{k_1} &=& \spaceInd{V^1_1,}  &\;{\rm and}\;&
    \dimInd{1,,,1}^{k_{14}} &=& \spaceInd{V^1_{14},,,V^4_{14}},
%    \dimInd{,,1}^{k_3} &\simeq& \spaceInd{,,V^3_3,} &\;\;&
%    \dimInd{,,1,,,1}^{k_{36}} &\simeq& \spaceInd{,,V^3_{36},,,V^6_{36}}.
\end{array}
\]
where $k_1 = \dim V^1_1$ and $k_{14} = \dim V^1_{14} = \dim V^4_{14}$.
As noted at the beginning of Step 1, the direct summands $\dimInd{,,1}^{k_3}$ and $\dimInd{,,1,,,1}^{k_{36}}$ are derived in a way similar to the above.
We remove these direct summands from 
\[
V = \left(\dimInd{1,}^{k_1} \oplus \dimInd{1,,,1}^{k_{14}} \oplus
    \dimInd{,,1}^{k_3} \oplus \dimInd{,,1,,,1}^{k_{36}}\right) \oplus V'
\]
and send $V'$ to the next step. In the level of the matrices of
$f_{ij}$, we delete the appropriate rows and columns.

% ----------------------------------------
% Step 2
% ----------------------------------------
\begin{flushleft}{\bf Step 2:} $\dimInd{1,1}$, $\dimInd{1,1,1,0,0,1}$, $\dimInd{1,1,1}$, $\dimInd{1,1,1,1}$, $\dimInd{,1,1}$
\end{flushleft}
By renaming the summand $V'$ in Step 1 to $V$, we assume that the maps $f_{12}$
and $f_{32}$ are injective. We extract the indecomposables $\dimInd{1,1}, \dimInd{1,1,1}, \dimInd{1,1,1,0,0,1}$ 
by tracking the subspace $\kernel f_{14}$. Since the computations are similar to those in the previous step, we explain only its core parts.

\begin{enumerate}[S1.]
\item Perform a change of basis on $V^1$ to obtain a column echelon form
\[
f_{14}=\left[\begin{array}{cc}
\vector{0} &\vector{*}
\end{array}\right]: V^1=U^1_1\oplus W^1\rightarrow V^4.
\]
\item Perform a change of basis on $V^2$ to obtain a row echelon form
\[
f_{12}={\scriptsize \left[\begin{array}{cc}
\vector{1}&\vector{0}\\
\vector{0}&\vector{*}
\end{array}
\right]}: 
V^1=U^1_1\oplus W^1\rightarrow V^2=U^2_{12}\oplus W^2.
\]
Since $f_{12}$ is injective, we can obtain the identity submatrix $\vector{1}$ in the above row echelon form.
\item Perform a change of basis on $V^3$ to obtain a column echelon form
\[
f_{32}={\scriptsize \left[\begin{array}{c|c}
\vector{*}&\vector{*}\\
\vector{0}&A
\end{array}
\right]}: 
V^3=U^3_{123}\oplus W^3\rightarrow V^2=U^2_{12}\oplus W^2,
\]
where $U^3_{123}$ (equivalently, the position of the vertical line in $f_{32}$) is determined by the preimage $U^3_{123}=f^{-1}_{32}(U^2_{12})$. Similar to S2 in Step 1, we also use this way of decomposing domains of maps without further comments.
\item Perform a change of basis on $V^6$ to obtain a row echelon form
\[
f_{36}={\scriptsize \left[\begin{array}{cc}
B_1 & B_2\\
\vector{0} & \vector{*}
\end{array}\right]}: 
V^3=U^3_{123}\oplus W^3\rightarrow V^6=V^6_{1236}\oplus W^6.
\]
\end{enumerate}

It follows from commutativity that we have $f_{65}(V^6_{1236})=0$, and hence the subspace tracking may stop here.
Again, we remark that the representation
\[
	U^1_1\longrightarrow U^2_{12}\longleftarrow U^3_{123}\longrightarrow V^6_{1236}
\]
is a left-streamlined module in $\rep(\A_4)$.
Next, we switch to the basis arrangement.
\begin{enumerate}[B1.]
\item Perform a change of basis on $U^3_{123}$ to transform $B_1$ as follows
\[
f_{36}={\scriptsize \left[\begin{array}{cc|c}
\vector{0}&\vector{1} & B_2\\
\vector{0}&\vector{0}& \vector{*}
\end{array}\right]}: 
V^3= \left(V^3_{123}\oplus V^3_{1236}\right)\oplus W^3\rightarrow V^6=V^6_{1236}\oplus W^6.
\]
Replace the subspace $W^3$ with $X^3$ in such a way that $f_{36}$ becomes
\[
f_{36}={\scriptsize \left[\begin{array}{ccc}
\vector{0}&\vector{1} & \vector{0}\\
\vector{0}&\vector{0}& \vector{*}
\end{array}\right]}: 
V^3=V^3_{123}\oplus V^3_{1236}\oplus X^3\rightarrow V^6=V^6_{1236}\oplus W^6
\]
by the basis change matrix 
${\scriptsize \left[\begin{array}{ccc}
\vector{1} & \vector{0} & \vector{0} \\
\vector{0} & \vector{1} & -B_2\\
\vector{0} & \vector{0} & \vector{1} 
\end{array}\right]}$ on $V^3$. 
\item The map $f_{32}$ is now expressed as 
\[
	f_{32}={\scriptsize \left[\begin{array}{ccc}
\vector{*} & \vector{*} & \vector{*} \\
\vector{0} & \vector{0} & A
\end{array}\right]}: V^3=V^3_{123}\oplus V^3_{1236}\oplus X^3\rightarrow V^2=U^2_{12}\oplus W^2
\]
after the changes of bases. We note that $f_{32}$ is injective and that the submatrix $A$, which has no zero columns by construction of $U^3_{123}$, is in a column echelon form. Perform changes of bases on $U^2_{12}$ and $W^2$, independently of each other, to obtain the matrix form
\[
f_{32}={\scriptsize \left[\begin{array}{ccc}
\vector{0} & \vector{0} & C_1 \\
\hline
\vector{1} & \vector{0} & C_2 \\
\hline
\vector{0} & \vector{1} & C_3 \\
\hline
\vector{0} & \vector{0} & \vector{0} \\
\vector{0} & \vector{0} & \vector{1} 
\end{array}\right]}: 
V^3=V^3_{123}\oplus V^3_{1236}\oplus X^3\rightarrow V^2=V^2_{12}\oplus V^2_{123}\oplus V^2_{1236}\oplus W^2.
\]
Replace the subspace $W^2$ with $X^2$ in such a way that $f_{32}$ becomes
\[
f_{32}={\scriptsize \left[\begin{array}{ccc}
\vector{0} & \vector{0} & \vector{0} \\
\hline
\vector{1} & \vector{0} & \vector{0} \\
\hline
\vector{0} & \vector{1} & \vector{0} \\
\hline
\vector{0} & \vector{0} & \vector{0} \\
\vector{0} & \vector{0} & \vector{1} 
\end{array}\right]}: 
V^3=V^3_{123}\oplus V^3_{1236}\oplus X^3\rightarrow V^2=V^2_{12}\oplus V^2_{123}\oplus V^2_{1236}\oplus X^2
\]
by the basis change matrix
$
{\scriptsize \left[\begin{array}{ccccc}
\vector{1} & \vector{0} & \vector{0} & \vector{0}& -C_1\\
\vector{0} & \vector{1} & \vector{0} & \vector{0}& -C_2\\
\vector{0} & \vector{0} & \vector{1} & \vector{0}& -C_3\\
\vector{0} & \vector{0} & \vector{0} & \vector{1}& \vector{0}\\
\vector{0} & \vector{0} & \vector{0} & \vector{0}& \vector{1}
\end{array}\right]^{-1}
}$
on $V^2$.
\item The map $f_{12}$ is now expressed as 
\[
f_{12}={\scriptsize \left[\begin{array}{cc}
\vector{*} &\vector{*}\\
\vector{*} &\vector{*}\\
\vector{*} &\vector{*}\\
\vector{0} &\vector{*}
\end{array}\right]}: V^1=U^1_1\oplus W^1\rightarrow V^2=V^2_{12}\oplus V^2_{123}\oplus V^2_{1236}\oplus X^2.
\]
We note that $U^1_1\simeq U^2_{12}=V^2_{12}\oplus V^2_{123}\oplus V^2_{1236}$. Perform a change of basis on $U^1_1$ to obtain
\[
f_{12}={\scriptsize \left[\begin{array}{cccc}
\vector{1} & \vector{0} & \vector{0} & D_1\\
\vector{0} & \vector{1} & \vector{0} & D_2\\
\vector{0} & \vector{0} & \vector{1} & D_3\\
\vector{0} & \vector{0} & \vector{0} & \vector{*}
\end{array}\right]}: 
V^1=V^1_{12}\oplus V^1_{123}\oplus V^1_{1236}\oplus W^1\rightarrow V^2=V^2_{12}\oplus V^2_{123}\oplus V^2_{1236}\oplus X^2. 
\]
Replace the subspace $W^1$ with $X^1$ in such a way that $f_{12}$ becomes
\[
f_{12}={\scriptsize \left[\begin{array}{cccc}
\vector{1} & \vector{0} & \vector{0} & \vector{0} \\
\vector{0} & \vector{1} & \vector{0} & \vector{0} \\
\vector{0} & \vector{0} & \vector{1} & \vector{0} \\
\vector{0} & \vector{0} & \vector{0} & \vector{*} 
\end{array}\right]}: V^1=V^1_{12}\oplus V^1_{123}\oplus V^1_{1236}\oplus X^1\rightarrow V^2=V^2_{12}\oplus V^2_{123}\oplus V^2_{1236}\oplus X^2\\
\]
by the basis change matrix
${\scriptsize 
\left[\begin{array}{cccc}
\vector{1} & \vector{0} & \vector{0} & -D_1 \\
\vector{0} & \vector{1} & \vector{0} & -D_2 \\
\vector{0} & \vector{0} & \vector{1} & -D_3 \\
\vector{0} & \vector{0} & \vector{0} & \vector{1} 
\end{array}\right]
}$ on $V^1$. 
This change of basis does not affect $\kernel f_{14}$ at all.
\end{enumerate}

We obtain the following direct summands of $V$:
%\[
%\begin{array}{rcc}
%    \dimInd{1,1}^{k_{12}} &\simeq& \spaceInd{V^1_{12},V^2_{12}} \\
%    \dimInd{1,1,1}^{k_{123}} &\simeq& \spaceInd{V^1_{123},V^2_{123},V^3_{123}} \\
%    \dimInd{1,1,1,0,0,1}^{k_{1236}} &\simeq& \spaceInd{V^1_{1236},V^2_{1236},V^3_{1236},0,0,V^6_{1236}}, 
%\end{array}
%\]
\begin{eqnarray*}
    &&\dimInd{1,1}^{k_{12}} = \spaceInd{V^1_{12},V^2_{12}},~~~
    \dimInd{1,1,1}^{k_{123}} = \spaceInd{V^1_{123},V^2_{123},V^3_{123}},\\
    &&\dimInd{1,1,1,0,0,1}^{k_{1236}} = \spaceInd{V^1_{1236},V^2_{1236},V^3_{1236},0,0,V^6_{1236}}, 
\end{eqnarray*}
where 
\[
   k_{12} = \dim V^i_{12}~(i=1,2),~~
   k_{123} = \dim V^i_{123}~(i=1,2,3),~~
   k_{1236} = \dim V^i_{1236}~(i=1,2,3,6).
\]
%\begin{eqnarray*}
%    &&k_{12} = \dim V^1_{12} = \dim V^2_{12},~~~~~
%    k_{123} = \dim V^1_{123} = \dim V^2_{123} = \dim V^3_{123}, \\
%    &&k_{1236} = \dim V^1_{1236} = \dim V^2_{1236} = \dim V^3_{1236} = \dim V^6_{1236}.
%\end{eqnarray*}
%\[
%\begin{array}{ccccccccc}
%    k_{12} &=& \dim V^1_{12} &=& \dim V^2_{12} && &&\\
%    k_{123} &=& \dim V^1_{123} &=& \dim V^2_{123} &=& \dim V^3_{123} && \\
%    k_{1236} &=& \dim V^1_{1236} &=& \dim V^2_{1236} &=& \dim V^3_{1236} &=& \dim V^6_{1236}.
%\end{array}
%\]
By a symmetric argument, we obtain $\dimInd{,1,1,,,}^{k_{23}}$ and $\dimInd{1,1,1,1,,}^{k_{1234}}$. The remaining direct summand $V'$ in
\[
V=\left(\dimInd{1,1}^{k_{12}}\oplus\dimInd{1,1,1}^{k_{123}}\oplus\dimInd{1,1,1,0,0,1}^{k_{1236}}\oplus\dimInd{,1,1,,,}^{k_{23}}\oplus\dimInd{1,1,1,1,,}^{k_{1234}}\right)\oplus V'
\]
is sent to the next step.

\begin{flushleft}{\bf Step 3:} $\dimInd{0,0,0,1}$, $\dimInd{1,1,,1}$, $\dimInd{1,1,1,1,,1}$, $\dimInd{,1,1,,,1}$, 
$\dimInd{,,,,,1}$
\end{flushleft}
We extract the indecomposables 
$\dimInd{0,0,0,1}, 
    \dimInd{1,1,,1}, 
    \dimInd{1,1,1,1,,1}, 
    \dimInd{,1,1,,,1}, 
    \dimInd{,,,,,1}
$ in this step by applying the subspace tracking and the basis arrangement to $\kernel f_{45}$ and $\kernel f_{65}$. At the start of this step, the maps $f_{12}, f_{32}, f_{14}, f_{36}$ are injective. 
Since the tasks here are essentially the same as the previous two steps, we omit the exposition of this step.

\begin{flushleft}{\bf Step 4:}
\end{flushleft}
This step is divided into two parts (Steps 4.1 and 4.2). 
The reason for this division is that 
we need to handle the case of $\dimInd{1,1,1,1,1,1}$ and $\dimInd{1,2,1,1,1,1}$ in a special manner, and this is facilitated by first extracting 
$\dimInd{,1,1,,1,1},
    \dimInd{,1,1,1,1,1},
    \dimInd{1,1,,1,1,},
    \dimInd{1,1,,1,1,1}$. We also note that all maps except $f_{25}$ are injective as input from Step 3. 

\begin{flushleft}{\bf Step 4.1:}  $\dimInd{,1,1,,1,1}$, $\dimInd{,1,1,1,1,1}$, $\dimInd{1,1,,1,1,}$, 
$\dimInd{1,1,,1,1,1}$
\end{flushleft}
We extract the indecomposables 
$\dimInd{,1,1,,1,1}$ and $\dimInd{,1,1,1,1,1}$ by applying the subspace tracking and basis arrangement to $V^3$.

\begin{enumerate}[S1.]
\item Perform a change of basis on $V^6$:
$
f_{36}={\scriptsize \left[\begin{array}{c}
\vector{1}\\
\vector{0}
\end{array}\right]}: V^3\rightarrow V^6=U^6_{36}\oplus W^6.
$
\item Perform a change of basis on $V^5$:
$
f_{65}={\scriptsize \left[\begin{array}{cc}
\vector{1} & \vector{0}\\
\vector{0} & \vector{*}
\end{array}\right]}: V^6=U^6_{36}\oplus W^6\rightarrow V^5=U^5_{356}\oplus W^5.
$
\item Perform a change of basis on $V^4$ to obtain a column echelon form 
\[
f_{45}={\scriptsize \left[\begin{array}{cc}
\vector{*} & \vector{*}\\
\vector{0} & \vector{*}
\end{array}\right]}: 
V^4=U^4_{3456}\oplus W^4\rightarrow V^5=U^5_{356}\oplus W^5.
\]
\item Perform a change of basis on $V^1$ to obtain a column echelon form
\[
f_{14}={\scriptsize \left[\begin{array}{cc}
\vector{*} & \vector{*}\\
\vector{0} & \vector{*}
\end{array}\right]}: 
V^1=U^1_{13456}\oplus W^1\rightarrow V^4=U^4_{3456}\oplus W^4.
\]
\end{enumerate}

\begin{enumerate}[B1.]
\item Perform changes of bases on $U^4_{3456}$ and $W^4$, independently of each other, to obtain
\[
f_{14}={\scriptsize \left[\begin{array}{cc}
\vector{0} & A_1\\
\hline
\vector{1} & A_2\\
\hline
\vector{0} & \vector{0}\\
\vector{0} & \vector{1}
\end{array}\right]}: V^1=U^1_{13456}\oplus W^1\rightarrow V^4=V^4_{23456}\oplus U^4_{13456}\oplus W^4.
\]
Replace $W^4$ with $X^4$ in such a way that $f_{14}$ becomes
\[
f_{14}={\scriptsize \left[\begin{array}{cc}
\vector{0} & \vector{0}\\
\hline
\vector{1} & \vector{0}\\
\hline
\vector{0} & \vector{0}\\
\vector{0} & \vector{1}
\end{array}\right]}: V^1=U^1_{13456}\oplus W^1\rightarrow V^4= V^4_{23456}\oplus U^4_{13456}\oplus X^4
\]
by the basis change matrix
$
{\scriptsize \left[\begin{array}{cccc}
\vector{1} & \vector{0} & \vector{0} & -A_1 \\
\vector{0} & \vector{1} & \vector{0} & -A_2 \\
\vector{0} & \vector{0} & \vector{1} & \vector{0} \\
\vector{0} & \vector{0} & \vector{0} & \vector{1} 
\end{array}\right]^{-1}}
$
on $V^4$. 

\item Note that the columns of the two submatrices $[B_1~B_2]$ and $C$ in 
\[
f_{45}={\scriptsize \left[\begin{array}{ccc}
B_1&B_2 & \vector{*}\\
\vector{0}&\vector{0} & C
\end{array}\right]}: 
V^4=V^4_{23456}\oplus U^4_{13456}\oplus X^4\rightarrow V^5=U^5_{356}\oplus W^5
\]
are linearly independent. Perform changes of bases on $U^5_{356}$ and $W^5$, independently of each other, to obtain
\[
f_{45}={\scriptsize \left[\begin{array}{ccc}
\vector{0} & \vector{0} & D_1\\
\hline
\vector{1} & \vector{0} & D_2\\
\hline
\vector{0} & \vector{1} & D_3\\
\hline
\vector{0} & \vector{0} & \vector{1}\\
\vector{0} & \vector{0} & \vector{0}
\end{array}\right]}: 
V^4=V^4_{23456}\oplus U^4_{13456}\oplus X^4\rightarrow V^5=V^5_{2356}\oplus V^5_{23456}\oplus U^5_{13456}\oplus W^5.
\]
Replace $W^5$ with $X^5$ in such a way that $f_{45}$ becomes
\[
f_{45}={\scriptsize \left[\begin{array}{ccc}
\vector{0} & \vector{0} & \vector{0}\\
\hline
\vector{1} & \vector{0} & \vector{0}\\
\hline
\vector{0} & \vector{1} & \vector{0}\\
\hline
\vector{0} & \vector{0} & \vector{1}\\
\vector{0} & \vector{0} & \vector{0}
\end{array}\right]}: 
V^4=V^4_{23456}\oplus U^4_{13456}\oplus X^4\rightarrow V^5=V^5_{2356}\oplus V^5_{23456}\oplus U^5_{13456}\oplus X^5
\]
by the basis change matrix
$
{\scriptsize \left[\begin{array}{ccccc}
\vector{1} & \vector{0} & \vector{0} & -D_1 & \vector{0} \\
\vector{0} & \vector{1} & \vector{0} & -D_2 & \vector{0} \\
\vector{0} & \vector{0} & \vector{1} & -D_3 & \vector{0} \\
\vector{0} & \vector{0} & \vector{0} & \vector{1} & \vector{0}\\
\vector{0} & \vector{0} & \vector{0} & \vector{0} & \vector{1}
\end{array}\right]^{-1}}
$
on $V^5$.
\item Note that the map $f_{65}$ is now in the form
\[
f_{65}={\scriptsize \left[\begin{array}{cc}
\vector{*} & \vector{*}\\
\vector{*} & \vector{*}\\
\vector{*} & \vector{*}\\
\vector{0} & \vector{*}
\end{array}\right]
}
: V^6=U^6_{36}\oplus W^6\rightarrow V^5=V^5_{2356}\oplus V^5_{23456}\oplus U^5_{13456}\oplus X^5
\]
and $U^6_{36}\simeq V^5_{2356}\oplus V^5_{23456}\oplus U^5_{123456}$.
Perform a change of basis on $V^6$:
\[
f_{65}={\scriptsize \left[\begin{array}{cccc}
\vector{1} & \vector{0} & \vector{0} &\vector{0}\\
\vector{0} & \vector{1} & \vector{0} &\vector{0}\\
\vector{0} & \vector{0} & \vector{1} &\vector{0}\\
\vector{0} & \vector{0} & \vector{0} & \vector{*}
\end{array}\right]}: 
V^6=V^6_{2356}\oplus V^6_{23456}\oplus U^6_{13456}\oplus X^6 \rightarrow V^5=V^5_{2356}\oplus V^5_{23456}\oplus U^5_{13456}\oplus X^5.
\]
\item Perform a change of basis on $V^3$:
\[
f_{36}={\scriptsize \left[\begin{array}{ccc}
\vector{1} & \vector{0} &\vector{0}\\
\vector{0} & \vector{1} & \vector{0}\\
\vector{0} & \vector{0} & \vector{1}\\
\vector{0} & \vector{0} & \vector{0}
\end{array}\right]}:
V^3=V^3_{2356}\oplus V^3_{23456}\oplus U^3_{13456}\rightarrow V^6=V^6_{2356}\oplus V^6_{23456}\oplus U^6_{13456}\oplus X^6.
\]
\item Perform a change of basis on $V^2$:
\[
f_{32}={\scriptsize \left[\begin{array}{ccc}
\vector{1} & \vector{0} & \vector{0}\\
\vector{0} & \vector{1} & \vector{0}\\
\vector{0} & \vector{0} & \vector{1}\\
\vector{0} & \vector{0} & \vector{0}
\end{array}\right]}: 
V^3=V^3_{2356}\oplus V^3_{23456}\oplus U^3_{13456}\rightarrow V^2=V^2_{2356}\oplus V^2_{23456}\oplus U^2_{123456}\oplus W^2.
\]
\item By commutativity, the map $f_{25}$ is expressed by 
\[
f_{25}={\scriptsize \left[\begin{array}{cccc}
\vector{1} &\vector{0}& \vector{0} & E_1\\
\vector{0} & \vector{1}&\vector{0} & E_2\\
\vector{0} & \vector{0}&\vector{1} & E_3\\
\vector{0} & \vector{0}&\vector{0} & \vector{*}
\end{array}\right]}:
V^2=V^2_{2356}\oplus V^2_{23456}\oplus U^2_{123456}\oplus W^2\rightarrow V^5=V^5_{2356}\oplus V^5_{23456}\oplus U^5_{13456}\oplus X^5.
\] 
Replace the subspace $W^2$ with $X^2$ in such a way that $f_{25}$ becomes
\[
f_{25}={\scriptsize \left[\begin{array}{cccc}
\vector{1} &\vector{0}& \vector{0} & \vector{0}\\
\vector{0} & \vector{1}&\vector{0}& \vector{0}\\
\vector{0} & \vector{0}&\vector{1}& \vector{0}\\
\vector{0} & \vector{0}&\vector{0}& \vector{*}
\end{array}\right]}:
V^2=V^2_{2356}\oplus V^2_{23456}\oplus U^2_{123456}\oplus X^2\rightarrow V^5=V^5_{2356}\oplus V^5_{23456}\oplus U^5_{13456}\oplus X^5
\]
by the basis change matrix
$
{\scriptsize \left[\begin{array}{cccc}
\vector{1} & \vector{0} & \vector{0}& -E_1\\
\vector{0} & \vector{1} & \vector{0}& -E_2\\
\vector{0} & \vector{0} & \vector{1}& -E_3\\
\vector{0} & \vector{0} & \vector{0}& \vector{1}
\end{array}\right]}
$ on $V^2$.
We note that this change of basis does not affect on $f_{32}$ at all. 
\end{enumerate}
From these subspace tracking and basis arrangement, we obtain the direct summands $\dimInd{0,1,1,0,1,1}^{k_{2356}}$ and $\dimInd{0,1,1,1,1,1}^{k_{23456}}$.
The direct summands $\dimInd{1,1,,1,1,}^{k_{1245}}, \dimInd{1,1,,1,1,1}^{k_{12456}}$ are obtained in a similar way by symmetry. 

\begin{flushleft}{\bf Step 4.2:}
    $\dimInd{1,1,1,1,1,1}$, $\dimInd{1,2,1,1,1,1}$.
\end{flushleft}
In this step, the subspace tracking is similar to previous steps. However, the basis arrangement requires a different and careful treatment.

\begin{enumerate}[S1.]
\item Perform a change of basis on $V^4$:
$
    f_{14}={\scriptsize \left[
        \begin{array}{c}
            \vector{1}\\ 
            \vector{0}
        \end{array}
    \right]}: V^1\rightarrow V^4=U^4_{14}\oplus W^4
$.
%where $U^4_{123456} = f_{14}V^1 \simeq V^1$.    
\item Perform a change of basis on $V^5$:
$
    f_{45}={\scriptsize \left[
        \begin{array}{cc}
            \vector{1} &\vector{0}\\ 
            \vector{0} &\vector{*}
        \end{array}\right]}: V^4=U^4_{14}\oplus W^4\rightarrow
    V^5=U^5_{145}\oplus W^5
$.
%and $U^5_{123456} = f_{45}U^4_{123456} \simeq U^4_{123456}$.
%Since $f_{45}$ is injective, $f_{45}$ restricted to $U^4_{123456}$ is
%also injective. This gives the required isomorphism.
\item Perform a change of basis on $V^6$:
$
    f_{65}={\scriptsize \left[
        \begin{array}{cc}
            \vector{1} & \vector{0} \\ 
            \vector{0} & \vector{*}
        \end{array}\right]}: V^6=U^6_{1456}\oplus W^6 \rightarrow V^5=U^5_{145}\oplus W^5.$
Note that we obtain $\vector{1}$ in the top left of $f_{65}$, since the direct summand consisting of $\dimInd{1,1,0,1,1,0}$ is already removed in the previous step.

\item Perform a change of basis on $V^3$:
    \[
    f_{36}={\scriptsize \left[
        \begin{array}{c}
            \vector{1}\\ 
            \vector{0}
        \end{array}\right]}: V^3 \rightarrow U^6_{1456}\oplus W^6.
    \]
Similarly, we obtain $\vector{1}$ in $f_{36}$, since the direct summand consisting of $\dimInd{1,1,,1,1,1}$ is already removed.
% Thus, we have a sequence of isomorphisms
%    \[
%    V^1 \simeq U^4_{123456} \simeq U^5_{123456} \simeq U^6_{123456} \simeq
%    V^3.
%    \]
In particular, this implies $\dim V^1=\dim V^3$, and we denote this dimension by $s$. 
\end{enumerate}

Now we switch to the basis arrangement to derive $\dimInd{1,1,1,1,1,1}$ and $\dimInd{1,2,1,1,1,1}$, where the linear maps in $\dimInd{1,2,1,1,1,1}$ are shown in (\ref{eq:twodimentry1}). They are extracted by induction in the following. Let us denote the matrix forms of $f_{12}$ and $f_{32}$ by
\[
f_{12}=\left[\begin{array}{ccc}
        a_1 & \cdots & a_s
    \end{array}\right] \text{~~and~~}
f_{32}=\left[\begin{array}{ccc}
        b_1 & \cdots & b_s
    \end{array}\right],
\]
where $a_i$ and $b_i$ are the columns of $f_{12}$ and $f_{32}$, respectively.
By commutativity in $\CL(fb)$ and the bases chosen in the subspace tracking, we have $f_{25}(a_i)=f_{25}(b_i) \in V^5$ for $i=1,\dots,s$ and these vectors in $V^5$ are linearly independent.

Let us introduce a matrix defined by
\[
C=\left[\begin{array}{ccccccc}
        a_1 & b_1 & a_2 & b_2 & \cdots & a_s & b_s
    \end{array}\right].
\]
In the following induction, we transform $C$ by row and column operations into 
\begin{equation}\label{eq:hatC}
    \hat C = \left[\begin{array}{ccc|ccccc}
            1&&            &&         &&            &\\
            &\ddots&       &&         &\vector{0}&  &\\
            &&1            &&&&&\\
            \hline
            &&             &1&1       &&            &\\
            &\vector{0}&   &&         &\ddots&      &\\
            &&             &&         &&1           &1\\
            \hline
            &\vector{0}&   &&         &\vector{0}&  &
        \end{array}\right]
\end{equation}
with the $2s_1\times 2s_1$ identity matrix in the top left and the $s_2\times
2s_2$ matrix in the middle right for some $s_1,s_2\in\N_0$ with
$s=s_1+s_2$. In transforming $C$ to the above form, we perform each
column operation on a pair of columns $c_i, c_j$ in $C=[c_1\dots c_{2s}]$ with $i=j$ (mod 2).
%column numbers of the same
%parity. That is, we swap (or add multiples of) even-numbered columns
%only to even-numbered columns, and
%likewise odd-numbered columns only to odd-numbered columns.

%Before proving that we can indeed write $C$ in this form, we note the following.
Recall that a row operation on $C$ corresponds to a basis change on $V^2$, while
each column operation following the rule mentioned above can be viewed as a basis change for either $V^1$ or $V^3$. Hence, from $\hat{C}$, we immediately obtain 
\[
f_{12}=\left[\begin{array}{cccc|ccc}
        1&     &      &     &    &          &\\
        0&     &      &     &    &          &\\
        &1     &      &     &    &          &\\
        &0     &      &     &    &\vector{0}&\\
        &      &\ddots&     &    &          &\\
        &      &      &    1&    &          &\\
        &      &      &    0&    &          &\\
        \hline
        &      &      &     &1   &          &\\
        &      &\vector{0}& &    &\ddots    &\\
        &&&&&                               &1\\
        \hline
        &      &\vector{0}& &    &\vector{0}&
    \end{array}\right],~~~
f_{32}=\left[\begin{array}{cccc|ccc}
        0&&&&&&\\
        1&&&&&&\\
        &0&&&&&\\
        &1&&&&\vector{0}&\\
        &&\ddots&&&&\\
        &&&1&&&\\
        &&&0&&&\\
        \hline
        &&&&1&&\\
        &&\vector{0}&&&\ddots&\\
        &&&&&&1\\
        \hline
        &&\vector{0}&&&\vector{0}&
    \end{array}\right].
\]
Of course, the matrix forms of $f_{14}, f_{45}, f_{65}$, $f_{36}$ will change under the process of the transformation of $C$. Hence, in the induction process, we also perform changes of bases on $V^4, V^5, V^6$ to keep these matrices in the forms shown in the subspace tracking (S1 to S4). 

The inductive assumption is that the matrix $C$ can be written in the form
\[
C=\left[\begin{array}{cc}
        C' & \vector{*}\\
        \vector{0} & \vector{*}
    \end{array}\right],
\]
where $C'$ is given by
\[
C'=\left[\begin{array}{ccc|ccccc}
        1&&            &&         &&            &\\
        &\ddots&       &&         &\vector{0}&  &\\
        &&1            &&&&&\\
        \hline
        &&             &1&1       &&            &\\
        &\vector{0}&   &&         &\ddots&      &\\
        &&             &&         &&1           &1
    \end{array}\right]
\]
with the $2\ell_1\times 2\ell_1$ identity matrix in the top left and the
$\ell_2\times 2\ell_2$ matrix in the bottom right. Let
$\ell=\ell_1+\ell_2$ and $m=2\ell_1+\ell_2$. The size of the matrix
$C'$ is $m \times 2\ell$. We show that the number $2\ell$ of the columns in $C'$ can
be extended to $2\ell+2$ without affecting the matrix forms of $f_{14}, f_{45}, f_{65}$, $f_{36}$ obtained in the subspace tracking. This gives the desired
form~\eqref{eq:hatC} of the matrix $\hat C$ by induction.

The column in $C$ next to $C'$ is numbered $2\ell+1$, corresponding to the $(\ell+1)$st column $a_{\ell+1}$ of $f_{12}$. We claim that there exists a nonzero element $a_{\ell+1,i}$ at some $i>m$ in this column.
If not, $a_{\ell+1}$ can be expressed as a linear combination
\[
a_{\ell+1} =
\sum_{i=1}^{\ell_1}\left(\alpha_i a_i + \beta_i b_i\right) +
\sum_{i= \ell_1+1}^{\ell}\alpha_i a_i
\]
due to the current form of $C$.
Then, by mapping both sides to $V^5$ by $f_{25}$, we have 
\[
f_{25}(a_{\ell+1}) =
\sum_{i=1}^{\ell_1}\left(\alpha_i+\beta_i\right)f_{25}(a_i) +
\sum_{i=\ell_1+1}^{\ell}\alpha_if_{25}(a_i).
\]
This contradicts the fact that $f_{25}(a_1),\dots,f_{25}(a_{\ell+1})$
are linearly independent.
Hence, the matrix $C$ can be expressed by 
%\[
%\begin{array}{c}
%\begin{array}{cccc}
%&&&\tiny{\mbox{$(\ell+1)$st}}
%\end{array}\\
%C=\left[\begin{array}{c|c|c}
%        C' & \vector{0} & \vector{*}\\
%        \hline
%        \vector{0} &1& \vector{*}\\
%        \hline
%        \vector{0} & \vector{0} & \vector{*}
%    \end{array}\right]
%\end{array}
%\]
\[
\begin{array}{ll}
%\begin{array}{c}
%\hspace{1.0cm}\tiny{\mbox{$(\ell+1)$st}}
%\end{array} 
\hspace{1.3cm}\tiny{\mbox{$(2\ell+1)$st}}& \\
C=\left[\begin{array}{c|c|c}
        C' & \vector{~0~} & \vector{~*}\\
        \hline
        \vector{0} &1& \vector{~*}\\
        \hline
        \vector{0} & \vector{0} & \vector{~*}
    \end{array}\right] & \hspace{-0.2cm}
\begin{array}{c}
\\
\tiny{\mbox{$(m+1)$st}}\\
\\
\end{array}
\end{array}
\]
after suitable row operations.

Next, we move to the $(2\ell+2)$nd column of $C$, which is the column $b_{\ell+1}$ of $f_{32}$.

\vspace{0.2cm}\noindent\textbf{Case 1:} 
If the vector $b_{\ell+1}$ has a nonzero element $b_{\ell+1,i}$
with $i>m+1$, then $C$ can be transformed by suitable row operations into 
\[
C=\left[\begin{array}{c|cc|c}
        C' & \vector{~0~} & \vector{~0~} & \vector{~*}\\
        \hline
        \vector{0} &1&0 & \vector{~*}\\
        \vector{0} &0&1 & \vector{~*}\\
        \hline
        \vector{0} & \vector{0} & \vector{0}& \vector{~*} 
    \end{array}\right]
\hspace{-0.1cm}    
	\begin{array}{l}
    \\
    \tiny\mbox{$(m+1)$st}\\
    \tiny\mbox{$(m+2)$nd}\\
    \\
    \end{array}~~.
\]
This updates $\ell_1$ in $C'$ to $\ell_1+1$, and appropriate permutations finish the induction step.

\vspace{0.2cm}\noindent\textbf{Case 2:} 
Otherwise, suppose $b_{\ell+1,i}=0$ for all $i>m+1$. 
Then, the column vector $b_{\ell+1}$ can be expressed as a linear combination
\[
b_{\ell+1} = 
\sum_{i=1}^{\ell_1}(\alpha_ia_i+\beta_ib_i) +
\sum_{i=\ell_1+1}^{\ell}\alpha_ia_i +\alpha_{\ell+1} a_{\ell+1}.
\]
Mapping both sides to $V^5$ by $f_{25}$ leads to
\[
(1-\alpha_{\ell+1})f_{25}(a_{\ell+1}) =
\sum_{i=1}^{\ell_1}(\alpha_i+\beta_i)f_{25}(a_i) +
\sum_{i=\ell_1+1}^{\ell}\alpha_if_{25}(a_i).
\]
By linear independence, we have 
\[
\alpha_{\ell+1}=1,~~~~~\alpha_i=\left\{\begin{array}{ll}
        -\beta_i, & i=1,\dots,\ell_1,\\
        0,&i=\ell_1+1,\dots,\ell.
    \end{array}\right.
\]
Therefore, the matrix $C$ has the form
\[
C=\left[\begin{array}{c|c|c|c}
        & & \alpha_1& \\
        & & -\alpha_1& \\
        C' & \vector{~0~}& \vdots& \vector{~*}\\
        & & \alpha_{\ell_1}& \\
        & & -\alpha_{\ell_1}& \\
        & & \vector{0}& \\
        \hline
        \vector{0} &1 & 1& \vector{~*}\\
        \hline
        \vector{0} & \vector{0}&\vector{0} & \vector{~*}
    \end{array}\right]
\]
and, equivalently, the maps $f_{12}$ and $f_{32}$ are now in the forms
\[
f_{12}=
\begin{array}{c}
\hspace{0.9cm}\tiny{\mbox{$(\ell+1)$st}}
\\
    \left[\begin{array}{c|c|c|c}
            D_a & \vector{~0~} & \vector{~0~} &\vector{~*}\\
            \hline
            \vector{0}&\vector{1}&\vector{0}&\vector{~*}\\
            \hline
            \vector{0}&\vector{0}&1&\vector{~*}\\
            \hline
            \vector{0}&\vector{0}&\vector{0}&\vector{~*}
        \end{array}\right]
\end{array}
,~~~
f_{32}=
\begin{array}{c}
\hspace{0.6cm}\tiny{\mbox{$(\ell+1)$st}}
\\
    \left[\begin{array}{c|c|c|c}
            D_b & \vector{~0~} & \vector{\alpha_\pm} &\vector{~*}\\
            \hline
            \vector{0}&\vector{1}&\vector{0}&\vector{~*}\\
            \hline
            \vector{0}&\vector{0}&1&\vector{~*}\\
            \hline
            \vector{0}&\vector{0}&\vector{0}&\vector{~*}
        \end{array}\right],
\end{array}
\]
where $D_a$ and $D_b$ are the $2\ell_1\times \ell_1$ matrices
\[
D_a=\left[\begin{array}{cccc}
        1&&&\\
        0&&&\\
        &1&&\\
        &0&&\\
        &&\ddots&\\
        &&&1\\
        &&&0
    \end{array}\right],~~~
D_b=\left[\begin{array}{cccc}
        0&&&\\
        1&&&\\
        &0&&\\
        &1&&\\
        &&\ddots&\\
        &&&0\\
        &&&1
    \end{array}\right].
\]
Here, we define $\vector{\alpha}_\pm$ and $\vector{\alpha}_{+}, \vector{\alpha}$ (for later use) as
\[
\vector{\alpha}_\pm=\left[ \begin{array}{c}
        \alpha_1\\
        -\alpha_1\\
        \vdots\\
        \alpha_{\ell_1}\\
        -\alpha_{\ell_1}
    \end{array}\right],~
\vector{\alpha}_{+}=\left[\begin{array}{c}
        \alpha_1\\
        0\\
        \vdots\\
        \alpha_{\ell_1}\\
        0
    \end{array}\right],~
\vector{\alpha}=\left[\begin{array}{c}
        \alpha_1\\
        \vdots\\
        \alpha_{\ell_1}\\
    \end{array}\right].
\]

Note that the nonzero entries to the left of
$\vector{\alpha}_\pm$ in $f_{32}$ occur in the even-numbered rows.
By a change of basis on $V^3$ (column operations on $f_{32}$), we are
thus able to zero out the even-numbered entries in
$\vector{\alpha}_\pm$ as follows:
\[
f_{32}=
\begin{array}{c}
\hspace{0.6cm}\tiny{\mbox{$(\ell+1)$st}}
\\
    \left[\begin{array}{c|c|c|c}
            D_b & \vector{~0~} & \vector{\alpha}_{+} &\vector{~*}\\
            \hline
            \vector{0}&\vector{1}&\vector{0}&\vector{~*}\\
            \hline
            \vector{0}&\vector{0}&1&\vector{~*}\\
            \hline
            \vector{0}&\vector{0}&\vector{0}&\vector{~*}
        \end{array}\right].
\end{array}
\]
This can be done by using the basis change matrix on $V^3$
\begin{equation}
    \label{eq:bchng36}
    \begin{array}{c}
\hspace{1.3cm}\tiny{\mbox{$(\ell+1)$st}}
\\
        R=\left[\begin{array}{c|c|c|c}
                \vector{~1~} & &\vector{\alpha} & \\
                \hline
                &\vector{~1~} &\vector{~0~} & \\
                \hline                
                & &1 & \\
                \hline                
                & & &\vector{~1~} 
            \end{array}\right],
    \end{array}
\end{equation}
and it changes $f_{36}$ into 
\[
f_{36}=\left[\begin{array}{c}
        \vector{1}\\
\hline        
        \vector{0}
    \end{array}\right]
R
=\left[\begin{array}{c}
       R\\
\hline        
        \vector{0}
    \end{array}\right].
\]
%\[
%f_{36}=\left[\begin{array}{c}
%        \vector{1}\\
%\hline        
%        \vector{0}
%    \end{array}\right]
%\left[\begin{array}{cccc}
%        \vector{1} & &\vector{\alpha} & \\
%        &\vector{1} &\vector{0} & \\
%        & &1 & \\
%        & & &\vector{1} 
%    \end{array}\right]
%=
%\left[\begin{array}{cccc}
%        \vector{1} & &\vector{\alpha} & \\
%        &\vector{1} &\vector{0} & \\
%        & &1 & \\
%        & & &\vector{1} \\
%        \hline
%        &&\vector{0}&
%    \end{array}\right].
%\]
%Originally, $f_{36}$ is already in a form that we want. 
We will see later that  $f_{36}$ can be
restored to the form ${\scriptsize \left[\begin{array}{c} \vector{1} \\\hline \vector{0}
    \end{array} \right]}$ within the inductive step.

We next zero out $\vector{\alpha}_+$ in $f_{32}$ with row operations to obtain
\[
f_{32}=
\begin{array}{c}
\hspace{0.9cm}\tiny{\mbox{$(\ell+1)$st}}
\\
    \left[\begin{array}{c|c|c|c}
            D_b & \vector{~0~} & \vector{~0~} &\vector{~*}\\
            \hline
            \vector{0}&\vector{1}&\vector{0}&\vector{~*}\\
            \hline
            \vector{0}&\vector{0}&1&\vector{~*}\\
            \hline
            \vector{0}&\vector{0}&\vector{0}&\vector{~*}
        \end{array}\right]
\end{array}
\]
by using the basis change matrix
\[
\begin{array}{c}
\hspace{0.1cm}\tiny{\mbox{$(m+1)$st}}
\\
    \left[\begin{array}{c|c|c|c}
            \vector{~1~} & &\vector{-\alpha}_{+} & \\
            \hline
            &\vector{~1~} &\vector{0} & \\
            \hline
            & &1 & \\
            \hline
            & & &\vector{~1~} 
        \end{array}\right]^{-1}
\end{array}
\]
on $V^2$.
This changes $f_{12}$ into
\[
f_{12}=
\left[\begin{array}{c|c|c|c}
        \vector{~1~} & &\vector{-\alpha}_{+} & \\
        \hline
        &\vector{~1~} &\vector{0} & \\
        \hline        
        & &1 & \\
        \hline        
        & & &\vector{~1~} 
    \end{array}\right]
\left[\begin{array}{c|c|c|c}
        D_a & \vector{~0~} & \vector{~0~} &\vector{~*}\\
        \hline
        \vector{0}&\vector{1}&\vector{0}&\vector{~*}\\
        \hline
        \vector{0}&\vector{0}&1&\vector{~*}\\
        \hline
        \vector{0}&\vector{0}&\vector{0}&\vector{~*}
    \end{array}\right]=
\left[\begin{array}{c|c|c|c}
        D_a & \vector{~0~} & \vector{-\alpha}_{+} &\vector{~*}\\
        \hline
        \vector{0}&\vector{1}&\vector{0}&\vector{~*}\\
        \hline
        \vector{0}&\vector{0}&1&\vector{~*}\\
        \hline
        \vector{0}&\vector{0}&\vector{0}&\vector{~*}
    \end{array}\right].
\]
Essentially, we transferred $\vector{\alpha}_+$ from $f_{32}$ to
$f_{12}$ as $-\vector{\alpha}_+$. Then, we zero out $-\vector{\alpha}_+$ in $f_{12}$ by column
operations (change of basis on $V^1$) to obtain
\[
f_{12}=
\begin{array}{c}
\hspace{0.8cm}\tiny{\mbox{$(\ell+1)$st}}
\\
    \left[\begin{array}{c|c|c|c}
            D_a & \vector{~0~} & \vector{~0~} &\vector{~*}\\
            \hline
            \vector{0}&\vector{1}&\vector{0}&\vector{~*}\\
            \hline
            \vector{0}&\vector{0}&1&\vector{~*}\\
            \hline
            \vector{0}&\vector{0}&\vector{0}&\vector{~*}
        \end{array}\right].
\end{array}
\]
This is achieved by the same basis change matrix (\ref{eq:bchng36}), and it changes $f_{14}$ into
\[
f_{14}=\left[\begin{array}{c}
        \vector{1}\\
        \hline
        \vector{0}
    \end{array}\right]
R
=\left[\begin{array}{c}
        R\\
        \hline
        \vector{0}
    \end{array}\right].
\]
At this point, $C'$ is updated to the desired form with $\ell_2$ increasing to $\ell_2+1$. 

It remains to transform the maps $f_{14}, f_{45}, f_{65}, f_{36}$ into the forms shown in the subspace tracking (S1 to S4). From the matrix forms 
$
f_{14} = {\scriptsize \left[\begin{array}{c}
        R\\
        \vector{0}
    \end{array}\right]}
$
and
$
f_{36} = {\scriptsize \left[\begin{array}{c}
        R\\
        \vector{0}
    \end{array}\right]}
$ (size of the matrices $\vector{0}$ may be different), we can perform changes of bases on $V^4$ and $V^6$ to obtain
\[
f_{14}=\left[\begin{array}{c}
        \vector{1}\\
        \vector{0}
    \end{array}\right],~~~
f_{36}=\left[\begin{array}{c}
        \vector{1}\\
        \vector{0}
    \end{array}\right]
\]
by the basis change matrices of the form
$
{\scriptsize \left[\begin{array}{cc}
        \vector{R} & \vector{0}\\
        \vector{0} & \vector{1}
    \end{array}\right]}
$. This changes the matrix forms of $f_{45}$ and $f_{65}$ to
\[
f_{45}=\left[\begin{array}{cc}
        \vector{1} & \vector{0}\\
        \vector{0} & \vector{*}
    \end{array}\right]
\left[\begin{array}{cc}
       {R} & \vector{0}\\
        \vector{0} & \vector{1}
    \end{array}\right]
=    \left[\begin{array}{cc}
        {R} & \vector{0}\\
        \vector{0} & \vector{*}
    \end{array}\right]
    ,~~~
f_{65}=\left[\begin{array}{cc}
        \vector{1} & \vector{0}\\
        \vector{0} & \vector{*'}
    \end{array}\right]
\left[\begin{array}{cc}
       {R} & \vector{0}\\
        \vector{0} & \vector{1}
    \end{array}\right]
=\left[\begin{array}{cc}
        {R} & \vector{0}\\
        \vector{0} & \vector{*'}
    \end{array}\right]    .
\]
Finally, a change of basis on $V^5$ by 
${\scriptsize \left[\begin{array}{cc}
        {R} & \vector{0}\\
        \vector{0} & \vector{1}
    \end{array}\right]}$ 
simultaneously transforms the maps $f_{45}$ and $f_{65}$ into the desired forms
\[
f_{45}=\left[\begin{array}{cc}
        \vector{1}&\vector{0}\\
        \vector{0}&\vector{*}
    \end{array}\right],~~~
f_{65}=\left[\begin{array}{cc}
        \vector{1}&\vector{0}\\
        \vector{0}&\vector{*'}
    \end{array}\right].
\]
This finishes the induction step. 
%At the end of every induction step, by performing these operations, we
%are able to preserve the forms of $f_{14},f_{36},f_{45},f_{65}$. 
%

Thus, we obtain the following decompositions:
\[
\begin{array}{rcl}
    f_{12} =\left[\begin{array}{cc}
            D_a & \vector{0}\\
            \vector{0} & \vector{1}\\
            \vector{0} & \vector{0}
        \end{array}\right] &:&
    V^1 = V^1_{1223456}\oplus V^1_{123456}\rightarrow
    V^2 = V^2_{1223456}\oplus V^2_{123456}\oplus W^2,\\
    f_{32} =\left[\begin{array}{cc}
            D_b & \vector{0}\\
            \vector{0} & \vector{1}\\
            \vector{0} & \vector{0}
        \end{array}\right] &:&
    V^3 = V^3_{1223456}\oplus V^3_{123456}\rightarrow
    V^2 = V^2_{1223456}\oplus V^2_{123456}\oplus W^2,\\
    f_{14} =\left[\begin{array}{cc}
            \vector{1} & \vector{0}\\
            \vector{0} & \vector{1}\\
            \vector{0} & \vector{0}
        \end{array}\right] &:& 
    V^1 = V^1_{1223456}\oplus V^1_{123456}\rightarrow
    V^4 = V^4_{1223456}\oplus V^4_{123456}\oplus X^4,\\
    f_{36} =\left[\begin{array}{cc}
            \vector{1} & \vector{0}\\
            \vector{0} & \vector{1}\\
            \vector{0} & \vector{0}
        \end{array}\right] &:& 
    V^3 = V^3_{1223456}\oplus V^3_{123456}\rightarrow
    V^6 = V^6_{1223456}\oplus V^6_{123456}\oplus X^6,\\
    f_{45} =\left[\begin{array}{ccc}
            \vector{1} & \vector{0} & \vector{0}\\
            \vector{0} & \vector{1} & \vector{0}\\
            \vector{0} & \vector{0} & \vector{*}
        \end{array}\right]&:& 
    V^4 = V^4_{1223456}\oplus V^4_{123456}\oplus
    X^4\rightarrow
    V^5 = V^5_{1223456}\oplus V^5_{123456}\oplus X^5,\\
    f_{65} = \left[\begin{array}{ccc}
            \vector{1} & \vector{0} & \vector{0}\\
            \vector{0} & \vector{1} & \vector{0}\\
            \vector{0} & \vector{0} & \vector{*}
        \end{array}\right] &:& 
    V^6 = V^6_{1223456}\oplus V^6_{123456}\oplus
    X^5\rightarrow
    V^6 = V^5_{1223456}\oplus V^5_{123456}\oplus X^5,
\end{array}
\]
where the size of $D_a$ and $D_b$ is $2s_1\times s_1$.

By commutativity and the series of executed changes of bases, $f_{25}$ is now in the form
\begin{eqnarray*}
    &&f_{25}: V^2=V^2_{1223456}\oplus V^2_{123456}\oplus W^2\rightarrow V^5=V^5_{1223456}\oplus V^5_{123456}\oplus X^5\\
    &&f_{25}=\left[\begin{array}{ccccc|ccc|ccc}
            1&1&&&&&&&\lambda_{11}&\cdots&\lambda_{1t}\\
            &&\ddots&&&&\vector{0}&&\vdots&&\vdots\\
            &&&1&1&&&&\lambda_{s_11}&\cdots&\lambda_{s_1t}\\
            \hline
            &&&&&1&&&&&\\
            &&\vector{0}&&&&\ddots&&&\Xi&\\
            &&&&&&&1&&&\\
            \hline
            &&\vector{0}&&&&\vector{0}&&&\vector{*}&
        \end{array}\right].
\end{eqnarray*}
We replace the subspace $W^2$ with $X^2$ in such a way that $f_{25}$ becomes
\begin{eqnarray*}
    &&f_{25}: V^2=V^2_{1223456}\oplus V^2_{123456}\oplus X^2\rightarrow V^5=V^5_{1223456}\oplus V^5_{123456}\oplus X^5\\
    &&f_{25}=\left[\begin{array}{ccccc|ccc|c}
            1&1&&&&&&&\\
            &&\ddots&&&&\vector{0}&&\vector{0}\\
            &&&1&1&&&&\\
            \hline
            &&&&&1&&&\\
            &&\vector{0}&&&&\ddots&&\vector{0}\\
            &&&&&&&1&\\
            \hline
            &&\vector{0}&&&&\vector{0}&&\vector{*}\\
        \end{array}\right]
\end{eqnarray*}
by the basis change matrix on $V^2$:
\[
\left[\begin{array}{ccccc|ccc|ccc}
        1&&&&&&&&-\lambda_{11}&\cdots&-\lambda_{1t}\\
        &1&&&&&&&0&\cdots&0\\
        &&\ddots&&&&&&\vdots&&\vdots\\
        &&&1&&&&&-\lambda_{s_11}&\cdots&-\lambda_{s_1t}\\
        &&&&1&&&&0&\cdots&0\\
        \hline
        &&&&&1&&&&&\\
        &&&&&&\ddots&&&-\Xi&\\
        &&&&&&&1&&&\\
        \hline
        &&&&&&&&1&&\\
        &&&&&&&&&\ddots&\\
        &&&&&&&&&&1
    \end{array}\right].
\]

Note that this basis change on $V^2$ does not affect $f_{12}$ and $f_{32}$ at all. 
Thus, we obtain the summands
\[
\dimInd{1,1,1,1,1,1}^{k_{123456}} =
\spaceInd{V^1_{123456},V^2_{123456},V^3_{123456},
    V^4_{123456},V^5_{123456},V^6_{123456}}
\]
and 
\[
\dimInd{1,2,1,1,1,1}^{k_{1223456}} =
\spaceInd{V^1_{1223456},V^2_{1223456},V^3_{1223456},
    V^4_{1223456},V^5_{1223456},V^6_{1223456}},
\]
where $k_{123456} = \dim V^i_{123456}$ for $i=1,\dots,6$ and $k_{1223456} = \dim V^i_{1223456}$ for $i=1,3,\dots,6$.

%\[ \begin{array}{rcl}
%    k_{123456} &=& \dim V^1_{123456} = \dim V^2_{123456} = \dim
%    V^3_{123456} \\
%    &=& \dim V^4_{123456} = \dim V^5_{123456} = \dim
%    V^6_{123456}, \\
%    k_{1223456} &=& \dim V^1_{1223456} =\frac{1}{2} \dim V^2_{1223456} = \dim
%    V^3_{1223456} \\
%    &=& \dim V^4_{1223456} = \dim V^5_{1223456} = \dim V^6_{1223456}.
%\end{array}
%\]

\begin{flushleft}{\bf Step 5:}
\end{flushleft}
At this stage, the input $V$ satisfies $V^1=V^3=0$. It means that $V$ is a representation on the quiver
\[
    \xymatrix{
        \circ \ar[r]&\circ&\ar[l]\circ\\
        &\circ \ar[u] &
    },
\]    
which is listed in Gabriel's theorem as $\mathbb{D}_4$.
We divide this step into 2 parts. 

\begin{flushleft}{\bf Step 5.1:}  $\dimInd{,1,,,,}$, $\dimInd{,1,,1,1,1}$, $\dimInd{,1,,,1,1}$, $\dimInd{,1,,1,1}$, $\dimInd{,,,1,1,1}$
\end{flushleft}
We easily obtain $\dimInd{,1,,,,}$ from $\kernel f_{25}$, since $V^1=V^3=0$. The other indecomposables can be obtained in a similar way as the previous steps.

\begin{flushleft}{\bf Step  5.2:}
    $\dimInd{,1,,1,2,1}$, 
    $\dimInd{,,,1,1}$, 
    $\dimInd{,1,,,1}$, 
    $\dimInd{,,,,1,1}$, 
    $\dimInd{,,,,1}$
\end{flushleft}

Somewhat different from the previous steps is the appearance of $\dimInd{,1,,1,2,1}$.
Let us define the matrix
\[
C=\left[\begin{array}{cc} f_{45} &  f_{65} \end{array} \right].
\]
We claim $\rank C=\rank f_{45} + \rank f_{65}$.
$\rank C\leq \rank f_{45} + \rank f_{65}$ is obvious. 
Suppose $\rank C<\rank f_{45} + \rank f_{65}$. Then there exist
vectors $v_4\in V^4$ and $v_6\in V^6$ such that
$f_{45}(v_4)=f_{65}(v_6)$. This allows us to extract a summand isomorphic to 
$\dimInd{,,,1,1,1}$ or  $\dimInd{,1,,1,1,1}$, contradicting the fact that these types of summands have already been removed.
It follows from the claim that we can consider the matrix $C$ as an injective map 
$C: V^4 \oplus V^6\rightarrow V^5$.

\begin{enumerate}[S1.]
\item From the above claim, perform a change of basis on $V^5$ to obtain
    \[
    C={\scriptsize \left[\begin{array}{cc}
            \vector{1} & \vector{0}\\
            \vector{0} & \vector{1}\\
            \vector{0} & \vector{0}\\
        \end{array}\right]}: V^4\oplus V^6\rightarrow V^5 =
    U^5_{45}\oplus U^5_{65} \oplus U^5.
    \]
This tracks the subspaces $V^4$ and $V^6$ into $V^5$.
%    This leads to the decomposition of $V^5$ as $V^5 =
%    U^5_{45}\oplus U^5_{65} \oplus U^5$. In terms of subspaces,
%    $U^5_{45} = f_{45}V^4$, $U^5_{65} = f_{65}V^6$.
%    % 
\item Perform a change of basis on $V^2$ to obtain a column echelon form
    \[
    f_{25}={\scriptsize \left[\begin{array}{cc}
            \vector{*} & \vector{*}\\
            \vector{*} & \vector{*}\\
            \vector{0} & \vector{*}
        \end{array}\right]}: 
    V^2 = V^2_{24556}\oplus V^2_{25}\rightarrow V^5 = U^5_{45}\oplus U^5_{65}\oplus U^5.
    \]
%    Note that the direct summands isomorphic to $\dimInd{,1,,1,1}$, $\dimInd{,1,,,1,1}$, and
%    $\dimInd{,1,,1,1,1}$ are already removed. Together with the fact
%    that $f_{25}$ is injective, this means that, for
%    all $x \in V^2$, $f_{25}(x) \neq 0$ and either
%    $f_{25}(x) \notin \image f_{45}=U^5_{45}$ and $f_{25}(x)\notin\image f_{65}=U^5_{65}$, or $f_{25}(x) \in
%    \image f_{45} + \image f_{65}$. The first case corresponds to the
%    subspace $V^2_{25}$, while the second is $V^2_{24556}$.
The remaining direct summands in $V$ that have nonzero vector spaces at vertex 2 consist of  
$\dimInd{,1,,1,2,1}$ and $\dimInd{,1,,,,}$. We denote those vector spaces by $V^2_{24556}$ and $V^2_{25}$, respectively.
  \end{enumerate}
\begin{enumerate}[B1.]
\item  From the above remark, we can perform changes of bases on $U^5_{45}$ and $U^5_{65}$, respectively, to express $f_{25}$ in the form
     \[
    f_{25}={\scriptsize \left[\begin{array}{cc}
            \vector{1} & \vector{*}\\
            \vector{0} & \vector{*}\\
            \hline
            \vector{1} & \vector{*}\\
            \vector{0} & \vector{*}\\
            \hline
            \vector{0} & \vector{*}
        \end{array}\right]}: 
    V^2=V^2_{24556}\oplus V^2_{25}\rightarrow V^5 = (U^5_{245}\oplus
    W^5_{45} )\oplus (U^5_{265}\oplus W^5_{65})\oplus U^5.
    \]
   
\item As before, perform a change of basis on $V^5$ to obtain
    \[
    f_{25}={\scriptsize \left[\begin{array}{cc}
            \vector{1} & \vector{0}\\
            \vector{0} & \vector{0}\\
            \hline
            \vector{1} & \vector{0}\\
            \vector{0} & \vector{0}\\
            \hline
            \vector{0} & \vector{0}\\
            \vector{0} & \vector{1}
        \end{array}\right]}: 
    V^2=V^2_{24556}\oplus V^2_{25}\rightarrow V^5 = (U^5_{245}\oplus
    W^5_{45}) \oplus (U^5_{265}\oplus W^5_{65}) \oplus (V^5_5\oplus V^5_{25})
    \]
without affecting $C$.
    From this expression, we can extract the direct summands
    \[
    \dimInd{0,1,0,0,1,0}^{k_{25}} = \spaceInd{0,V^2_{25},0,0,V^5_{25},0},~~~
    \dimInd{0,0,0,0,1,0}^{k_{5}} = \spaceInd{0,0,0,0,V^5_{5},0}.
    \]    
    Hence, by removing them and switching the order of $W^5_{45}$ and $U^5_{265}$, we have
    \[
    f_{25}={\scriptsize \left[\begin{array}{c}
            \vector{1} \\
            \vector{1} \\
            \vector{0} \\
            \vector{0} 
        \end{array}\right]}: 
    V^2=V^2_{24556}\rightarrow V^5 = U^5_{245}\oplus U^5_{265}\oplus W^5_{45}\oplus W^5_{65}.
    \]

\item Perform changes of bases on $V^4$ and $V^5$ to obtain
    \[
    f_{45}={\scriptsize \left[\begin{array}{cc}
            \vector{1} & \vector{0}\\
            \vector{0} & \vector{0}\\
            \vector{0} & \vector{1}\\
            \vector{0} & \vector{0}
        \end{array}\right]}: 
    V^4=V^4_{24556}\oplus V^4_{45}\rightarrow V^5=U^5_{245}\oplus U^5_{265}\oplus V^5_{45}\oplus W^5_{65}
    \]
        Note that this basis change on $V^5$ does not change $f_{25}$ at all.

\item
    In a similar manner, we can derive the summands $V^6_{24556}$ and
    $V^6_{65}$:
    \[
    f_{65} = {\scriptsize \left[\begin{array}{cc}
            \vector{0} & \vector{0} \\
            \vector{1} & \vector{0}\\ 
            \vector{0} & \vector{0} \\ 
            \vector{0} & \vector{1}
        \end{array} \right]}:
    V^6=V^6_{24556} \oplus V^6_{65} \rightarrow V^5 = U^5_{245}\oplus U^5_{265} \oplus
    V^5_{45}\oplus V^5_{65}.
    \]

    Thus, by setting $V^5_{24556}=U^5_{245}\oplus U^5_{265}$, we obtain the desired summands
    \[
    \dimInd{0,1,0,1,2,1}^{k_{24556}} = \spaceInd{0,V^2_{24556},0,V^4_{24556}, V^5_{24556},
        V^6_{24556}},
    \]
    \[
    \dimInd{0,0,0,1,1,0}^{k_{45}} = \spaceInd{0,0,0,V^4_{45},V^5_{45},0}~~~{\rm and}~~~
    \dimInd{0,0,0,0,1,1}^{k_{65}} = \spaceInd{0,0,0,0,V^5_{65},V^6_{65}}.
    \]
\end{enumerate} 
This completes the algorithm.
All operations involved in the algorithm consist only of row and column operations on the matrices.  
%Let $\dim V^i = n_i$ for all $i$, and $n = \max_i n_i$. Clearly, most of the steps involve $O(n^3)$ time. Step $4.2$ is less straightforward, as it involves the construction of a matrix $C$ by alternating the columns of $f_{12}$ and $f_{32}$. Then, operations on $C$ take $\max(n_2,  (n_1 + n_3)^3) \leq 8 n^3$ in the worst case. A similar analysis holds for step $5.2$
Clearly, the algorithm has $O(n^3)$ worst-case complexity, where $n$ is the maximum of $\dim V^i$ over all $i$.

\section{Numerical Examples}\label{sec:numerics}
We provide two numerical examples to demonstrate the feasibility of our algorithm and to illustrate the interpretation of persistence diagrams, as discussed in Section \ref{sec:arq_clfb}. As before, let us focus on persistence modules on $\CL(fb)$ of the form:
\begin{equation}\label{eq:cl_numerics}
\spaceInd{H_\ell(X_r), H_\ell(X_r \cup Y_r), H_\ell(Y_r),
    H_\ell(X_s), H_\ell(X_s \cup Y_s), H_\ell(Y_s)}.
\end{equation}
In these examples, we construct $X_r, X_s, Y_r$, and $Y_s$ by some geometric models on finite sets $P_X$ and $P_Y$ of points in $\R^3$. The linear maps in (\ref{eq:cl_numerics}) are induced by simplicial maps on $X_\alpha, Y_\alpha$, and $X_\alpha\cup Y_\alpha$.

\vspace{0.3cm}\noindent
\textbf{Two Cubes.} Let $P_X=\{x_i\in \R^3\mid i=1,\dots,8\}$ be the set of the eight vertices of a cube of side length four, and let $P_Y=\{y_i=x_i+\epsilon_i\in\R^3\mid i=1,\dots,8\}$ be the set obtained by adding random noise $\epsilon_i$ with $||\epsilon_i||_\infty\leq 0.1$ to the vertices in $P_X$. Let $X_\alpha$ and $Y_\alpha$ be Vietoris-Rips  complexes \cite{eh} constructed from $P_X$ and $P_Y$, respectively, with parameter $\alpha$. Then, the nontrivial homology group $H_1(X_\alpha)$ is generated by 5 cycles which persist in the interval $\alpha\in[2,2\sqrt{2})$. 

Let us choose two parameters $r<s$ inside this interval and compute the persistence diagram of (\ref{eq:cl_numerics}) with $\ell=1$, where the simplicial maps are determined by the identification of the vertices $x_i$ and $y_i$, for $i=1,\dots,8$.
In the left and right part of Figure~\ref{fig:two_cubes}, the parameters are set to be $r=2.1,~s=2.5$ and $r=2.052,~s=2.5$, respectively. The left diagram shows that we detect the indecomposable $\dimVec{1,1,1,1,1,1}$ with multiplicity five. It means that five nontrivial cycles in $H_1(X_\alpha)$ are shared with $H_1(Y_\alpha)$ in the interval $\alpha\in [2.1,2.5]$. 
On the other hand, the right diagram shows that three of these five generators split into the indecomposable $\dimVec{1,1,0,1,1,1}$. It means that these three shared nontrivial cycles are less robust compared to the remaining two.
%This is a reasonable behavior, because the interval $[r,s]$ is enlarged by decreasing $r$.
\begin{figure}[h!]
        \centering
        \includegraphics[width=0.45\textwidth]{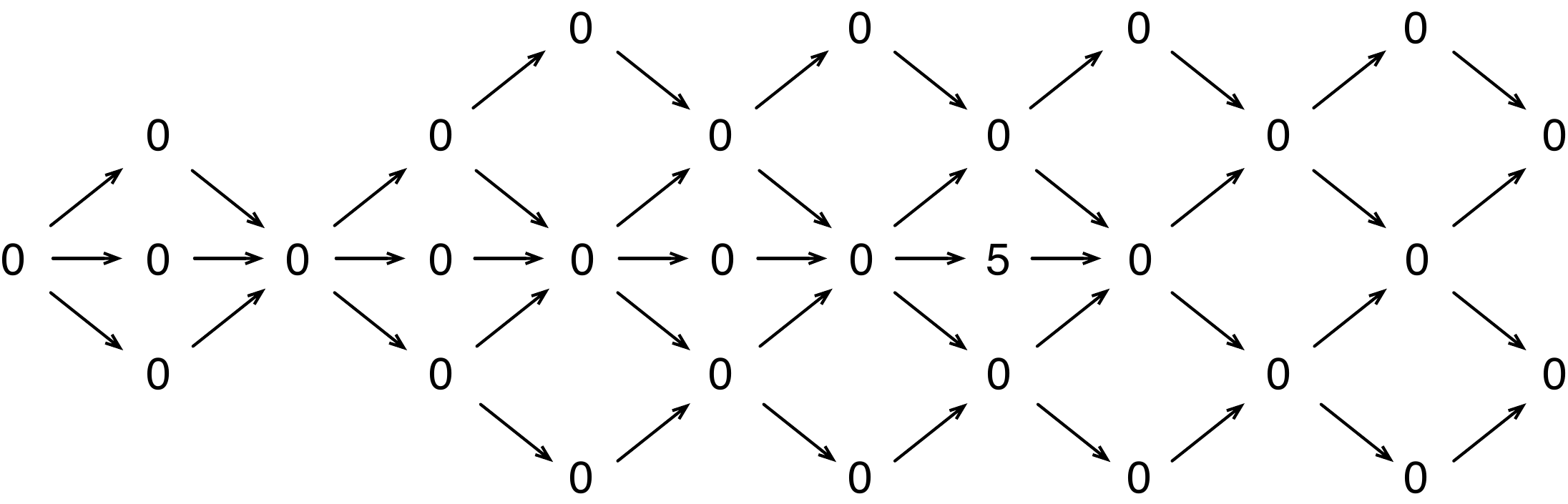}~~~~~
        \includegraphics[width=0.45\textwidth]{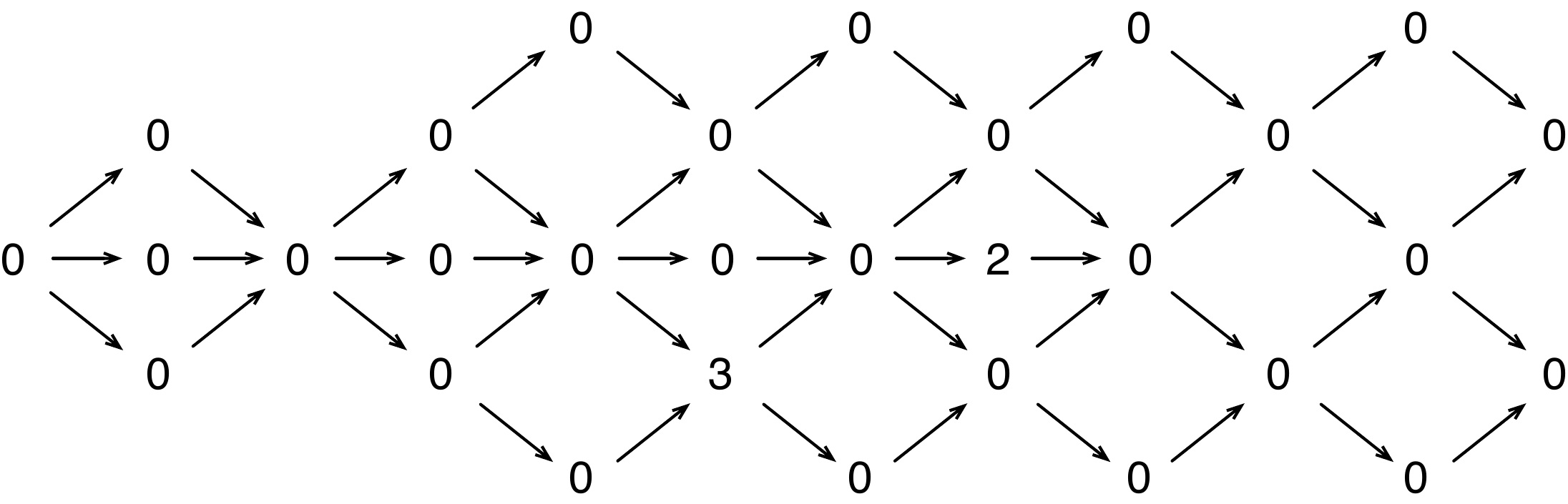}
        \caption{Persistence diagrams of the two cubes with parameters $r=2.1,~s=2.5$ on the left and $r=2.052,~s=2.5$ on the right.}
        \label{fig:two_cubes} 	
\end{figure}

\vspace{0.3cm}\noindent
\textbf{Silica Glasses.} We apply our newly developed TDA tool to the original motivation explained in Section \ref{sec:introduction}. Let $P_X$ be a set of 8100 atoms in an \sio glass obtained from a molecular dynamics simulation. The left part of Figure \ref{fig:pd_glass} shows its persistence diagram. We denote the set of generators in the blue region by $G_X$. Our interest from the viewpoint of materials science is to investigate how many generators in $G_X$ persist under physical pressurization. 

Here, we remark that the vineyard of persistence diagrams \cite{vineyard} is not sufficient for this purpose, since the deformation induced by the pressurization may not be sufficiently small. Furthermore, the discussion in Section \ref{sec:arq_clfb} provides us with a tool to explicitly observe the transition of the generators $G_X$ under pressurization. 
Now, by setting $P_Y$ to be the atomic locations after physical pressurization via simulation, let us compute the persistence diagram of (\ref{eq:cl_numerics}) and study the transition of the generators in $G_X$. 

Figure \ref{fig:press_0_1} shows the persistence diagram of (\ref{eq:cl_numerics}) with $\ell=1$, $r=0.0$ and $s=0.6$, and where $X_\alpha$ and $Y_\alpha$ are the weighted alpha complexes \cite{alpha} computed using CGAL \cite{cgal}. Recall that the generators in $G_X$ are specified by the intervals $\I[r,s]^{m_{rs}}$ in the 2-step persistence 
\[
	H_1(X_r)\rightarrow H_1(X_s)\simeq \I[r,r]^{m_r}\oplus \I[s,s]^{m_s}\oplus \I[r,s]^{m_{rs}}.
\]	
As explained in Section \ref{sec:arq_clfb}, the transition of the generators $\I[r,s]^{m_{rs}}$ can be described by the colored region in the figure, and the direct summand consisting of the indecomposable $\dimInd{1,1,1,1,1,1}$ expresses the generators of our interest. From this computation, we can conclude that $99.18 \%~(\approx 2304/2323)$ of the generators in $G_X$  persist under physical pressurization. 
\begin{figure}[h]
    \begin{center}\includegraphics{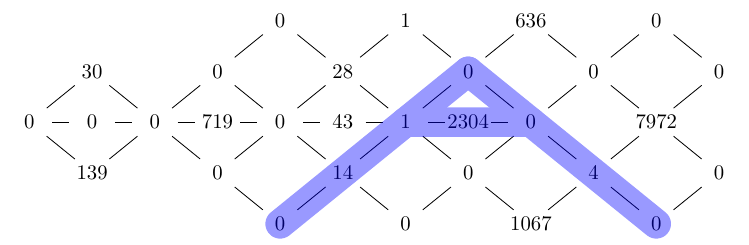}
    \caption{Persistence diagram of the glass example.}
    \label{fig:press_0_1}    
\end{center}
\end{figure}

%%
%\begin{figure}[h]
%    \centering
%               \includegraphics[width=0.6\textwidth]{a_0000_1000}

%\end{figure}
%%

%
%In the paper \cite{glass}, we discuss geometric and topological characterizations of silica glasses, and the computational tool developed in this paper is used to clarify some of the physical properties studied there.
%

%
%
%
%
%
\section{Concluding Remarks}\label{sec:concluding}
In this paper, we have studied persistence modules on commutative ladders. We have proposed a new algebraic framework that treats persistence modules as representations on associative algebras and applies the Auslander-Reiten theory.   This approach allows us to extend the theory of topological persistence to quivers with relations such as the commutative ladders, even if these quivers are not included in the list provided by Gabriel's theorem. We have also generalized the concept of persistence diagrams as functions on the vertex sets of Auslander-Reiten quivers. Our algorithm for computing generalized persistence diagrams is constructed by echelon form reductions following the structure of Auslander-Reiten quivers, and this is applied to a practical problem of TDA on glasses. For further speed up of computations, we plan to apply the technique of morse reductions \cite{nanda} to this algorithm \cite{eh_morse}.
%We believe that these new theoretical concepts and tools shed new light on the theory of topological persistence.

%
%
%
%
\section*{Acknowledgments}
The authors would like to thank Hideto Asashiba, Hiroyuki Ochiai, and Dai Tamaki for valuable discussions and comments. This work is partially supported by JSPS Grant-in-Aid (24684007, 26610042).
\appendix
\section{The Auslander-Reiten Quivers of $\CL(\tau_n)$ with $n\leq 4$}\label{sec:appendix}
\begin{figure}[h!]
\begin{center}
\includegraphics[width=7.5cm]{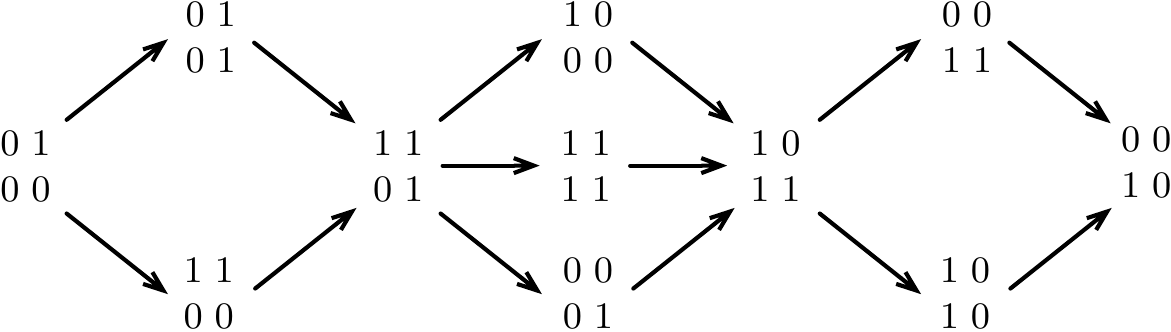}
\caption{The Auslander-Reiten quiver of $\CL(f)$.}
\label{fig:AR_f}
\end{center}
\end{figure}
\begin{figure}[h!]
\begin{center}
\includegraphics[width=15cm]{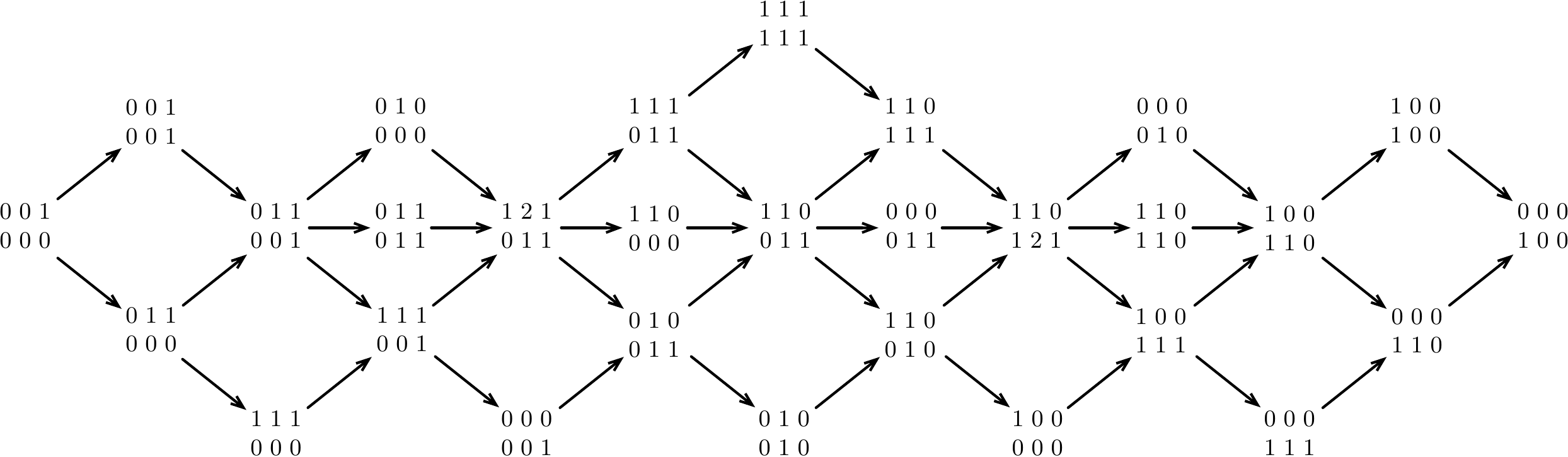}
\caption{The Auslander-Reiten quiver of $\CL(ff)$.}
\label{fig:AR_ff}
\end{center}
\end{figure}
\begin{figure}[h!]
\begin{center}
\includegraphics[width=14cm]{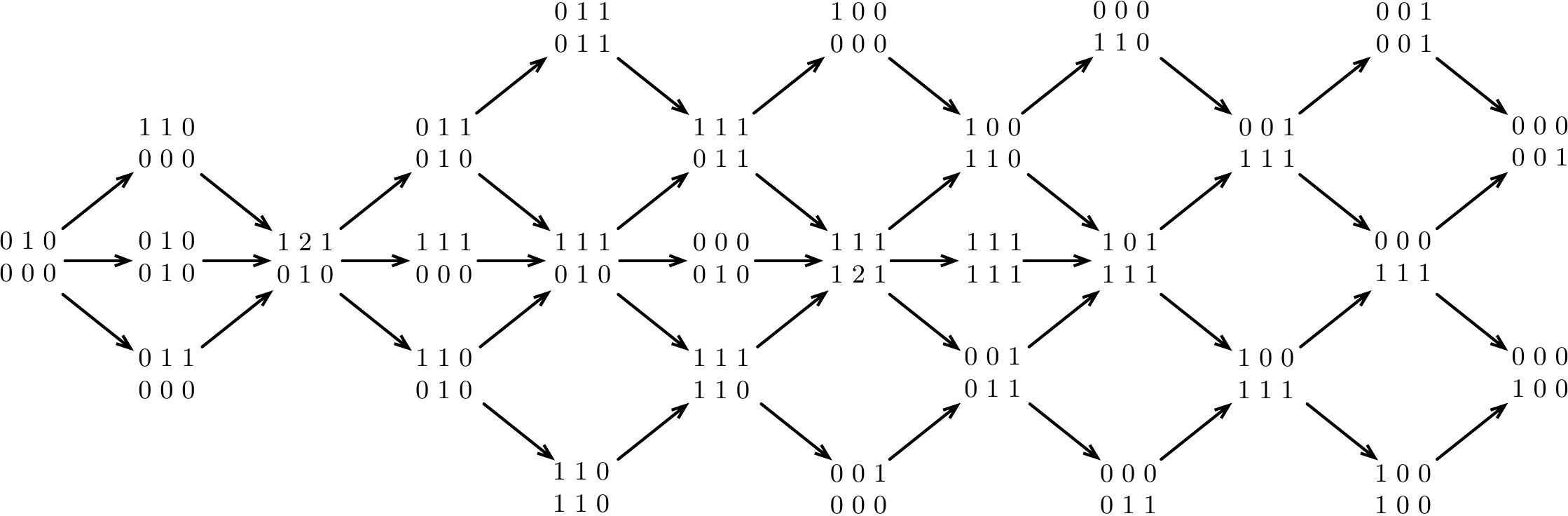}
\caption{The Auslander-Reiten quiver of $\CL(fb)$.}
\label{fig:AR_fb}
\end{center}
\end{figure}
\begin{figure}[h!]
\begin{center}
\includegraphics[width=14cm]{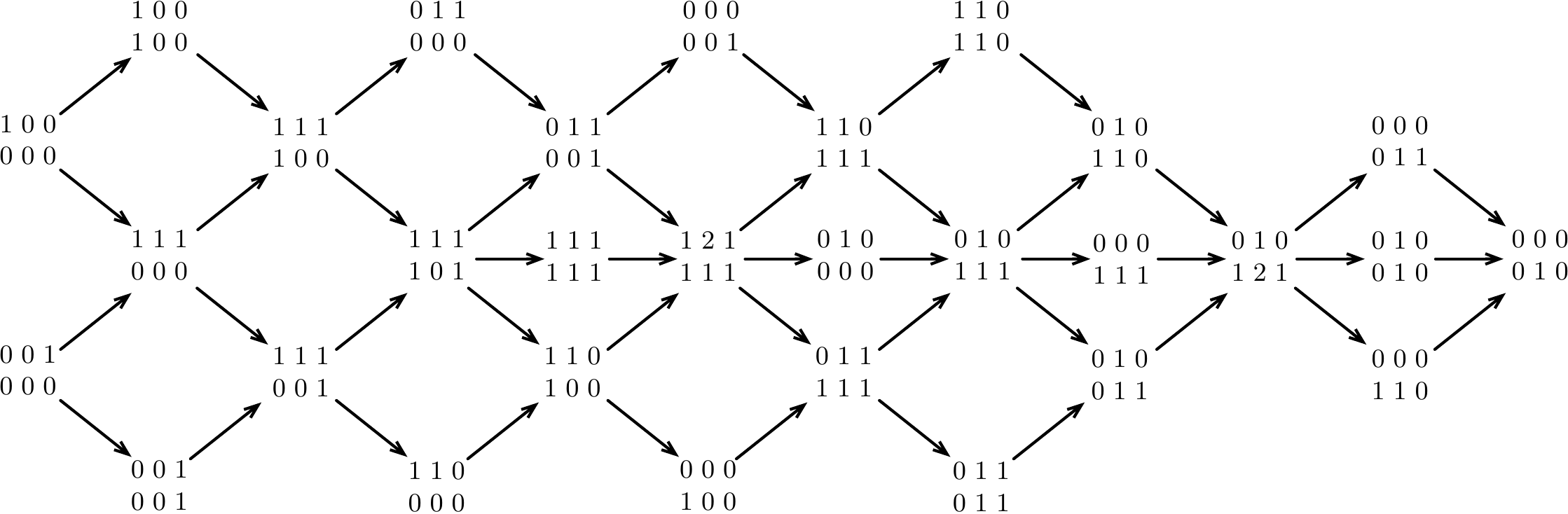}
\caption{The Auslander-Reiten quiver of $\CL(bf)$.}
\label{fig:AR_bf}
\end{center}
\end{figure}
\begin{figure}[h!]
\hspace{-1.5cm}
\includegraphics[width=19cm]{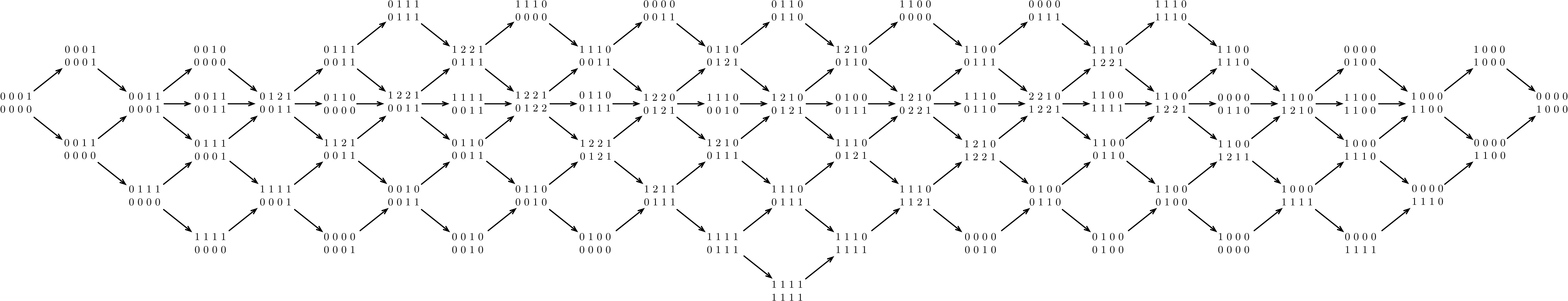}
\caption{The Auslander-Reiten quiver of $\CL(fff)$.}
\label{fig:AR_fff}
\end{figure}
\begin{figure}[h!]
\hspace{-1.5cm}
\includegraphics[width=19cm]{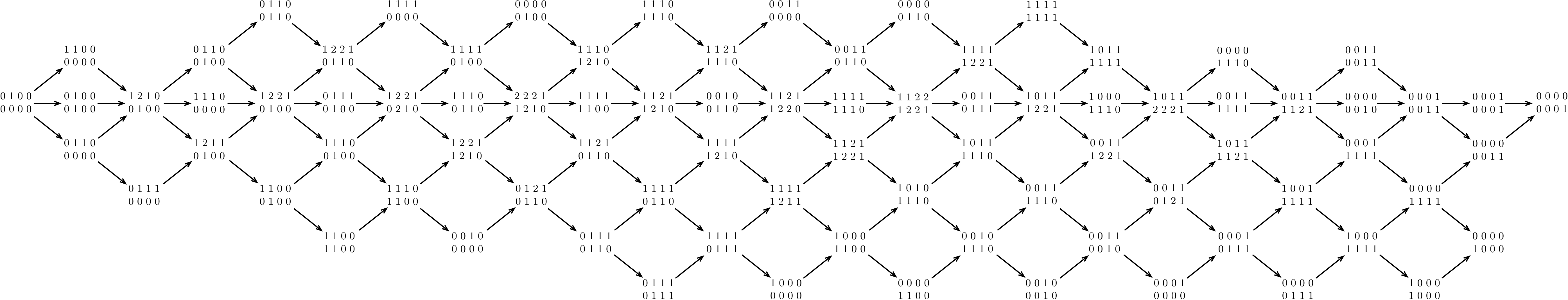}
\caption{The Auslander-Reiten quiver of $\CL(fbb)$.}
\label{fig:AR_fbb}
\end{figure}
\begin{figure}[h!]
\hspace{-1.5cm}
\includegraphics[width=19cm]{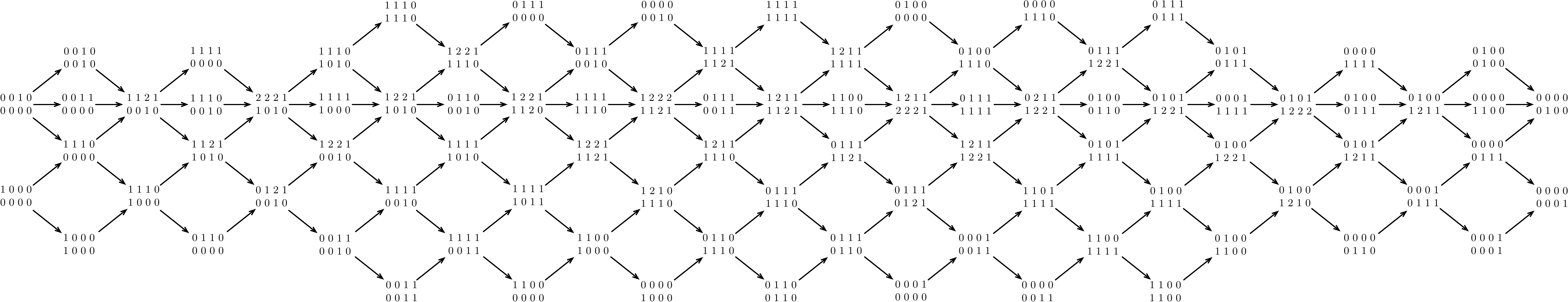}
\caption{The Auslander-Reiten quiver of $\CL(bfb)$.}
\label{fig:AR_bfb}
\end{figure}
\begin{figure}[h!]
\hspace{-1.5cm}
\includegraphics[width=19cm]{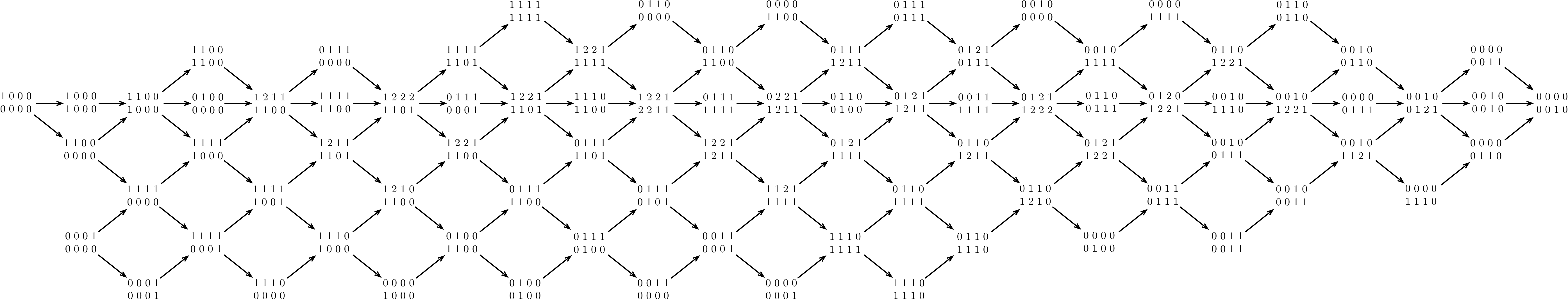}
\caption{The Auslander-Reiten quiver of $\CL(bbf)$.}
\label{fig:AR_bbf}
\end{figure}
%
%%
%\begin{figure}[h!]
%\hspace{-1cm}
%\includegraphics[width=19cm]{./AR_bfb}
%\caption{The Auslander-Reiten quiver of $\CL(fbf)$.}
%\label{fig:AR_bfb}
%\end{figure}
%%

\end{document}